\documentclass[a4paper,reqno,12pt]{amsart}

\usepackage{graphicx}
\usepackage{mathrsfs}
\usepackage{amsfonts}
\usepackage{amssymb}
\usepackage{amsmath}
\usepackage{amsthm}
\usepackage{amscd}
\usepackage[all,2cell]{xy}
\usepackage{xcolor}
\usepackage[pagebackref,colorlinks]{hyperref}
\usepackage{enumerate}
\usepackage[normalem]{ulem}
\usepackage{mathrsfs}
\usepackage{lineno}
\usepackage{bm}
\usepackage{bbm}
\usepackage{comment}

\numberwithin{equation}{section}

\makeatletter
\@namedef{subjclassname@2020}{\textup{2020} Mathematics Subject Classification}
\makeatother

\theoremstyle{plain}
\newtheorem{thm}{Theorem}[section]
\newtheorem{prop}[thm]{Proposition}
\newtheorem{cor}[thm]{Corollary}
\newtheorem{lemma}[thm]{Lemma}

\theoremstyle{definition}
\newtheorem{deff}[thm]{Definition}
\newtheorem{example}[thm]{Example}

\newcommand{\ran}{\bm{r}}
\newcommand{\domr}{\bm{d}}
\newcommand{\G}{\Gamma}

\newcommand{\im}{\operatorname{Im}}
\newcommand{\dom}{\operatorname{Dom}}
\newcommand{\pth}{\mathrm{Path}}

\newcommand{\so}{\mathbf{s}}
\newcommand{\ra}{\mathbf{r}}

\newcommand{\gr}{\operatorname{gr}}

\newcommand{\Typ}{\operatorname{Typ}}
\newcommand{\typ}{\operatorname{typ}}
\newcommand{\Gc}[1]{\boldsymbol #1 \operatorname{\bf-Gr}}
\newcommand{\Moc}[1]{\boldsymbol #1\operatorname{\bf-Mod}}
\newcommand{\FAct}[1]{\boldsymbol #1\operatorname{\bf-FAct}}
\newcommand{\FGrAct}[1]{\boldsymbol #1\operatorname{\bf-GrFAct}}
\newcommand{\inv}{^{-1}}

\def\a{\alpha}
\def\b{\beta}

\def\SS{\mathcal{S}}

\def\M{\mathbb{M}}
\def\E{\mathbb{E}}

\def \Z{\mathbb Z}
\def\-{\text{-}}

\begin{document}

\title{Graded semigroups}

\author{Roozbeh Hazrat}
\address{Roozbeh Hazrat:
Centre for Research in Mathematics and Data Science\\
Western Sydney University\\
AUSTRALIA} \email{r.hazrat@westernsydney.edu.au}

\author{Zachary Mesyan}
\address{Zachary Mesyan:
Department of Mathematics\\ University of Colorado, Colorado Springs, CO 80918
\\USA} \email{zmesyan@uccs.edu}



\date{\today}

\begin{abstract} 
We systematically develop a theory of graded semigroups, that is, semigroups $S$ partitioned by groups $\Gamma$, in a manner compatible with the multiplication on $S$. We define a smash product $S\#\Gamma$, and show that when $S$ has local units, the category $\Moc{S\# \Gamma}$ of sets admitting an $S\#\Gamma$-action is isomorphic to the category $\Gc{S}$ of graded sets admitting an appropriate $S$-action. We also show that when $S$ is an inverse semigroup, it is strongly graded if and only if $\Gc{S}$ is naturally equivalent to $\Moc{S_\varepsilon}$, where $S_\varepsilon$ is the partition of $S$ corresponding to the identity element $\varepsilon$ of $\Gamma$. These results are analogous to well-known theorems of Cohen/Montgomery and Dade for graded rings. Moreover, we show that graded Morita equivalence implies Morita equivalence for semigroups with local units, evincing the wealth of information encoded by the grading of a semigroup. We also give a graded Vagner-Preston theorem, provide numerous examples of naturally-occurring graded semigroups, and explore connections between graded semigroups, graded rings, and graded groupoids. In particular, we introduce graded Rees matrix semigroups, and relate them to smash product semigroups. We pay special attention to graded graph inverse semigroups, and characterise those that produce strongly graded Leavitt path algebras.
\end{abstract}

\maketitle

\tableofcontents

\section{Introduction}

The purpose of this paper is to build a theory of graded semigroups that parallels the theory of graded rings. We start with an overview of the motivating features of graded ring theory, which has become a vibrant subject, thanks to crucial applications to various  areas of mathematics.

Graded rings frequently appear when there is a group acting on an algebraic structure, or when the ring structure arises from a free construction modulo ``homogeneous'' relations. A ring $A$ is \emph{graded} by a group $\Gamma$ when, roughly, $A$ can be partitioned by $\Gamma$ in a way that is compatible with the structure of $A$. (See \S\ref{semigroupringsec} for more details.) Studying graded rings also naturally leads to studying representations of such rings, which can be compatibly partitioned by the relevant group, namely graded modules. Consequently, three categories play a prominent role in this setting: the category of left $A$-modules $\Moc{A}$, the category of graded left $A$-modules $\Gc{A}$, and the category of left $A_{\varepsilon}$-modules $\Moc{A_\varepsilon}$, where $A_{\varepsilon}$ is the subring of $A$ consisting of elements in the partition corresponding to the identity element $\varepsilon$ of $\Gamma$. A substantial portion of the theory of graded rings concerns the relationships between these categories. While the applications of this theory are numerous, one prominent example is the fundamental theorem of $K$-theory, proved by Quillen~\cite{quillen}, using the category of graded modules in a crucial way (see also~\cite{thomas}). 

Three constructions that play an essential role in the study of graded rings are strongly graded rings, smash products, and graded matrix rings. According to a theorem of Dade~\cite{dade}, a ring $A$ is strongly graded if and only if $\Gc{A}$ is naturally equivalent to $\Moc{A_\varepsilon}$. The smash product, constructed by Cohen and Montgomery~\cite{cohenmont}, allows one to produce a ring $A\#\Gamma$ such that $\Moc{A \# \Gamma}$ is isomorphic to $\Gc{A}$. Finally, graded matrix rings, aside from providing interesting examples, can be used to describe the relationship between graded rings $A$ and $B$ that are Morita equivalent, i.e., for which the categories $\Gc{A}$ and $\Gc{B}$ are equivalent. 
 
In this note we study grading on another class of algebraic objects, namely that of semigroups. Analogously to the case of rings, we say that a semigroup $S$ (with zero) is $\Gamma$-\emph{graded}, for some group $\Gamma$, if there is a partition or ``degree" map $\phi : S \setminus \{0\} \rightarrow \Gamma$ such that $\phi(st) = \phi(s)\phi(t)$ whenever $st \neq 0$. Our motivation for systematically studying graded semigroups comes from recent advances in the theory of combinatorial algebras, where graded rings played a prominent role. These algebras first arose in the work of Cuntz, in the context of operator algebras, and Leavitt, in the context of noncommutative rings. Ideas arising from these investigations were subsequently extended in various directions, leading to the paradigm summarised in the following diagram. (See~\cite{exel20081} for an account of these developments from the $C^*$ perspective, \cite{lisarooz} for the algebraic side, and~\cite{lawson1} for an exploration of the connections between some of the relevant semigroups and groupoids.) 
\begin{equation}\label{gffhdkcnejfksh}
\xymatrix@=1pc{ 
&&& C^*\text{-algebra} \ar@{<:>}[dd]\\
\text{combinatorial data} \ar[r] & \text{inverse semigroup} \ar[r] & \text{groupoid} \ar[ur] \ar[dr]\\
&&& \text{noncommutative algebra}}
\end{equation}
The now well-trodden path in (\ref{gffhdkcnejfksh}) starts with a natural inverse semigroup constructed from combinatorial data. The groupoid of germs of the semigroup turns out to be very well-behaved, and the corresponding convolution algebras (i.e., the groupoid $C^*$-algebra and the Steinberg algebra, discussed below in more detail) have very rich structures. These algebras are naturally graded, via lifting grading from the combinatorial data, and the grading plays a crucial role in describing their structure. For example, Cuntz and Krieger used it in their early work in the field, to prove graded uniqueness theorems. Our intention is to introduce and study the grading earlier along the path, pushing it from algebras to groupoids and inverse semigroups. 

The idea of assigning degrees to the elements of a semigroup has appeared in the literature before. For example, Howie \cite[p.\ 239]{howie} calls a semigroup $S$ equipped with a map $|\cdot |:S\rightarrow \mathbb N$, such that $|st|=|s|+|t|$, a \emph{semigroup with length}. He uses this construction to measure the lengths of words in a free semigroup. More relevantly to the paradigm described above, graded \emph{inverse} semigroups have been studied in connection with \'etale groupoids and Steinberg algebras~\cite{aradia,steinbergdia}. A more general notion has been explored in connection with congruences and Cayley graphs~\cite{ilic1,ilic2}. However, no systematic investigation of graded semigroups seems to have been undertaken before.

In addition to their relation to combinatorial algebras, graded semigroups actually arise quite naturally on their own. For example, it is well-known that any semigroup can be embedded in the full transformation semigroup of a set $X$. Now, if $X$ happens to come equipped with a map to a group, i.e., an assignment of degrees, then the collection of transformations of $X$ that respect the degrees constitutes a graded semigroup, and any graded semigroup can be embedded in one of this sort (Proposition~\ref{gradedrepresentationthrm}). An analogous statement can be proved for inverse semigroups (Proposition~\ref{grvagnerpreston}), giving a graded Vagner-Preston theorem. Also, for any $\Gamma$-graded ring $A$, the multiplicative semigroup of $A$ is likewise a $\Gamma$-graded semigroup. Moreover, one can construct graded analogues of Rees matrix semigroups (\S \ref{reesmatrixgroup}), and graph inverse semigroups are naturally $\Z$-graded (\S \ref{hfgfhdhksdkkls}). Finally, as mentioned before, free semigroups are likewise naturally $\Z$-graded, and this grading can be used to induce ones on various quotients of free semigroups, i.e., semigroups presented by generators and relations. Other examples are given below.

The heart of this paper consists of three categorical results, which parallel the aforementioned ones from ring theory. To state them, we require some notation. For a semigroup $S$ and a set $X$, we say that $X$ is a \emph{(unital pointed left) $S$-set} if there is an action of $S$ on $X$ such that $SX=X$, and $X$ has a distinguished ``zero" element $0_X$ (see \S \ref{grfdrferfsa} for more details). Also, if $S$ is $\Gamma$-graded, for some group $\Gamma$, then we say that an $S$-set $X$ is $\Gamma$-\emph{graded} if there is a function $X \setminus \{0_X\} \rightarrow \Gamma$ that respects the $S$-action. Let $\Moc{S}$ denote the category of unital pointed left $S$-sets, with functions that respect the $S$-action as morphisms, and let $\Gc{S}$ denote the subcategory of $\Moc{S}$ whose objects are the $\Gamma$-graded $S$-sets, and whose morphisms respect the $\Gamma$-grading. We show in Theorem~\ref{smashthrm} that, analogously to the Cohen/Montgomery result mentioned above, if $S$ has local units, then $\Gc{S}$ is isomorphic to $\Moc{S\# \Gamma}$, where $S\#\Gamma$ is a suitably defined \emph{smash product} for semigroups. We also show in Theorem~\ref{dadesthm} that when $S$ is an inverse semigroup, it is \emph{strongly} $\Gamma$-graded if and only if $\Gc{S}$ and $\Moc{S_\varepsilon}$ are naturally equivalent, in parallel to the aforementioned result of Dade. Our third categorical result, Theorem~\ref{hgbvgfhhf}, shows that for a pair of graded semigroups with local units, being \emph{graded} Morita equivalent implies being Morita equivalent, using the notion of Morita equivalence for semigroups introduced by Talwar \cite{talwar1} (see \S \ref{moritasection} for more details).

Aside from providing examples of graded semigroups, and studying the relevant categories, another goal of this paper is to investigate the relationships between graded semigroups, graded rings, and graded groupoids. In \S \ref{groupoidsection} we recall the relevant concepts about groupoids, show that strongly graded inverse semigroups produce strongly graded groupoids of germs, and describe the gradings on inverse semigroups constructed from strongly graded ample groupoids. In \S \ref{semigroupringsec} we review semigroup rings, show that strongly graded semigroups produce strongly graded semigroup rings, and relate smash product rings with smash product semigroups. In \S \ref{hfgfhdhksdkkls} we characterise the strongly graded graph inverse semigroups (Theorem~\ref{strgrgis}) and relate graph inverse semigroups to smash product semigroups (Theorem~\ref{gdhfthfhfhdsjje}). 

We are particularly interested in graph inverse semigroups, since they are built from graphs, as are other well-studied algebraic objects alluded to above, namely certain combinatorial algebras and groupoids. More specifically, starting from a graph $E$, in addition to the graph inverse semigroup $\SS(E)$, one can construct the graph groupoid $\mathscr{G}_E$ and the inverse semigroup $\mathscr{G}^h_E$ of slices of $\mathscr{G}_E$ (see \S \ref{topgroupoid} for more details, and~\cite{jones} for an explanation of the relations between these objects). From these groupoids and semigroups one can then build algebras, namely the Cohn path algebra $C_K(E)$ (which is a semigroup rings over $\SS(E)$), the Leavitt path algebra $L_K(E)$ (which is a certain quotient of the former), the graph groupoid Steinberg algebra $A_K(\mathscr{G}_E)$ (which is a convolution algebra over $\mathscr{G}_E$), and the enveloping algebra $K\langle \mathscr{G}^h_E \rangle$ of $\mathscr{G}^h_E$ (see \S \ref{hfgfhyhhvgfgffd}). All these algebras inherit natural $\Z$-gradings from $\SS(E)$ or $\mathscr{G}_E$, and the last three are graded isomorphic, for a fixed graph~\cite{CS,lisarooz}--see diagram below.
\begin{equation}\label{graphobjects}
\xymatrix{
& E \ar[d] \ar[dl]  \ar[dr] \\
\SS(E)   \ar[d]  & \mathscr{G}^h_E   \ar[d]& \mathscr{G}_E  \ar[d]\\
L_K(E) \ar @{} [r] |{\cong_{\gr}}  & K\langle \mathscr{G}^h_E  \rangle  \ar @{} [r] |{\cong_{\gr}} & A_K(\mathscr{G}_E)}
\end{equation}
In Theorem~\ref{hfghfyhff} and Corollary~\ref{hfghfyhff88} we characterise the graph inverse semigroups $\SS(E)$ for which the Leavitt path algebras $L_K(E)$ and graph groupoids $\mathscr{G}_E$ are strongly graded in the natural $\Z$-grading. We also describe the strongly graded Cohn path algebras $C_K(E)$ in Corollary~\ref{Cohncor}.

The paper concludes with some ideas for further research on our topic.

\section{Definitions and basics}

We begin by recalling relevant concepts from semigroup theory, defining \emph{graded} semigroups and related concepts, and providing some simple examples and observations.

\subsection{Graded semigroups} Recall that a \emph{semigroup} is a nonempty set equipped with an associative binary operation. A \emph{monoid} is a semigroup with an identity element $1$.  Throughout this note we assume that semigroups have a zero element $0$, unless specified otherwise. A semigroup $S$ is called \emph{regular} if every element $s\in S$ has an \emph{inner inverse} $t\in S$ such that $sts=s$. One can show that $S$ is regular if and only if for any $s\in  S$ there exists $t\in S$ such that $sts=s$ and $tst=t$. If every $s \in S$ has a unique inner inverse, denoted $s^{-1}$, then $S$ is called an \emph{inverse semigroup}.  

We say that a semigroup $S$ has \emph{local units} if for every $s \in S$ there exist $u,v \in E(S)$ such that $us=s=sv$, where $E(S)$ denotes the set of idempotents of $S$. It is easy to see that every regular semigroup has local units. A semigroup $S$ has \emph{common local units} if for all $s,t \in S$ there are idempotents $u,v \in S$ such that $us=s=sv$ and $ut=t=tv$. Clearly, every monoid has common local units. We refer the reader to~\cite{howie} for the theory of semigroups and~\cite{Lawson} for that of inverse semigroups. 
 
Next we define the main object of our interest.
 
\begin{deff}
Let $S$ be semigroup and $\Gamma$ a group. Then $S$ is called a $\Gamma$-\emph{graded semigroup} if there is a map $\phi : S \setminus\{0\} \to \Gamma$ such that $\phi(st) = \phi(s)\phi(t)$, whenever $st \not= 0$. For each $\alpha \in \Gamma$, we set $S_\alpha:=\phi^{-1}(\alpha) \cup \{0\}$. 

Equivalently, $S$ is a $\Gamma$-\emph{graded semigroup} if there exist subsets $S_\alpha$ of $S$ ($\alpha \in \Gamma$) such that
\[S = \bigcup_{\alpha \in \G} S_\alpha,\]
where $S_\a S_\b \subseteq S_{\a\b}$ for all $\a,\b \in \Gamma$, and $S_\a \cap S_\b = \{0\}$ for all distinct $\a,\b \in \Gamma$.
\end{deff}

Let $S$ be a $\Gamma$-graded semigroup. For each $\a \in \Gamma$ we refer to $S_{\alpha}$ as the \emph{component of $S$ of degree $\alpha$}. Also, for each $\a \in \Gamma$ and $s \in S_\alpha \setminus \{0\}$, we say that the \emph{degree} of $s$ is $\alpha$, and write $\deg(s)=\alpha$. Note that $\deg(s)=\varepsilon$ for all $s \in E(S) \setminus \{0\}$, and $\deg(s^{-1}) = \deg(s)^{-1}$ for all $s \in S \setminus \{0\}$ in the case where $S$ is an inverse semigroup. (Here, and throughout the article, the identity element of $\Gamma$ is denoted by $\varepsilon$.) The set $\big\{\a \in \Gamma \mid S_\alpha \not = \{0\} \big\}$ is called the \emph{support of $S$}. We say that $S$ is \emph{trivially graded} if the support of $S$ is contained in the trivial group $\{\varepsilon\}$, that is $S_\varepsilon =S$, in which case $S_\alpha =\{0\}$ for each $\alpha \in \Gamma \setminus \{\varepsilon\}$. Any semigroup admits a trivial grading by any group. It is also easy to see that $S_\varepsilon$ is a semigroup (with zero), and that $S_\varepsilon$ is an inverse semigroup whenever $S$ is.
 
A homomorphism $\phi:S \rightarrow T$ of $\Gamma$-graded semigroups is called a \emph{graded homomorphism} if $\phi(S_\alpha)\subseteq T_\alpha$ for every $\alpha \in \Gamma$. Thus a graded homomorphism is a homomorphism that preserves the degrees of the elements. 

\begin{example}
Given a group $\Gamma$, any free semigroup (with or without zero) $F=\langle x_i \mid i \in I \rangle$ can be made into a $\Gamma$-graded semigroup by assigning (freely) elements of $\Gamma$ to the generators $x_i$ of the semigroup. In particular, if $\Gamma = \mathbb{Z}$, the group of the integers, and we assign $1\in \mathbb Z$ to every generator of $F$, then $F=\bigcup_{i\in \mathbb N} F_n$, where $F_n$ is the set of words of length $n$, and so $F$ becomes a $\mathbb Z$-graded semigroup with support $\mathbb N$. 
\end{example}

\begin{example}\label{jgjgjejjiii}
Given a group $\Gamma$, a semigroup $S=\langle x_i \mid r_k=s_k\rangle$, defined by generators and relations, can be graded by assigning $\phi(x_i) \in \Gamma$ to each generator $x_i$, so that $\phi(r_k)=\phi(s_k)$, where $\phi : F\setminus\{0\} \rightarrow \Gamma$. In particular, any free inverse semigroup (with or without zero) can be graded in this manner. More concretely,
\begin{align*}
B&=\langle a,b \mid ab=1\rangle, \\
B'&=\langle a,b \mid a^2=0, b^2=0, aba=a, bab=b \rangle
\end{align*} 
are $\Gamma$-graded semigroups, via assigning any $\alpha\in \Gamma$ to $a$ and assigning $\alpha^{-1}$ to $b$.
\end{example}

In \S \ref{hfgfhdhksdkkls}, we investigate graph inverse semigroups, which constitute a vast class of semigroups that includes $B$ above. See also Example~\ref{mcalistereg}.
 
\begin{example}
Let $S=\bigcup_{\alpha \in \Gamma} S_\alpha$ be a $\Gamma$-graded semigroup, where $\Gamma$ is a torsion-free abelian group. For each $n \in \mathbb{Z}$ define the \emph{$n$-th Veronese semigroup} by $S^{(n)}:= \bigcup_{\alpha \in \Gamma} S_{n\alpha}$, where $S^{(n)}_\alpha=S_{n\alpha}$ for each $\alpha \in \Gamma$. Clearly $S^{(n)}$ is a subsemigroup of $S$, and if $S$ is a regular or an inverse semigroup, then so is $S^{(n)}$. Note that $S^{(-1)}=S$ as semigroups, but with the components flipped, i.e., $S^{(-1)}_\alpha= S_{-\alpha}$.
\end{example}
 
Recall that given a set $X$, the set $\mathcal{T}(X)$ of all functions $\psi :X\rightarrow X$ is a semigroup under composition of functions, with the empty function as the zero element, called the \emph{full transformation semigroup of $X$}. Our next goal is to show that every graded semigroup can be embedded in an appropriately defined graded subsemigroup of $\mathcal{T}(X)$.

We say that a set $X$ is \emph{pointed} if there is a distinguished element $0_X \in X$. Given a group $\Gamma$, we say that a pointed set $X$ is $\Gamma$-\emph{graded} if there is a map $\phi: X \backslash \{0_X\} \rightarrow \Gamma$. In this situation we set $X_\alpha=\phi^{-1}(\alpha) \cup \{0_X\}$ for each $\alpha \in \Gamma$. (In \S\ref{grfdrferfsa} we define a more elaborate version of this notion.) Denote by $\mathcal{T}'(X)$ the set of all pointed maps $\psi :X\rightarrow X$, i.e., ones for which $\psi(0_X)=0_X$. Clearly $\mathcal{T}'(X)$ is a subsemigroup of $\mathcal{T}(X)$, with $0_{\mathcal{T}'} \in \mathcal{T}'(X)$, defined by $0_{\mathcal{T}'}(x) = 0_X$ for all $x \in X$, as the zero element. For each $\alpha \in \Gamma$ define 
\[\mathcal{T}(X)_\alpha :=\big \{\psi \in \mathcal{T}'(X) \mid \psi (X_\beta) \subseteq X_{\alpha\beta} \text{ for all } \beta \in \Gamma \big \},\] and
\[\mathcal{T}^{\gr}(X) := \bigcup_{\alpha \in \Gamma} \mathcal{T}(X)_\alpha.\]
Then it is easy to check that $\mathcal{T}^{\gr}(X)$ is a $\Gamma$-graded subsemigroup of $\mathcal{T}'(X)$. 

\begin{prop} \label{gradedrepresentationthrm}
Let $S$ be a $\Gamma$-graded semigroup. Then there is a graded injective homomorphism $\psi:S \rightarrow \mathcal{T}^{\gr}(X)$ for some $\Gamma$-graded set $X$.
\end{prop}

\begin{proof}
Let $X= S^1:= S\cup\{1\}$ be the monoid obtained by adjoining an identity element $1$ to $S$. Then letting $\deg(1) = \varepsilon$, turns $X$ into a $\Gamma$-graded set, under the grading induced from that of $S$. For each $s \in S$ define $\theta_s:X \rightarrow X$ by $\theta_s(x)= sx$, and let 
\begin{align*}
\psi: S & \longrightarrow \mathcal{T}^{\gr}(X)\\ 
s & \longmapsto \theta_s.
\end{align*}
To show that this is a well-defined graded function, let $\alpha, \beta \in \Gamma$, $s \in S_\alpha$, and $x \in X_\beta$. If $sx = 0$, then $\theta_s(x) = 0 \in X_{\alpha\beta}$. Otherwise $\deg(\theta_s(x)) = \alpha\beta$, and so, once again, $\theta_s(x) \in X_{\alpha\beta}$. Thus $\theta_s \in \mathcal{T}(X)_\alpha$, from which it follows that $\psi$ is well-defined and graded. Moreover, $\psi$ is injective, since $\theta_s=\theta_t$ implies that $s=\theta_s(1)=\theta_t(1)=t$, for all $s,t \in S$. Finally, it is easy to see that $\psi$ is a homomorphism.
\end{proof}

We note that the ``regular" transformation semigroup $\mathcal{T}(X)$ would not have worked in the above context, since every grading on $\mathcal{T}(X)$ is trivial. To see this, note that if $\phi \in \mathcal{T}(X)$ is any constant map (i.e., one whose image has cardinality $1$), then $\phi^2 = \phi$, and hence $\deg(\phi)=\varepsilon$ in any grading on $\mathcal{T}(X)$. Thus, for all nonzero $\psi \in \mathcal{T}(X)$ we have 
\[\deg(\phi)\deg(\psi) = \deg(\phi\psi) = \deg(\phi) = \varepsilon,\] 
which implies that $\deg(\psi) = \varepsilon$. Similar reasoning shows that every grading on $\mathcal{T}'(X)$ is trivial as well.

A recurring theme of this paper is that there tends to be a close connection between the structure of $S$ and that of $S_\varepsilon$, for a graded semigroup $S$. We give the first two instances of this next. (See also, e.g., Proposition~\ref{bvgfjsirdsw} and Theorem~\ref{dadesthm}.)

Recall that for an inverse semigroup $S$, \emph{the natural partial order} $\leq$ on $S$ is defined by $s\leq t$ ($s,t \in S$) if $s=tu$ for some $u \in E(S)$. Equivalently, $s\leq t$ if $s=ut$ for some $u \in E(S)$. (See~\cite[\S 5.2]{howie} for more details.) In particular, if $u,v \in E(S)$, then $u \leq v$ amounts to $u=uv=vu$. Recall also that an inverse semigroup (with  zero) $S$ is called \emph{$0$-$E$-unitary} if for all $s \in S$ and $u \in E(S)\setminus \{0\}$, such that $u \leq s$, one has $s \in E(S)$.

\begin{prop} \label{0euni}
Let $S$ be a $\Gamma$-graded inverse semigroup. Then $S$ is $0$-$E$-unitary if and only if the inverse semigroup $S_{\varepsilon}$ is $0$-$E$-unitary.
\end{prop}

\begin{proof}
Suppose that $S_{\varepsilon}$ is $0$-$E$-unitary. Let $s \in S$ and $u \in E(S)\setminus \{0\}$ such that $u \leq s$, i.e., $u=sv$ for some $v \in E(S)$. Then $\deg(v) = \varepsilon = \deg(s)\deg(v)$, and so $s \in S_{\varepsilon}$. Thus, by hypothesis, $s \in E(S_{\varepsilon})$, from which it follows that $S$ is $0$-$E$-unitary.

The converse follows from the fact that $E(S_{\varepsilon}) = E(S)$.
\end{proof}
 
\begin{prop}\label{hfbvjkwofhbasks}
Let $S$ be a $\Gamma$-graded regular semigroup. Then sending each left (respectively, right) ideal $I$ of $S$ to $I_\varepsilon:=S_\varepsilon \cap I$ gives a one-to-one inclusion-preserving correspondence between the left (respectively, right) ideals of $S$ and the left (respectively, right) ideals of $S_\varepsilon$. 
\end{prop}

\begin{proof}
We treat only the case of left ideals, as the proof for right ideals is very similar.  

Clearly, $I_\varepsilon$ is a left ideal of $S_\varepsilon$, for each left ideal $I$ of $S$, and the map $I \mapsto I_\varepsilon$ is inclusion-preserving. Now, for each left ideal $J$ of $S_\varepsilon$ it is easy to see that $SJ$ is a left ideal of $S$. We conclude the proof by showing that $SI_\varepsilon = I$ and $(SJ)_\varepsilon = J$ for each relevant left ideal.

Let $I$ be a left ideal of $S$. Then clearly $SI_\varepsilon \subseteq I$. For the opposite inclusion, let $s\in I \setminus \{0\}$. Then $s=sts$ for some $t \in S$, and so comparing the degrees gives $\deg(ts)=\varepsilon$. Thus $s=sts\in SI_\varepsilon$, from which it follows that $SI_\varepsilon = I$.

Next, let $J$ be a left ideal of $S_\varepsilon$, and let $s\in J \setminus \{0\}$. Writing $s=sts$ for some $t\in S$, a degree comparison again gives $st\in S_\varepsilon$, and so $s=sts \in S_\varepsilon \cap SJ = (SJ)_\varepsilon$. It follows that $J \subseteq (SJ)_\varepsilon$. Now let $r \in (SJ)_\varepsilon \setminus \{0\}$. Then $r=st$ for some $s \in S$ and $t \in J$. Since $\deg(r) = \varepsilon = \deg(t)$, necessarily $s \in S_\varepsilon$, and so $r \in J$. Thus $(SJ)_\varepsilon = J$.
\end{proof}
  
We show in Example~\ref{gfbchfkhidsbhr} that the above proposition cannot be extended to two-sided ideals. 

\begin{example} \label{mcalistereg}
Recall that the McAlister inverse semigroup $M_n$ (for $n$ a positive integer) is the  inverse semigroup generated by $\{x_1, \dots, x_n\}$, subject to the relations $x_i x_j^{-1}=0=x_i^{-1} x_j$, for all $i\neq j$. (See~\cite[\S 9.4]{Lawson} for more details.) Then sending $x_i \mapsto 1$ and $x_i^{-1} \mapsto -1$ for each $1 \leq i \leq n$ induces a $\Z$-grading on $M_n$ (see Example~\ref{jgjgjejjiii}).

Using the normal form for elements of $M_n$ \cite[Proposition 9.4.11]{Lawson}, it is easy to show that $(M_n)_0=E(M_n)$. Thus, by Proposition~\ref{0euni}, each McAlister inverse semigroup is $0$-$E$-unitary, giving an alternative quick proof of this well-known fact. Moreover, by Proposition~\ref{hfbvjkwofhbasks}, the left (and right) principal ideals of $M_n$ are in one-to-one correspondence with the sets $\{u\in E(M_n) \mid u\leq v\}$, for $v \in E(M_n)$.
\end{example}

\subsection{Strongly graded semigroups} \label{strgrsect}

Let us next introduce \emph{strongly graded} semigroups--a particularly interesting class of graded semigroups, which we study in detail throughout the rest of the paper. 

\begin{deff} \label{strgrdef}
Let $S$ be a $\Gamma$-graded semigroup. We say that $S$ is \emph{strongly $\Gamma$-graded} if $S_\a S_\b =S_{\a\b}$ for all $\a,\b\in \G$.  
\end{deff}

Here are a couple of simple examples of strongly graded semigroups.

\begin{example}
Let $\Gamma'$ be any group, and let $T=\Gamma' \cup \{0\}$ (the \emph{group with zero} corresponding to $\Gamma'$). For any group homomorphism $\phi: \Gamma' \rightarrow \Gamma$, it is easy to see that $T$ is a strongly $\Gamma$-graded (inverse) semigroup if and only if $\phi$ is surjective.  
\end{example}

\begin{example}
Let $I \subseteq \mathbb{Z}$ be an interval, i.e., a subset with the property that if $i,k \in I$ and $i < j <k$ for some $i,j,k \in \mathbb{Z}$, then $j\in I$. Set $B : =(I\times I) \cup \{0_B\}$, and define multiplication on $B$ by
\[(p,q)(s,t) =
\left\{ \begin{array}{ll}
(p,t) & \text{if } q=s   \\
0_B & \text{otherwise,} 
\end{array}\right.\] 
and \[(p,q) \cdot 0_B = 0_B = 0_B\cdot (p,q)\] for all $(p,q), (s,t) \in B \setminus \{0_B\}$. Then it is easy to see that $B$ is an inverse semigroup, with $(p,q)^{-1} = (q,p)$ for each $(p,q) \in B \setminus \{0_B\}$. Moreover, sending $(p,q) \mapsto p-q$ turns $B$ into a $\mathbb Z$-graded semigroup.

Now, if $I=\mathbb Z$, then for all $i,j \in \mathbb{Z}$ and $(p,q) \in B_{i+j}$ we have $(p,p-i) \in B_{i}$, $(p-i,q) \in B_{j}$, and $(p,p-i)(p-i,q) = (p,q)$. Thus, $B_{i+j} \subseteq  B_{i}B_{j}$, and so $B$ is strongly $\mathbb{Z}$-graded in this case. It is also easy to show that, conversely, if $I \neq \mathbb{Z}$, then $B$ is not strongly $\mathbb{Z}$-graded.
\end{example}

Recall that given a semigroup $S$ and $s,t \in S$ we write $s \, \mathscr{L} \, t$ if $S^1 s = S^1 t$ and $s \, \mathscr{R} \, t$ if $s S^1 = t S^1$, where $S^1$ denotes the monoid obtained from $S$ by adjoining an identity element. These (along with $\mathscr{J}$, $\mathscr{H}$, and $\mathscr{D}$, which we will not review here) are known as \emph{Green's relations}.

For a $\Gamma$-graded unital ring $A$ (see \S \ref{semigroupringsec} for more details), being strongly graded is equivalent to $1\in A_{\alpha} A_{\alpha^{-1}}$ for all $\alpha \in \Gamma$. The following is an analogue of this statement for semigroups, which gives various convenient characterisations of strongly graded semigroups.

\begin{prop}\label{hgfgftgr}
Let $S$ be a $\G$-graded semigroup with local units. Then the following are equivalent.

\begin{enumerate}[\upshape(1)]

\item $S$ is strongly graded;

\smallskip

\item $S_\alpha S_{\alpha^{-1}}=S_\varepsilon$ for every $\alpha \in  \G$;

\smallskip

\item $S_\alpha S_{\alpha^{-1}}$ contains all the local units of $S$, for every $\alpha \in  \G$.

\end{enumerate}

Moreover, if $S$ is an inverse semigroup, then these are also equivalent to the following.

\begin{enumerate}[\upshape(1)]  \setcounter{enumi}{3}

\item $E(S)= \{ss^{-1} \mid s \in S_\alpha\}$, for every $\alpha \in  \G$;

\smallskip

\item $E(S)= \{s^{-1}s \mid s \in S_\alpha\}$, for every $\alpha \in  \G$;

\smallskip

\item For all $u\in E(S)$ and $\alpha \in \Gamma$, there exists $s\in S_\alpha$ such that $u\, \mathscr{L} \, s$;

\smallskip

\item For all $u\in E(S)$ and $\alpha \in \Gamma$, there exists $s\in S_\alpha$ such that $u\, \mathscr{R} \, s$.
\end{enumerate}
\end{prop}

\begin{proof}
(1) $\Rightarrow$ (2) This follows immediately from Definition~\ref{strgrdef}. 

(2) $\Rightarrow$ (3) Since the local units of $S$ are idempotent, they are elements $S_\varepsilon$. Thus if $S_\varepsilon= S_\alpha S_{\alpha^{-1}}$ for some $\alpha \in \Gamma$, then $S_\alpha S_{\alpha^{-1}}$ contains the local units. 

(3) $\Rightarrow$ (1) Let $\alpha, \beta \in \Gamma$ and $s\in S_{\alpha \beta}$. Then there is a local unit $u \in S$ such that $su=s$. By (3), $u=rt$, for some $r \in S_{\beta^{-1}}$ and $t\in S_{\beta}$, and so $s=su=(sr)t\in S_\alpha S_\beta$. It follows that $S_{\alpha \beta}=S_\alpha S_\beta$, and so $S$ is strongly graded. 

\smallskip

Suppose now that $S$ is an inverse semigroup. 

(2) $\Rightarrow$ (4) Let $\alpha \in \Gamma$. Then clearly $\{ss^{-1} \mid s \in S_\alpha\} \subseteq E(S)$. For the opposite inclusion, let $u\in E(S)$. Then, by (2), $u=st$ for some $s\in S_\alpha$ and $t \in S_{\alpha^{-1}}$. Since $u \in E(S)$, we have $u=u^{-1}$, and hence 
\begin{equation}\label{gfgfgfhdshd}
u=st=st(st)^{-1}=(stt^{-1})(tt^{-1}s^{-1}),
\end{equation}
where $stt^{-1}\in S_\alpha$. It follows that $u\in \{ss^{-1} \mid s \in S_\alpha\}$. 

(4) $\Rightarrow$ (2) Let $s\in S_{\varepsilon}$ and $\alpha \in \Gamma$. Then, by (4), $s^{-1} s = t^{-1} t$ for some $t \in S_{\alpha^{-1}}$. Hence $s =ss^{-1} s=(st^{-1})t \in S_\alpha S_{\alpha^{-1}}$. Thus  $S_{\varepsilon} = S_\alpha S_{\alpha^{-1}}$.

(4) $\Leftrightarrow$ (5) This follows from the fact that $s \in S_\alpha$ if and only if $s^{-1} \in S_{\alpha^{-1}}$.

(5) $\Rightarrow$ (6) Given $u \in E(S)$ and $\alpha \in \Gamma$, we have $u = s^{-1}s$ for some $s \in S_{\alpha}$, by (5). Since $s = su$, it follows that $u\, \mathscr{L} \, s$.

(6) $\Rightarrow$ (5) Given $u \in E(S)$ and $\alpha \in \Gamma$, we have $u = st$ for some $s \in S$ and $t \in S_{\alpha}$, by (6), where necessarily $s \in S_{\alpha^{-1}}$. Then $u = (stt^{-1})(tt^{-1}s^{-1})$, by (\ref{gfgfgfhdshd}), where $tt^{-1}s^{-1} \in S_{\alpha}$. Therefore $u \in \{s^{-1}s \mid s \in S_\alpha\}$, from which (5) follows.

(4) $\Leftrightarrow$ (7) Since (5) is equivalent to (6), this follows, by symmetry.
\end{proof}

The next result gives a first glimpse at the special nature of strongly graded semigroups.

\begin{prop}\label{bvgfjsirdsw}
Let $S$ be a strongly $\Gamma$-graded semigroup with local units. Then the following hold. 
\begin{enumerate}[\upshape(1)]

\item $S$ is a regular semigroup if and only if $S_\varepsilon$ is a regular semigroup. 

\smallskip

\item $S$ is an inverse semigroup if and only if $S_\varepsilon$ is an inverse semigroup. 

\end{enumerate}
\end{prop}

\begin{proof}
(1) Suppose that $S$ is regular, and let $s\in S_\varepsilon \setminus \{0\}$. Then there exists $t \in S$ such that $sts=s$ and $tst=t$. A degree comparison shows that $\deg(s)=\deg(t)=\varepsilon$, and hence $S_\varepsilon$ is regular. 

Conversely, suppose that $S_\varepsilon$ is regular, and let $s\in S \setminus \{0\}$. Then $s\in S_\alpha$,  for some $\alpha\in \Gamma$, and $s=su$, for a local unit $u \in S_\varepsilon$. Since $S$ is strongly graded, $u=rt$ for some $r \in S_{\alpha^{-1}}$ and $t\in S_\alpha$. Now, $sr\in S_\varepsilon$, and so $sr=srpsr$ for some $p\in S_\varepsilon$, as $S_\varepsilon$ is regular. Hence 
\[s=su=srt=srpsrt=s(rp)s,\] which shows that $S$ is regular.  
  
(2) Recall that a regular semigroup $S$ is an inverse semigroup if and only if $E(S)$ is commutative \cite[Theorem 5.1.1]{howie}. Since $E(S)=E(S_\varepsilon)$, the statement follows from (1). 
\end{proof}

Next we use the natural partial order to define a weaker version of ``strongly graded" for inverse semigroups, which will be of use in the sequel.

\begin{deff} \label{locallystrgrdef}
Let $S$ be a $\Gamma$-graded inverse semigroup. We say that $S$ is \emph{locally strongly $\Gamma$-graded} if for all $\a,\b\in \G$ and $s\in S_{\a\b} \setminus \{0\}$, there exists $t \in S_\alpha S_\beta \setminus \{0\}$ such that  $t\leq s$.
\end{deff}

Clearly any strongly graded inverse semigroup is locally strongly graded. Now, given a $\Gamma$-graded inverse semigroup $S$ and $\alpha \in \G$, set $E(S)_\alpha:=\{ss^{-1} \mid s \in S_\alpha\}$. According to Proposition~\ref{hgfgftgr}, $S$ is strongly graded if and only if $u\in E(S)_\alpha$ for all $u\in E(S)$ and $\alpha\in \Gamma$. In the next proposition we show that $S$ is locally strongly graded if and only if a local version of this condition holds, and provide other equivalent statements. 

\begin{prop}\label{sussa}
Let $S$ be a $\G$-graded inverse semigroup. Then the following are equivalent.

\begin{enumerate}[\upshape(1)]
\item $S$ is locally strongly graded;

\smallskip

\item For all $\alpha \in \Gamma$ and $u \in E(S) \setminus \{0\}$, there exists $v \in E(S)_\alpha \setminus \{0\}$ such that $v \leq u$;

\smallskip

\item For all $\a,\b\in \G$ and $s\in S_{\a\b} \setminus \{0\}$, there exists $u \in  E(S) \setminus \{0\}$ such that  $u\leq s^{-1}s$ and $su\in S_\alpha S_\beta$;

\smallskip

\item For all $\a,\b\in \G$ and $s\in S_{\a\b} \setminus \{0\}$, there exists $u \in  E(S) \setminus \{0\}$ such that  $u\leq ss^{-1}$ and $us\in S_\alpha S_\beta$.
\end{enumerate}
\end{prop}

\begin{proof}
(1) $\Rightarrow$ (4) Let $\a,\b\in \G$ and $s\in S_{\a\b}\setminus \{0\}$. By (1), there exists $t \in S_\alpha S_\beta \setminus \{0\}$ such that  $t\leq s$. Then $t=vs$ for some $v \in E(S) \setminus \{0\}$. Setting $u = tt^{-1}$, and using the fact that $E(S)$ is commutative, gives \[u = vss^{-1}v = vss^{-1} \leq ss^{-1}.\] Moreover, $us = t(t^{-1}s) \in S_\alpha S_\beta$, since $\deg(t^{-1}) = \b^{-1}\a^{-1}$.

(4) $\Rightarrow$ (3) Let $\a,\b\in \G$ and $s\in S_{\a\b}\setminus \{0\}$. By (4), there exists $u \in  E(S) \setminus \{0\}$ such that  $u\leq ss^{-1}$ and $us\in S_\alpha S_\beta$. Then we have $us = uss^{-1}s = s(s^{-1}us)$, where $s^{-1}us \in E(S)\setminus \{0\}$ and $s^{-1}us \leq s^{-1}s$.

(3) $\Rightarrow$ (2) Let $\alpha\in \G$ and $u \in E(S) \setminus \{0\}$. Then $u \in S_\varepsilon = S_{\alpha\alpha^{-1}}$. By (3), there exists $v \in  E(S) \setminus \{0\}$ such that  $v\leq u$ and $uv\in S_{\alpha}S_{\alpha^{-1}}$. Since $uv = v$, it follows that $v = st$ for some for some $s\in S_\alpha$ and $t \in S_{\alpha^{-1}}$. The computation in (\ref{gfgfgfhdshd}) then gives $v = (stt^{-1})(t t^{-1} s^{-1})$, where $stt^{-1}\in S_\alpha$. Hence $v \in E(S)_\alpha$. 

(2) $\Rightarrow$ (1) Let $\a,\b\in \G$ and $s\in S_{\a\b}\setminus \{0\}$. By (2), there exists $v \in E(S)_{\beta^{-1}} \setminus \{0\}$ such that $v \leq s^{-1}s$. Then $v=rr^{-1}$, where $r\in S_{\beta^{-1}}$ and $r^{-1} \in S_{\beta}$. Therefore, setting $t = sv$, we have $t \leq s$ and $t=(sr)r^{-1} \in S_\alpha S_\beta$. Moreover, $v \neq 0$ and $v \leq s^{-1}s$ imply that $t \neq 0$.
\end{proof}

In Proposition~\ref{meme} we exhibit a class of locally strongly graded inverse semigroups that are not strongly graded. Also, in \S \ref{hfgfhdhksdkkls} we discuss locally strongly graded graph inverse semigroups.

\subsection{Graded $S$-sets}\label{grfdrferfsa}

Next we introduce $S$-sets, which are of central importance to much of what follows.

Let $S$ be a semigroup and $X$ a set. Then $X$ is a \emph{left} \emph{$S$-set}, or a \emph{left} \emph{$S$-act}, if there is an action $S \times X \to X$ of $S$ on $X$, such that $s(tx)=(st)x$ for all $s,t \in S$ and $x \in X$. We say that a left $S$-set $X$ is \emph{unital} if $SX=X$ (i.e, for all $x \in X$ there exist $y \in X$ and $s \in S$ such that $sy=x$). Also, a left $S$-set $X$ is called \emph{pointed} if there exists a ``zero" element $0_X \in X$ such that $0x = 0_X$ for all $x \in X$ (in which case, necessarily $s0_X = 0_X$ for all $s \in S$). \emph{Right (unital pointed) $S$-set} is defined analogously. We say that $X$ is an $S$-\emph{biset} if it is both a left $S$-set and a right $S$-set, and $(sx)t=s(xt)$ for all $s,t\in S$ and $x \in X$. Note that this condition implies that the zero elements corresponding to the left and right $S$-set structures of an $S$-biset are equal. Throughout the paper, unless mentioned otherwise, all $S$-sets are assumed to be pointed. 

For left (respectively, right) $S$-sets $X$ and $Y$, a function $\phi:X\rightarrow Y$ is called an $S$-\emph{map} if $\phi(sx)=s\phi(x)$ (respectively, $\phi(xs) = \phi(x)s$) for all $s\in S$ and $x\in X$. If $X$ and $Y$ are $S$-bisets, then $\phi:X\rightarrow Y$ is called an $S-S$-\emph{map} if $\phi(sx)=s\phi(x)$ and $\phi(xs) = \phi(x)s$ for all $s\in S$ and $x\in X$. Note that in each of these situations we have $\phi(0_X) = 0_Y$.

To extend the $S$-set construction to the graded setting, let us now assume that $S$ is $\Gamma$-graded. A left $S$-set $X$ is \emph{$\Gamma$-graded} if there is a function $\phi : X \setminus \{0_X\} \to \Gamma$ that satisfies $\phi (sx) = \deg(s)\phi (x)$, for all $s \in S$ and $x \in X$, whenever $sx \neq 0_X$. For each $\alpha \in \Gamma$, we set $X_\alpha := \phi^{-1}(\alpha)\cup \{0\}$. Equivalently, $X$ is a $\Gamma$-\emph{graded} left $S$-set if there exist subsets $X_\alpha$ of $X$ ($\alpha \in \Gamma$) such that 
\[ X  = \bigcup_{\alpha \in \G} X_\alpha, \]
where $S_\a X_\b \subseteq X_{\a\b}$ for all $\a,\b \in \Gamma$, and $X_\a \cap X_\b = \{0_X\}$ for all distinct $\a,\b \in \Gamma$. For all $\alpha \in \Gamma$ and $x\in X_\alpha \setminus \{0\}$, we say that the \emph{degree} of $x$ is $\alpha$ and write $\deg(x)=\alpha$. \emph{$\Gamma$-graded} is defined analogously for right $S$-sets. An $S$-biset $X$ is \emph{$\Gamma$-graded} if it is $\Gamma$-graded as a left $S$-set, and additionally $X_\beta S_\alpha\subseteq X_{\a\b}$ for all $\a,\b \in \Gamma$.

For a semigroup $S$, we denote by $\Moc{S}$ the category whose objects are unital (pointed) left $S$-sets, and whose morphisms are $S$-maps. We denote by $\Gc{S}$ the category whose objects are $\Gamma$-graded unital (pointed) left $S$-sets, and whose morphisms are \emph{graded} $S$-maps, that is, $S$-maps $\phi : X \to Y$ such that $\phi(X_\alpha) \subseteq Y_\alpha$ for all $\alpha \in \Gamma$. We refer the reader to~\cite{bergman} for general category-theoretic information.

Given a $\Gamma$-graded left $S$-set $X$, for all $\alpha, \beta \in \Gamma$ let $X(\alpha)_\beta :=X_{\beta\alpha}$. Then 
\begin{equation}\label{butyestergdtgf}
X(\alpha): = \bigcup_{\beta \in \Gamma}X(\alpha)_\beta
\end{equation}
is a $\Gamma$-graded left $S$-set, called the \emph{$\alpha$-shift} of $X$. That is, $X(\alpha) = X$ as sets, but the grading is ``shifted" by $\alpha$. For a $\Gamma$-graded right $S$-set $X$, the \emph{$\alpha$-shift} is defined by setting $X(\alpha)_\beta : = X_{\alpha\beta}$ for all $\beta \in \Gamma$. This construction leads to a \emph{shift functor} for each $\alpha \in \Gamma$, defined by
\begin{align}\label{hfbvhjfkushke}
\mathcal{T}_\alpha: \Gc{S} &\longrightarrow \Gc{S} \\
X &\longmapsto X(\alpha),\notag
\end{align}
which takes each morphism to itself. It is easy to see that for all $\alpha, \beta \in \Gamma$, the functor $\mathcal{T}_\alpha$ is an isomorphism, and that $\mathcal{T}_\alpha \mathcal{T}_\beta= \mathcal{T}_{\alpha\beta}$. Shifting plays a crucial role in our theory of graded semigroups, similarly to the analogous concept in the graded ring theory. 

The following lemma gives a characterisation of strongly graded semigroups $S$ in terms of their actions on graded $S$-sets, which will be needed in \S\ref{iuhfgbdftie}. 

\begin{lemma}\label{strongly-gr-lemma}
Let $S$ be a $\Gamma$-graded semigroup with local units. Then the following are equivalent.
\begin{enumerate} [\upshape(1)]
\item $S$ is strongly graded; 

\smallskip

\item $S_\alpha X_\beta = X_{\alpha\beta}$ for all $\alpha, \beta \in \Gamma$ and unital $\Gamma$-graded left $S$-sets $X$;

\smallskip

\item $S X_\alpha = X$ for all $\alpha \in \Gamma$ and unital $\Gamma$-graded left $S$-sets $X$.
\end{enumerate}
\end{lemma}

\begin{proof}
(1) $\Rightarrow$ (2) Let $X$ be a unital $\Gamma$-graded left $S$-set, $\alpha, \beta \in \Gamma$, and $x \in  X_{\alpha\beta}$. Since $X$ is unital, $sy = x$ for some $s \in S$ and $y \in X$.  Let $\gamma = \alpha^{-1}\deg(s)$. By (1), we can then find $r \in S_\alpha$ and $t \in S_\gamma$ such that $s = rt$. Then, comparing degrees on the two sides of $rty=x$, gives $ty \in X_\beta$, and so $x \in S_\alpha X_\beta$. It follows that $S_\alpha X_\beta = X_{\alpha\beta}$.

(2) $\Rightarrow$ (3) By (2), we have $S_{\beta\alpha^{-1}}X_\alpha = X_{\beta}$ for all $\alpha,\beta \in \Gamma$ and unital $\Gamma$-graded left $S$-sets $X$, from which (3) follows.

(3) $\Rightarrow$ (1) Let $\alpha,\beta \in \Gamma$, and let $s \in S_{\alpha\beta}$. Since $S$ has local units, $X=S$ is a unital left $S$-set. Thus, by (3), we can find $r \in S$ and $t \in S_\beta$ such that $s = rt$. Then $\deg(r) = \alpha\beta\beta^{-1} = \alpha$, and so $s \in S_\alpha S_\beta$. It follows that $S_\alpha S_\beta = S_{\alpha \beta}$, proving (1).
\end{proof}

Next we recall tensor products for left $S$-sets, and extend them to the graded setting. Let $S$ be a semigroup, and let $X$, respectively $Y$, be a right, respectively left, $S$-set. Then $X\otimes_S Y$ is defined as the cartesian product $X\times Y$ modulo the equivalence relation generated by identifying $(xs,y)$ with $(x,sy)$ for all $x\in X, y\in Y, s\in S$. The equivalence class of $(x,y)$ is denoted by $x\otimes y$. (See \cite[\S 8]{howie} for more details.) If $X$ is an $S$-biset, then $X\otimes_S Y$ can be made into a left $S$-set, by defining $s(x\otimes y) :=sx\otimes y$. (It follows from \cite[Proposition 8.1.8]{howie} that this is well-defined.) Similarly, if $Y$ is an $S$-biset, then $X\otimes_S Y$ can be made into a right $S$-set, by defining $(x\otimes y)s :=x\otimes ys$. Note that 
$x\otimes 0_Y=0_X\otimes y=0_X\otimes 0_Y$ for all $x \in X$ and $y \in Y$, and that this acts as the zero element in each case, making $X\otimes_S Y$ pointed. Next, if $S$, $X$, and $Y$ are $\Gamma$-graded, then $X\otimes_S Y$ can be turned into a $\Gamma$-graded (left or right) $S$-set by letting
\begin{equation} \label{tensorgrade}
(X\otimes_S Y)_\alpha = \big \{x\otimes y \mid \deg(x)\deg(y)=\alpha\big \} \cup \{0_X\otimes 0_Y\}
\end{equation}
for all $\alpha \in \Gamma$. (Again, it follows from \cite[Proposition 8.1.8]{howie} that this is well-defined.) We note that if $T$ is a subsemigroup of $S$, $X$ is a ($\Gamma$-graded) $S$-biset, and $Y$ is a ($\Gamma$-graded) left $S$-set, then a slight modification of the above construction turns $X\otimes_T Y$ into a ($\Gamma$-graded) left $S$-set, and analogously for right $S$-sets.

Given a $\Gamma$-graded semigroup $S$, clearly $S$ and $S_\alpha$ are $S_\varepsilon$-bisets for each $\alpha \in \Gamma$. This fact, together with the tensor construction, can be used to give yet another characterisation of strongly graded semigroups. 
 
\begin{prop}\label{dtbhsdjkw}
Let $S$ be a $\Gamma$-graded semigroup with local units. Then the following are equivalent. 
\begin{enumerate}[\upshape(1)]

\item $S$ is strongly graded;

\smallskip 

\item For every $\alpha \in \Gamma$, the function $ \phi_\alpha: S_\alpha \otimes_{S_{\varepsilon}} S_{\alpha^{-1}}  \rightarrow S_\varepsilon$ defined by $x \otimes y \mapsto xy,$ is a surjective $S_{\varepsilon}-S_{\varepsilon}$-map;

\smallskip 

\item For every $\alpha \in \Gamma$, the function $\psi_\alpha: S \otimes_{S_{\varepsilon}} S_{\alpha}  \rightarrow S$ defined by $x \otimes y \mapsto xy,$ is a graded surjective $S-S_\varepsilon$-map. 
\end{enumerate}

Furthermore, if $S$ has common local units then the above maps are bijective. 
\end{prop}
 
\begin{proof}
(1) $\Rightarrow$ (2)  From the definition of tensor product it is easy to check that $\phi_\alpha$ is well-defined and is a $S_{\varepsilon}-S_{\varepsilon}$-map. If $S$ is strongly graded, then $S_\alpha S_{\alpha^{-1}}=S_\varepsilon$ for every $\alpha \in \Gamma$, which implies that $\phi_\alpha$ is surjective.

(2) $\Rightarrow$ (3) Again it is easy to see that $\psi_\alpha$ is well-defined and an $S-S_{\varepsilon}$-map. For all $\alpha, \beta \in \Gamma$, using (\ref{tensorgrade}), we have
\[\psi_\alpha \big ((S \otimes_{S_{\varepsilon}} S_{\alpha})_\beta \big)=\psi_\alpha \big (S_{\beta\alpha^{-1}} \otimes_{S_{\varepsilon}} S_{\alpha} \big)\subseteq S_{\beta\alpha^{-1}} S_\alpha\subseteq S_{\beta},\]
which shows that $\psi_\alpha$ is a graded map. To show that $\psi_\alpha$ is surjective, let $s\in S\setminus \{0\}$, and let $u \in E(S)$ be a local unit such that $s=su$. By (2), $u=rt$, where $r\in S_{\alpha}$ and $t\in S_{\alpha^{-1}}$. Therefore, $\psi_\alpha(sr \otimes t)=srt=su=s$, as desired. 
 
(3) $\Rightarrow$ (1) Let $s\in S_\varepsilon$ and $\alpha \in \Gamma$. Then, by (3), there are $r\in S$ and $t\in S_{\alpha}$ such that $\psi_\alpha (r\otimes t)=rt=s$. It follows that $r \in S_{\alpha^{-1}}$, and thus $S_\varepsilon = S_{\alpha^{-1}} S_\alpha$. Therefore, $S$ is strongly graded, by Proposition~\ref{hgfgftgr}. 
 
Now suppose that $S$ has common local units and that $S$ is strongly graded. To show that $\phi_\alpha$ is injective (for $\alpha \in \Gamma$), suppose further that $\phi_\alpha(x \otimes y) = \phi_\alpha(x' \otimes y')$ for some $x \otimes y, x' \otimes y' \in S_\alpha \otimes_{S_{\varepsilon}} S_{\alpha^{-1}}$. Then $xy=x'y'$. Let $u \in S_\varepsilon$ be a common local unit for $y$ and $y'$, so that $yu=y$ and $y'u=y'$. Since $S$ is strongly graded, $u=rt$ for some $r \in S_\alpha$ and $t \in S_{\alpha^{-1}}$. Since $yr, y'r \in S_\varepsilon$, we then have 
 \[x\otimes y = x\otimes yu= x\otimes yrt= xyr \otimes t= x'y'r \otimes t= x' \otimes y'rt= x' \otimes y'u=x' \otimes y',\] 
showing $\phi_\alpha$ is injective, and hence bijective. The argument for $\psi_\alpha$ being injective is similar. 
\end{proof}

We conclude this section with an example, where we build new semigroups from graded ones.

\begin{example}
Let $S$ be a $\Gamma$-graded semigroup, where $\Gamma$ is a totally-ordered abelian group (e.g., $\mathbb Z$). Then it is easy to see that 
\[S_{\geq 0}:=\bigcup_{\alpha\geq 0}S_\alpha  \text{ \,\,\, and  \,\,\, } S_{\leq 0}:=\bigcup_{\alpha\leq 0}S_\alpha,\] 
are semigroups, and that $S_{\geq \beta}:=\bigcup_{\alpha\geq \beta}S_\alpha$ is an $S_{\geq 0}$-biset for each $\beta \in \Gamma$. When $S$ is strongly graded, tensor product calculus similar to that in Proposition~\ref{dtbhsdjkw} can be used here. Specifically, $S_{\geq \beta} \otimes_{S_{\geq 0}} S_{\geq \gamma} \rightarrow S_{\geq \beta\gamma}$, defined by $x\otimes y \mapsto xy$, is a surjective $S_{\geq 0}-S_{\geq 0}$-map, for all $\beta, \gamma \in \Gamma$. 
\end{example}
 
\section{Smash products and matrices} 
 
\subsection{Smash product semigroups}

In this section we prove the first of our main results, which is an analogue of a theorem \cite[Theorem 2.2]{cohenmont} about graded modules over algebras. It says that for a $\Gamma$-graded semigroup $S$ with local units, the category $\Gc{S}$ of graded unital left $S$-sets is isomorphic to the category of non-graded unital left sets over a certain other semigroup. This construction requires the smash product, which we introduce next.
 
\begin{deff}\label{hghghgqq}
Let $S$ be a $\Gamma$-graded semigroup. We define the \emph{smash product of $S$ with $\Gamma$} as
\[S\# \Gamma :=\{s P_\alpha \mid s\in S \setminus \{0\}, \alpha \in \Gamma \}\cup \{0\}.\]
Also, define a binary operation on $S\# \Gamma$ by
\begin{equation} \label{smashmult1}
(s P_\alpha) (t P_\beta) =
\left\{ \begin{array}{ll}
stP_\beta & \text{if } st \not = 0  \text{ and } t\in S_{\alpha \beta^{-1}} \\
0 & \text{otherwise } 
\end{array}\right.
\end{equation}
and \[0=0^2=(s P_\alpha)0=0(s P_\alpha),\] for all $s,t \in S$ and $\alpha, \beta \in \Gamma$.
\end{deff}

We will next show that $S\# \Gamma$ is a semigroup, which inherits various characteristics of $S$. It is easy to see that $S\# \Gamma$ is $\Gamma$-graded via
\begin{align}\label{ghfgdthfyr1}
S\# \Gamma \backslash \{0\} &\longrightarrow \Gamma \\\notag
sP_\alpha &\longmapsto \deg(s).  \notag
\end{align}

\begin{lemma}\label{hgyfhyfhf1}
Let $S$ be a $\Gamma$-graded semigroup. Then the following hold.
\begin{enumerate}[\upshape(1)]

\item $S\# \Gamma$ is a semigroup.

\smallskip 

\item $E(S\# \Gamma)=\big \{u P_\alpha \mid u \in E(S) \backslash \{0\}, \alpha \in \Gamma \big \} \cup \{0\}$.

\smallskip 

\item $S$ has local units if and only if $S\#\Gamma$ has local units. 

\smallskip 

\item $S$ is an inverse semigroup if and only if $S\#\Gamma$ is an inverse semigroup. 
\end{enumerate}
\end{lemma} 

\begin{proof}
(1) It suffices to check that multiplication of nonzero elements is associative. Thus let $r, s,t \in S$ and $\alpha, \beta, \gamma \in \Gamma$. Then
\begin{align*}
((r P_\alpha) (s P_\beta)) (t P_\gamma) 
& = \left\{ \begin{array}{ll}
rstP_\gamma & \text{if } rst \not = 0, s\in S_{\alpha \beta^{-1}},  \text{ and } t\in S_{\beta\gamma^{-1}} \\
0 & \text{otherwise } 
\end{array}\right.  \\
& = \left\{ \begin{array}{ll}
rst P_\gamma  & \text{if } rst \not = 0, st\in S_{\alpha\gamma^{-1}}, \text{ and } t\in S_{\beta\gamma^{-1}}\\
0 & \text{otherwise } 
\end{array}\right. \\
& = (r P_\alpha) ((s P_\beta) (t P_\gamma)), 
\end{align*}
giving the desired conclusion.

(2) This follows immediately from the definition of the multiplication on $S\# \Gamma$.

(3) Suppose that $S$ has local units, and let $s \in S \setminus \{0\}$ and $\alpha \in \Gamma$. Then there are idempotents $u,v\in E(S)$ such that $s=vs=su$. Hence
\begin{equation*}
sP_\alpha = v P_{\beta \alpha} sP_\alpha = sP_\alpha u P_\alpha,
\end{equation*}
where $\beta = \deg(s)$, and $v P_{\beta \alpha}, u P_\alpha \in E(S\# \Gamma)$, by (2). It follows that $S\# \Gamma$ has local units. A similar argument gives the converse. 

(4) Suppose that $S$ is an inverse semigroup, and let $s \in S \setminus \{0\}$ and $\alpha \in \Gamma$. Then 
\begin{equation*}
(sP_\alpha)(s^{-1}P_{\beta \alpha})(sP_\alpha)=sP_\alpha,
\end{equation*}
where $\beta = \deg(s)$, from which it follows that $S\#\Gamma$ is a regular semigroup. Since $E(S)$ is commutative, by \cite[Theorem 5.1.1]{howie}, so is $E(S \# \Gamma)$, by (2). Thus $S\#\Gamma$ is an inverse semigroup.

Conversely, suppose that $S\#\Gamma$ is an inverse semigroup, and let $s \in S \setminus \{0\}$. Then 
\begin{equation*}
(sP_\varepsilon)(tP_{\alpha})(sP_\varepsilon)=sP_\varepsilon
\end{equation*}
for some $t \in S$ and $\alpha \in \Gamma$, which implies that $sts=s$. Thus $S$ is regular. Since $E(S\#\Gamma)$ is commutative, by (2), it follows that $E(S)$ is commutative, and so $S$ is an inverse semigroup.
\end{proof}

In preparation for Theorem~\ref{smashthrm}, we next relate $S$-sets to $S\#\Gamma$-sets.
 
\begin{lemma}\label{gdgdgdgdgss1}
Let $S$ be a $\Gamma$-graded semigroup and $X$ a $\Gamma$-graded left $S$-set. Then $X$ is a left $S\# \Gamma$-set via the action defined by $0x=0$, and 
\begin{equation}\label{dgbcgftgdjsaa}
(sP_\alpha)x =\left\{ \begin{array}{ll}
s x & \text{if } x \in X_\alpha \\
0 & \text{otherwise } 
\end{array}\right.
\end{equation}
for all $s \in S$, $\alpha \in \Gamma$, and $x \in X$. Moreover, if $X$ is a unital left $S$-set, then it is a unital left $S\#\Gamma$-set. 
\end{lemma}

\begin{proof}
Let $s,t \in S \setminus \{0\}$, $\alpha, \beta \in \Gamma$, and $x \in X$. Then 
\begin{align*}
((sP_\alpha)(tP_\beta)\big)x 
& = \left\{ \begin{array}{ll}
stx & \text{if }  t\in S_{\alpha \beta^{-1}}  \text{ and } x\in X_\beta \\
0 & \text{otherwise } 
\end{array}\right. \\ 
& = \left\{ \begin{array}{ll}
stx & \text{if }  tx \in X_\alpha  \text{ and } x\in X_\beta\\
0 & \text{otherwise } 
\end{array}\right. 
= (sP_\alpha) ((tP_\beta)x).
\end{align*}
Upon dealing with trivial cases involving zero, it follows that $X$ is a (pointed) left $S\# \Gamma$-set. The final claim is immediate.
\end{proof}

\begin{lemma}\label{gdgdgdgdgss2}
Let $S$ be a $\Gamma$-graded semigroup with local units, and let $X$ be a unital left $S\#\Gamma$-set. For all $s\in S$, $\alpha \in \Gamma$, and $x \in X$ let 
\begin{equation} \label{fdrefgdhgdaa}
X_\alpha =\big \{y\in X \mid \exists \, u P_\alpha\in E(S\#\Gamma) \text{ such that } (u P_\alpha) y=y \big\},
\end{equation}
define
\begin{equation}\label{mememein}
sx =\left\{ \begin{array}{ll}
(sP_\alpha ) x & \text{if } x \in X_\alpha  \text{ and } s \neq 0 \\
0 & \text{otherwise } 
\end{array}\right.,
\end{equation}
and define $\phi: X \backslash \{0\} \rightarrow \Gamma$ by $x \mapsto \alpha$ whenever $x \in X_\alpha$. Then $X$ is a $\Gamma$-graded unital left $S$-set via the above operations.
\end{lemma}

\begin{proof}
To show that the operations are well-defined, we need to prove that 
\begin{equation} \label{unioneqn}
X = \bigcup_{\alpha \in \Gamma} X_\alpha,
\end{equation}
and $X_\alpha \cap X_\beta = \{0\}$ for all distinct $\alpha, \beta \in \Gamma$. Let $x \in X \setminus \{0\}$. Since $X$ is a unital left $S\#\Gamma$-set, there exist $s P_\alpha\in S\#\Gamma$ and $y \in X$ such that $x = (s P_\alpha) y$. Since $S$ has local units, $us=s$ for some $u \in E(S)$. Then 
\[(uP_{\beta\alpha})x = (uP_{\beta\alpha}) (s P_\alpha) y = (s P_\alpha) y = x,\] 
where $\beta = \deg(s)$. Since $uP_{\beta\alpha} \in E(S\#\Gamma)$, by Lemma \ref{hgyfhyfhf1}(2), we have $x \in X_{\beta\alpha}$, from which (\ref{unioneqn}) follows.

Next let $\alpha, \beta \in \Gamma$ be distinct, and let $x\in X_\alpha \cap X_\beta$. Then we have $x= (uP_\alpha) x = (v P_\beta) x$ for some $u,v \in S_\varepsilon$, and so \[x=(uP_\alpha) x = (uP_\alpha)(vP_\beta )x= 0x=0.\] Thus  $X_\alpha \cap X_\beta = \{0\}$, as desired.

For all $s,t \in S \setminus \{0\}$ and $x \in X$ we have
\begin{align*}
s(tx) & = \left\{ \begin{array}{ll}
(sP_\alpha)((tP_\beta)x) & \text{if } x\in X_{\beta} \text{ and } tx\in X_{\alpha} \\
0 & \text{otherwise } 
\end{array}\right.  \\ 
& = \left\{ \begin{array}{ll}
((sP_\alpha)(tP_\beta))x & \text{if } x\in X_{\beta} \text{ and } t \in S_{\alpha\beta^{-1}} \\
0 & \text{otherwise } 
\end{array}\right. \\ 
& = \left\{ \begin{array}{ll}
(stP_\beta)x & \text{if } x\in X_{\beta} \text{ and } st \neq 0\\
0 & \text{otherwise } 
\end{array}\right.  
= (st)x, 
\end{align*}
from which it follows that $X$ is a (pointed) left $S$-set. Also, given $\alpha \in \Gamma$ and $x \in X_\alpha \setminus \{0\}$, there exist $u P_\alpha\in E(S\#\Gamma)$ such that $x = (u P_\alpha) x = ux$, by the definition of $X_\alpha$, from which it follows that $X$ is a unital left $S$-set.

Finally, to show that $X$ is $\Gamma$-graded as a left $S$-set, it suffices to check that $S_\alpha X_\beta \subseteq X_{\alpha \beta}$ for all $\alpha, \beta \in \Gamma$. Thus let $s \in S_\alpha \setminus \{0\}$ and $x \in X_\beta$. Since $S$ has local units, there is an idempotent $u \in S$ such that $us=s$. Then \[sx= (sP_\beta)x = (uP_{\alpha\beta })((sP_\beta )x) \in X_{\alpha\beta },\] as desired.
\end{proof}

We are now ready to show that the category of $\Gamma$-graded left $S$-sets is isomorphic to the category of left $S\#\Gamma$-sets. In the process we will show that shifting (\ref{butyestergdtgf}) is preserved by our isomorphism. For this we first define functors $\mathcal{T}_\alpha: \Moc{S\#\Gamma} \rightarrow \Moc{S\#\Gamma}$ which behave in a manner analogous to the shifting functors (\ref{hfbvhjfkushke}) on $\Gc{S}$.

Given a $\Gamma$-graded semigroup $S$, for each $\alpha\in \Gamma$ define
\begin{align*}
\tau_\alpha: S\#\Gamma & \longrightarrow S\#\Gamma\\
sP_\beta &\longmapsto sP_{\beta\alpha}.
\end{align*}
Letting $\tau_\alpha (0) = 0$, it is easy to see that $\tau_\alpha$ is a semigroup isomorphism. Each such isomorphism $\tau_\alpha$ induces a new $S\#\Gamma$-set structure on any left $S\#\Gamma$-set $X$ via
\begin{equation}\label{boundgsfret}
(sP_\beta).x= \tau_\alpha(sP_\beta )x= (sP_{\beta\alpha}) x,
\end{equation}
for all $sP_\beta \in S\#\Gamma \setminus \{0\}$ and $x\in X$. We denote this induced left $S\# \Gamma$-set by $X(\alpha)$. It is easy to check that sending each left $S\#\Gamma$-set $X$ to $X(\alpha)$ and each morphism to itself gives an isomorphism $\mathcal{T}_\alpha: \Moc{S\#\Gamma} \rightarrow \Moc{S\#\Gamma}$ of categories. Moreover $\mathcal{T}_\alpha \mathcal{T}_\beta =\mathcal{T}_{\alpha\beta}$ for all $\alpha, \beta \in \Gamma$. 

\begin{thm} \label{smashthrm}
Let $S$ be a $\Gamma$-graded semigroup with local units. Then there is an isomorphism of categories $\mathcal{F}_{\#}: \Gc{S} \rightarrow \Moc{S\#\Gamma}$ such that the following diagram commutes for every $\a \in \G$.
\begin{equation}\label{inducedshift}
\xymatrix{\Gc{S}\ar[r]^{\mathcal{F}_{\#} \ \ \ }\ar[d]_{\mathcal{T}_{\a}}&\Moc{S\#\Gamma}\ar[d]^{\mathcal{T}_{\a}}\\
\Gc{S}\ar[r]^{\mathcal{F}_{\#} \ \ \ } &\Moc{S\#\Gamma}.}
\end{equation}
\end{thm}

\begin{proof}
We begin by defining mappings $\mathcal{F}_{\#}: \Gc{S} \to \Moc{S\#\Gamma}$ and $\mathcal{F}_{\gr} : \Moc{S\#\Gamma} \to \Gc{S}$. For each object $X$ in $\Gc{S}$ let $\mathcal{F}_{\#}(X) = X_{\#}$, where $X_{\#}$ is defined to be $X$, viewed as a left $S \# \Gamma$-set, as in Lemma \ref{gdgdgdgdgss1}. For each morphism $\phi$ in $\Gc{S}$ let $\mathcal{F}_{\#}(\phi) = \phi_{\#}$, where $\phi_{\#} = \phi$. Next, for each object $X$ in $\Moc{S\#\Gamma}$ let $\mathcal{F}_{\gr}(X) = X_{\gr}$, where $X_{\gr}$ is defined to be $X$, viewed as a $\Gamma$-graded left $S$-set, as in Lemma \ref{gdgdgdgdgss2}. For each morphism $\phi$ in $\Moc{S\#\Gamma}$ let $\mathcal{F}_{\gr}(\phi) = \phi_{\gr}$, where $\phi_{\gr} = \phi$. We claim that $\mathcal{F}_{\#}$ and $\mathcal{F}_{\gr}$ are functors. To show this, it suffices to prove that each $\phi_{\#}$ is a morphism in $\Moc{S\#\Gamma}$, and each $\phi_{\gr}$ is a morphism in $\Gc{S}$.

Let $\phi: X \to Y$ be a morphism in $\Gc{S}$, i.e., a function that commutes with the $S$-action and respects the $\Gamma$-grading, and let $sP_\alpha\in S\#\Gamma \setminus \{0\}$ and $x \in X_{\#}$. Then, by (\ref{dgbcgftgdjsaa}) in Lemma \ref{gdgdgdgdgss1}, we have 
\begin{align*}
\phi_{\#}((sP_\alpha)x) 
& = \left\{ \begin{array}{ll}
\phi(sx) & \text{if } x \in X_\alpha    \\
0 & \text{otherwise } 
\end{array}\right. \\
& = \left\{ \begin{array}{ll}
s\phi(x) & \text{if } x \in X_\alpha    \\
0 & \text{otherwise } 
\end{array}\right. 
=(sP_\alpha) \phi_{\#}(x).
\end{align*}
It follows that $\phi_{\#}: X_{\#} \to Y_{\#}$ is a morphism in $\Moc{S\#\Gamma}$.

Next, let $\phi: X \to Y$ be a morphism in $\Moc{S\#\Gamma}$, and let $s \in S\setminus \{0\}$ and $x\in X$. Also, let $\alpha \in \Gamma$ be such that $x \in X_\alpha$, as in Lemma~\ref{gdgdgdgdgss2}. To see that $\phi_{\gr}: X_{\gr} \to Y_{\gr}$ is a morphism in $\Gc{S}$, first notice that $(uP_\alpha) \phi(x) = \phi((uP_\alpha) x)$ for any $u \in E(S) \setminus \{0\}$, and so, by (\ref{fdrefgdhgdaa}), $x \in X_\alpha$ implies that $\phi_{\gr} (x) = \phi(x) \in Y_\alpha$. Thus, by (\ref{mememein}), we have
\[\phi_{\gr}(sx) = \phi((sP_\alpha ) x) = (sP_\alpha)\phi(x) = s\phi_{\gr}(x).\]
We conclude that $\phi_{\gr}: X_{\gr} \to Y_{\gr}$ is a morphism in $\Gc{S}$.

It remains to prove that $\mathcal{F}_{\#}\circ \mathcal{F}_{\gr}$  and $\mathcal{F}_{\gr} \circ \mathcal{F}_{\#}$ are the identity functors. Clearly, $\mathcal{F}_{\gr} \circ \mathcal{F}_{\#}(\phi) = \phi$ for any morphism $\phi$ in $\Gc{S}$, and $\mathcal{F}_{\#}\circ \mathcal{F}_{\gr}(\phi) = \phi$ for any morphism $\phi$ in $\Moc{S\#\Gamma}$. 

Let $X$ be an object in $\Moc{S\#\Gamma}$. Note that for all $s \in S\backslash \{0\}$, $x \in X$, and $\alpha, \beta \in \Gamma$, if $x\in X_{\beta}$ and $\beta \neq \alpha$, then 
\[(sP_\alpha)x=(sP_\alpha)(uP_\beta)x=0x=0,\]
where $u \in S$ is such that $(uP_\beta)x = x$ (see (\ref{fdrefgdhgdaa})). It follows that $(sP_\alpha)x = 0$ unless $x \in X_{\alpha} \setminus \{0\}$. Thus, for all $s \in S\backslash \{0\}$, $\alpha \in \Gamma$, and $x \in X$, applying (\ref{mememein}) and then (\ref{dgbcgftgdjsaa}), it is immediate that $(sP_\alpha)x$, viewed as an element of $X$, agrees with $(sP_\alpha)x$, viewed as an element of $(X_{\gr})_{\#}$. Therefore the $S\#\Gamma$-action on $X$ is the same as that on $(X_{gr})_{\#}$. 

Next let $X$ be an object in $\Gc{S}$, let $s \in S\backslash \{0\}$, and let $x \in X$. Applying (\ref{dgbcgftgdjsaa}) and then (\ref{mememein}), it is apparent that $sx$, viewed as an element of $X$, agrees with $sx$, viewed as an element of $(X_{\#})_{\gr}$. Finally, using (\ref{dgbcgftgdjsaa}) and (\ref{fdrefgdhgdaa}), it is easy to see that $(X_{\#})_\alpha=X_\alpha$, for any $\alpha\in \Gamma$. Thus the $S$-action on $X$ agrees with that on $(X_{\#})_{\gr}$, as does the $\Gamma$-grading. Hence $\mathcal{F}_{\#}\circ \mathcal{F}_{\gr}$ and $\mathcal{F}_{\gr} \circ \mathcal{F}_{\#}$ are the identity functors, and so $\Moc{S}$ is isomorphic to $\Gc{S}$.

Finally, to show that diagram (\ref{inducedshift}) commutes for $\alpha \in \Gamma$, it suffices to check that $\mathcal{T}_\alpha \mathcal{F}_{\#} (X) =\mathcal{F}_{\#}  \mathcal{T}_\alpha (X)$ for all objects $X$ of $\Gc{S}$, i.e., that $(X_{\#})(\alpha)$ and $X(\alpha)_{\#}$ agree as objects of $\Moc{S\#\Gamma}$. Let $sP_\beta \in S\#\Gamma \setminus \{0\}$ and $x\in X$. Viewing $x$ as an element of $(X_{\#})(\alpha)$, by (\ref{boundgsfret}) and (\ref{dgbcgftgdjsaa}), we have
\[(sP_\beta).x = (sP_{\beta\alpha}) x
=\left\{ \begin{array}{ll}
s x & \text{if } x\in X_{\beta\alpha}\\
0 & \text{otherwise } 
\end{array}\right.. \]
On the other hand, viewing $x$ as an element of $X(\alpha)_{\#}$, by (\ref{butyestergdtgf}) , we have 
\[(sP_\beta).x 
=\left\{ \begin{array}{ll}
s x & \text{if } x\in X(\alpha)_\beta = X_{\beta\alpha}\\
0 & \text{otherwise } 
\end{array}\right.. \]
Thus (\ref{inducedshift}) commutes.
\end{proof}

\subsection{Graded Rees matrix semigroups}\label{reesmatrixgroup}
 
Matrix semigroups were introduced by Rees, who proved that all primitive $0$-simple semigroups are of this form. (See~\cite[\S 3.2]{howie} for details, and the Introduction of~\cite{james} for historical background and further motivating examples.) Here we construct graded versions of Rees matrix semigroups, thereby obtaining another class of examples of graded semigroups, which turn out to be related to smash product semigroups. Our construction parallels that of graded matrix rings \cite[\S 1.3]{hazi}. In \S \ref{semigroupringsec} we further use these semigroups to build an interesting class of graded rings.
 
Let $S$ be a $\Gamma$-graded semigroup, and let $I$ and $J$ be nonempty (index) sets. For all $i \in I$ and $j \in J$, fix $\alpha_i, \beta_j \in \Gamma$, and set $\overline{\alpha}=(\alpha_i)_{i\in I} \in \Gamma^I$ and $\overline{\beta}=(\beta_j)_{j\in J}\in \Gamma^J$. For each $\delta \in \Gamma$, denote by $M_\delta:=\M_{I,J}(S)[\overline{\alpha}][\overline{\beta}]_\delta$ the subset of $I \times S \times J$ consisting of all \emph{matrices} $(s_{ij})$, such that $s_{ij} \in S$ and $\deg(s_{ij}) = \alpha_i \delta \beta_j ^{-1}$, for all $i\in I$ and $j \in J$. For all $k\in I$, $l\in J$, and $a \in S$ with $\deg(a) = \alpha_i \delta \beta_j ^{-1}$, we denote by $e_{kl}(a)$ the \emph{elementary matrix} $(s_{ij}) \in M_\delta$, where $s_{kl} = a$, and $s_{ij} = 0$ for $(i,j) \neq (k,l)$. We denote by $0$ the matrix all of whose entires are zero. (So $0 = e_{kl}(0)$ for all $k\in I$, $l\in J$.) Let
\begin{equation}\label{Reesgrade}
\E_{I,J}(S)[\overline{\alpha}][\overline{\beta}]_\delta:= \big\{e_{ij}(a) \mid i\in I, j \in J, a\in S_{\alpha_i \delta \beta_j^{-1}} \big\} \subseteq M_{\delta},
\end{equation}
and 
\begin{equation}\label{Reessemi}
\E_{I,J}(S)[\overline{\alpha}][\overline{\beta}]:=\bigcup_{\delta \in \Gamma} \E_{I,J}(S)[\overline{\alpha}][\overline{\beta}]_\delta.
\end{equation}
To define a multiplication operation on this set, first choose a \emph{sandwich matrix} $p=(p_{ji}) \in \M_{J,I}(S^1)[\overline{\beta}][\overline{\alpha}]_\varepsilon$, where, as usual, $S^1$ denotes the monoid obtained by adjoining an identity element to $S$. Then for all $e_{ij}(a), e_{kl}(b) \in \E_{I,J}(S)[\overline{\alpha}][\overline{\beta}]$, define
\begin{equation}\label{hfbvhsphwq}
e_{ij}(a)e_{kl}(b):=e_{il}(ap_{jk}b).
\end{equation}
It is easy to see that this operation is associative, and so $\E_{I,J}(S)[\overline{\alpha}][\overline{\beta}]$, which we denote by $\E_{I,J}^p(S)[\overline{\alpha}][\overline{\beta}]$ in this context, is a semigroup with respect to it. 

Next, let us check that letting $T_\delta = \E_{I,J}(S)[\overline{\alpha}][\overline{\beta}]_\delta$ for all $\delta \in \Gamma$, we have $T_\gamma T_\lambda \subseteq T_{\gamma \lambda}$ for all $\gamma, \lambda \in \Gamma$. Thus let $e_{ij}(a) \in T_\gamma$ and $e_{kl}(b)\in T_\lambda$. Then $\deg(a) = \alpha_i \gamma\beta_j^{-1}$, $\deg(b) = \alpha_k \lambda\beta_l^{-1}$, and $\deg(p_{jk}) = \beta_j \varepsilon \alpha_k^{-1}$. Therefore 
\[ap_{jk}b \in S_{\alpha_i \gamma\beta_j^{-1}} S_{\beta_j \varepsilon \alpha_k^{-1}} S_{\alpha_k \lambda\beta_l^{-1}}\subseteq S_{\alpha_i \gamma\lambda\beta_l^{-1}}.\] 
It follows from~(\ref{hfbvhsphwq}) that $e_{ij}(a)e_{kl}(b) \in T_{\gamma \lambda}$, and so $T_\gamma T_\lambda \subseteq T_{\gamma \lambda}$. We conclude that $\E_{I,J}^p(S)[\overline{\alpha}][\overline{\beta}]$ is a $\Gamma$-graded semigroup, to which we refer as a \emph{graded Rees matrix semigroup}. Note that if $\Gamma$ is the trivial group, then the above construction reduces to the usual Rees matrix semigroup \cite[\S 3.2]{howie}. 

If $I$ and $J$ are finite, say $I = \{1, \dots, m\}$ and $J = \{1, \dots, n\}$, then taking $[\overline{\alpha}] = (\alpha_1,\dots,\alpha_m)$ and $[\overline{\beta}]=(\beta_1,\dots,\beta_n)$, each component set $M_\delta$, defined above, can be visualised as 
\[\begin{pmatrix}
S_{\alpha_1 \delta \beta_1^{-1}} & S_{\alpha_1 \delta  \beta_2^{-1}} & \cdots &
S_{\alpha_1 \delta \beta_n^{-1}} \\
S_{\alpha_2 \delta \beta_1^{-1}} & S_{\alpha_2 \delta \beta_2^{-1}} & \cdots &
S_{\alpha_2 \delta \beta_n^{-1}} \\
\vdots  & \vdots  & \ddots & \vdots  \\
S_{\alpha_m \delta \beta_1^{-1}} & S_{\alpha_m \delta \beta_2^{-1}} & \cdots &
S_{\alpha_m \delta \beta_n^{-1}}
\end{pmatrix}.\]

\begin{example}\label{gfbchfkhidsbhr}
Let $G$ be a group and $S=G \cup \{0\}$ the corresponding group with zero. Also let $I=J=\{1,2\}$ and let 
\[p=\left(\begin{array}{cc}
\varepsilon & 0 \\
0 & \varepsilon     
\end{array}
\right).\]
Then the corresponding Rees matrix semigroup $R = I \times S \times J$ can be represented as follows.
\[R=\left(\begin{array}{cc}
S & 0 \\
0 & 0     
\end{array}\right)
\bigcup
\left(\begin{array}{cc}
0 & S \\
0 & 0     
\end{array}\right)
\bigcup
\left(\begin{array}{cc}
0 & 0 \\
S & 0     
\end{array}\right)
\bigcup
\left(\begin{array}{cc}
0 & 0 \\
0 & S
\end{array}\right),\]
where multiplication becomes the usual matrix multiplication. By the Rees theorem \cite[Theorem 3.2.3]{howie}, $R$ is regular and completely $0$-simple (i.e., it has no nonzero proper ideals, and possesses a minimal idempotent). 

Now let $S \setminus \{0\} \rightarrow \mathbb Z/ 2\mathbb Z$ be the trivial grading for $S$ (so $S_0 = S$ and $S_1 = \{0\}$). Also, keeping $I$, $J$, and $p$ as before, let $\alpha_1=\beta_1=0$ and $\alpha_2=\beta_2=1$. Then the graded Rees matrix semigroup $T$ has the following graded components (see (\ref{Reesgrade}) and (\ref{Reessemi})):
\[T_0=\left(\begin{array}{cc}
S & 0 \\
0 & 0     
\end{array}\right)
\bigcup
\left(\begin{array}{cc}
0 & 0 \\
0 & S    
\end{array}\right) \, \, \text{ and } \, \, \, 
T_1=\left(\begin{array}{cc}
0 & S \\
0 & 0     
\end{array}\right)
\bigcup
\left(\begin{array}{cc}
0 & 0  \\
S & 0
\end{array}\right).\]
Thus $T=R$ as semigroups, and therefore $T$ is also 0-simple. On the other hand, $T_0$ has two nontrivial two-sided ideals (cf.\ Proposition~\ref{hfbvjkwofhbasks}). 
\end{example} 

Next, we consider a special case of the graded Rees matrix semigroup construction, which will shed additional light on the smash product semigroups discussed above. Let $S$ be a $\Gamma$-graded semigroup, let $I=J=\Gamma$, let $\overline{\alpha} = \overline{\beta}=(\delta)_{\delta \in \Gamma}$, and let $p=(p_{ji}) \in \M_{J,I}(S^1)[\overline{\beta}][\overline{\alpha}]_\varepsilon$ be the identity matrix. (That is, $p_{ji} = 1$ if $j=i$, and $p_{ji} = 0$ otherwise.) In this case we denote the semigroup $\E_{I,J}^p(S)[\overline{\alpha}][\overline{\beta}]$ by $S_\Gamma$. So 
\[S_\Gamma = \big \{e_{\alpha \beta} (s) \mid \alpha, \beta \in \Gamma,  s\in S \big \},\] with multiplication given by 
\begin{equation} \label{stableprod}
e_{\alpha \beta} (s)e_{\delta \gamma} (t) = 
\left\{ \begin{array}{ll}
e_{\alpha\gamma} (st) & \text{if } \beta = \delta \\
0 & \text{otherwise } 
\end{array}\right.,
\end{equation}
(see (\ref{hfbvhsphwq})), and grading given by
\begin{align} \label{ghfgdthfyr}
S_\Gamma \backslash \{0\} & \longrightarrow \Gamma\\
e_{\alpha \beta}(s) & \longmapsto \alpha^{-1} \deg(s) \beta. \notag
\end{align}
We refer to $S_\Gamma$ as a \emph{stable graded Rees matrix semigroup}. The semigroup $T$ constructed in Example~\ref{gfbchfkhidsbhr} is of this sort.
 
\begin{prop}\label{hgyfhyfhf}
Let $S$ be a $\Gamma$-graded semigroup with local units and $S_\Gamma$ the corresponding stable graded Rees matrix semigroup. Then the following hold.
\begin{enumerate}[\upshape(1)]
\item $E(S_\Gamma)= \{e_{\alpha \alpha}(u) \mid u\in E(S), \alpha \in \Gamma\}$. 

\smallskip 

\item $S_\Gamma$ has local units. 

\smallskip 

\item $S_\Gamma$ is strongly graded.

\smallskip 

\item Defining 
\begin{align*}
\phi :S \# \Gamma & \longrightarrow (S_\Gamma)_\varepsilon \\ \notag
sP_\alpha & \longmapsto e_{\deg(s) \alpha, \alpha}(s) \notag
\end{align*}
for all $sP_\alpha \in S \# \Gamma  \setminus \{0\}$, and $\phi(0)=0$, gives an isomorphism between $S \# \Gamma$ and $(S_\Gamma)_\varepsilon$.

\smallskip

\item $S$ is an inverse semigroup if and only if $S_\Gamma$ is an inverse semigroup. 
\end{enumerate}
\end{prop} 

\begin{proof}
(1) This follows easily from the above description (\ref{stableprod}) of multiplication in $S_\Gamma$.

(2) Let $e_{\alpha \beta}(s)\in S_\Gamma$, and let $u,v \in E(S)$ be such that $us=sv=s$. Then $e_{\alpha \alpha}(u), e_{\beta \beta}(v) \in E(S_\Gamma)$, by (1), and \[e_{\alpha \alpha}(u) e_{\alpha \beta}(s)= e_{\alpha \beta}(s) =e_{\alpha \beta}(s) e_{\beta \beta}(v),\] by (\ref{stableprod}). Hence $S_\Gamma$ has local units.

(3) By Proposition~\ref{hgfgftgr}, it suffices to show that $({S_\Gamma})_\varepsilon \subseteq ({S_\Gamma})_\gamma ({S_\Gamma})_{\gamma^{-1}}$, for all $\gamma \in \Gamma$. Thus let $e_{\alpha \beta}(s) \in ({S_\Gamma})_\varepsilon \setminus \{0\}$, let $\gamma \in \Gamma$, and let $v \in E(S)$ be such that $sv=s$. Then we have $e_{\alpha \beta}(s)=e_{\alpha, \beta\gamma }(s)e_{\beta\gamma, \beta}(v)$. Also, by (\ref{ghfgdthfyr}), $\deg(s) = \alpha\beta^{-1}$, and 
\[\deg(e_{\alpha, \beta\gamma }(s)) = \alpha^{-1}\deg(s)\beta\gamma = \alpha^{-1}\alpha\beta^{-1}\beta\gamma = \gamma.\] 
Similarly $\deg(e_{\beta\gamma, \beta}(v)) = \gamma^{-1}$, and so $e_{\alpha \beta}(s) \in ({S_\Gamma})_\gamma ({S_\Gamma})_{\gamma^{-1}}$. Thus, $({S_\Gamma})_\varepsilon \subseteq ({S_\Gamma})_\gamma ({S_\Gamma})_{\gamma^{-1}}$.

(4) For all $sP_\alpha, tP_\beta \in S \# \Gamma \setminus \{0\}$, we have
\begin{align*}
\phi(sP_\alpha)\phi (tP_\beta) 
& = e_{\deg(s) \alpha, \alpha}(s)e_{\deg(t) \beta, \beta}(t) \\
& = \left\{ \begin{array}{ll}
e_{\deg(st) \beta, \beta}(st) & \text{if } \deg(t)=\alpha\beta^{-1} \text{ and } st \neq 0  \\
0 & \text{otherwise } 
\end{array}\right. \\ 
& = \left\{ \begin{array}{ll}
\phi(stP_\beta) & \text{if } t \in S_{\alpha\beta^{-1}} \text{ and } st \neq 0  \\
0 & \text{otherwise } 
\end{array}\right.  \\
& = \phi((sP_\alpha)(tP_\beta)),
\end{align*}
and so $\phi$ is a homomorphism, which is clearly surjective. Also, if $e_{\deg(s) \alpha, \alpha}(s) = e_{\deg(t) \beta, \beta}(t)$ for some $s, t \in S \setminus \{0\}$ and $\alpha, \beta \in \Gamma$, then necessarily $s=t$ and $\alpha = \beta$, from which it follows that $\phi$ is injective.

(5) By Lemma~\ref{hgyfhyfhf1}(4), $S$ being an inverse semigroup is equivalent to $S \# \Gamma$ being such, which is, in turn, equivalent to $(S_\Gamma)_\varepsilon$ being an inverse semigroup, by (4). The desired conclusion now follows from Proposition~\ref{bvgfjsirdsw}(2).
\end{proof}

Proposition~\ref{hgyfhyfhf}(4) is analogous to a theorem about the smash product for rings. Specifically, if $\Gamma$ is a finite group, and $A$ is a $\Gamma$-graded ring, then $A\# \Gamma$ is isomorphic to a graded matrix ring \cite[7.2.1(2) Theorem]{nats1}.

\section{Graded Morita theory} \label{moritasection}
 
Morita theory for semigroups with local units was first explored in the 1990s by Talwar \cite{talwar1}. Many papers on the subject have appeared since then, culminating in the work of Lawson~\cite{lawsonmorita}, on semigroups with local units, and Funk/Lawson/Steinberg~\cite{funk}, on inverse semigroups. 

To obtain a Morita theory having a flavour similar to that in ring theory, Talwar worked with closed sets. Specifically, given a semigroup $S$ with local units, a left $S$-set $X$ is called \emph{closed} (or \emph{fixed}, or \emph{firm}, or \emph{ferm\'e}) if the $S$-map $S\otimes_S X\rightarrow X$, defined by $s\otimes x\mapsto sx$, is bijective (see \S \ref{grfdrferfsa}). Note that a closed left $S$-set is necessarily unital. The subcategory of $\Moc{S}$ consisting of closed left $S$-sets (or $S$-acts) is denoted by $\FAct{S}$ (where ``F" stands for ``fixed"). Talwar~\cite{talwar1} proved that for semigroups $S$ and $T$ with local units, $\FAct{S}$ is equivalent to $\FAct{T}$ if and only if there is a 6-tuple Morita context between $S$ and $T$. Other interesting statements equivalent to this one can be found in~\cite[Theorem~1.1]{lawsonmorita}. 

Graded Morita theory for rings was first studied by Gordon and Green~\cite{greengordon} in the setting of $\mathbb Z$-graded rings. For $\Gamma$-graded rings, with $\Gamma$ arbitrary, it was studied in \cite[\S 2]{hazi}. There it is shown that for unital $\Gamma$-graded rings $A$ and $B$, an equivalence of the categories of graded modules $\Gc{A}$ and $\Gc{B}$, which respects the shift functors, gives a 6-tuple Morita context between the rings, and, consequently, gives an equivalence of the categories of modules $\Moc{A}$ and $\Moc{B}$. This lifting of equivalence from the subcategories of graded modules to the categories of modules plays a crucial role in classifications of Leavitt path algebras~\cite{hazi}. 

Our next goal is to build a graded Morita theory for semigroups, in an analogous fashion. For a $\Gamma$-graded semigroup $S$, denote by $\FGrAct{S}$ the subcategory of $\Gc{S}$ consisting of closed $\Gamma$-graded left $S$-sets. Then $\FGrAct{S}$, as a subcategory of $\Gc{S}$, admits shift functors $\mathcal{T}_\alpha$ as in~(\ref{hfbvhjfkushke}). 

\begin{deff}\label{dhhchdhdfh}
Let $S$ and $T$ be $\Gamma$-graded semigroups.
\begin{enumerate}

\item A functor $\mathcal{F} :\Gc{S} \rightarrow \Gc{T}$ (or between their subcategories) is called a \emph{graded} functor if $\mathcal{F} \mathcal{T}_{\alpha} = \mathcal{T}_{\alpha} \mathcal{F}$, for any $\alpha \in \Gamma$. 

\smallskip

\item A graded functor $\mathcal{F} :\Gc{S} \rightarrow \Gc{T}$ (or between their subcategories)  is called a \emph{graded equivalence} if there exists a graded functor $\mathcal{F}' :\Gc{T} \rightarrow \Gc{S}$ such that $\mathcal{F}'\mathcal{F} \cong 1_{\Gc{S}}$ and $\mathcal{F} \mathcal{F}' \cong 1_{\Gc{T}}$. 

\smallskip

\item If there is a graded equivalence between $\FGrAct{S}$ and $\FGrAct{T}$, we say that $S$ and $T$ are \emph{graded Morita equivalent}. 
\end{enumerate}
\end{deff}

It is also possible to define \emph{graded Morita contexts} between graded semigroups, and use them to develop a theory for graded semigroups in a manner analogous to that in~\cite{talwar1} and~\cite{lawsonmorita}, but we do not pursue that line of inquiry here. Instead, in Theorem~\ref{hgbvgfhhf} we show that, analogously to the case of graded rings, if two $\Gamma$-graded semigroups $S$ and $T$ are graded Morita equivalent, then the equivalence can be lifted to the categories of closed left sets, implying that $S$ and $T$ are Morita equivalent. (See diagram below, where $\mathcal{U}$ denotes the forgetful functor.)
\begin{equation*}
\xymatrix{
\FAct{S} \ar@{.>}[r]^{?} & \FAct{T}\\
\FGrAct{S} \ar[r]^{\mathcal{F}} \ar[u]^{\mathcal{U}} & \FGrAct{T} \ar[u]_{\mathcal{U}}}
\end{equation*}
\smallskip
Note that this is not a priori obvious, since $\FAct{S}$ is ``bigger" than $\FGrAct{S}$.

Given a semigroup $S$, we denote by $\mathcal{C}(S)$  the \emph{Cauchy completion category of $S$}, whose objects are the idempotents of $S$, and whose morphisms are triples $(e,s,f) \in E(S) \times S \times E(S)$ such that $esf=s$. Here morphisms are composed using the rule $(e,s,f)(f,t,g) = (e,st,g)$. Lawson~\cite[Theorem 3.4]{lawsonmorita} showed that two semigroups with local units, $S$ and $T$, are Morita equivalent if and only if the corresponding Cauchy completion categories, $\mathcal{C}(S)$ and $\mathcal{C}(T)$, are equivalent. We will use this theorem to relate graded categories to non-graded ones in Theorem~\ref{hgbvgfhhf}. 

It should be noted that in~\cite{lawsonmorita} semigroups $S$ and left $S$-sets are not assumed to have zero elements. Since zero elements can be adjoined to any such semigroups and $S$-set, the results and proofs in~\cite{lawsonmorita} readily transfer to our setting, with one adjustment. While in the category of left $S$-sets with no zeros the coproduct of a collection of objects is their disjoint union, in our categories the coproduct is the $0$-disjoint union (since the disjoint union does not have a universal zero element), which we recall next. This observation is used liberally by Talwar in the original paper~\cite{talwar1} on Morita theory for semigroups.

Given a semigroup $S$ and a collection $\{X_i \mid i \in I\}$ of left $S$-sets, we denote by $\bigsqcup_{i \in I} X_i$ the \emph{$0$-disjoint union} of the relevant sets. That is, $\bigsqcup_{i \in I} X_i$ is the disjoint union of the sets $X_i \setminus \{0\}$, together with single zero element. It is immediate that $\bigsqcup_{i \in I} X_i$ is a left $S$-set whenever the $X_i$ are, upon identifying the zero of each $X_i$ with the common zero. If $S$ and the $X_i$ are $\Gamma$-graded, then $\bigsqcup_{i \in I} X_i$ inherits the grading from the $X_i$.

For a semigroup $S$, an object $X$ of $\FAct{S}$, respectively $\FGrAct{S}$, is said to be \emph{indecomposable} if $X$ is not isomorphic to $Y \sqcup Z$ (the $0$-disjoint union of $Y$ and $Z$), for any nonzero objects $Y$ and $Z$ of $\FAct{S}$, respectively $\FGrAct{S}$. Recall also that an object $X$ in a category is \emph{projective} if for every epimorphism $\phi: Y \to Z$ and morphism $\psi: X \to Z$ in the category, there is a morphism $\theta: X \to Y$ such that $\phi \theta = \psi$. To prove Theorem~\ref{hgbvgfhhf}, we require a description of the projective indecomposable objects in $\FGrAct{S}$. 

\begin{lemma} \label{projlemm}
Let $S$ be a $\Gamma$-graded semigroup with local units, and let $X$ be a closed $\Gamma$-graded left $S$-set. Then the following hold.
\begin{enumerate}[\upshape(1)]
\item $X$ is projective in $\FGrAct{S}$ if and only if it is projective as an object of $\FAct{S}$. 

\smallskip 

\item $X$ is indecomposable in $\FGrAct{S}$ if and only if it is indecomposable as an object of $\FAct{S}$. 
\end{enumerate}
\end{lemma}

\begin{proof}
(1) Suppose that $X$ is projective, viewed as an object of $\FAct{S}$, by forgetting the grading. Let $Y$ and $Z$ be objects in $\FGrAct{S}$, and let $\psi : X \to Z$ and $\phi : Y \to Z$ be morphisms (i.e., graded $S$-maps) in $\FGrAct{S}$, such that $\phi$ is surjective. We wish to find a morphism $\theta : X \to Y$ such that the following diagram commutes.
\begin{equation*}
\xymatrix{
& X \ar[d]^{\psi} \ar@{.>}[dl]_{\theta} \\
Y \ar[r]^\phi  & Z}
\end{equation*}
Since $X$ is projective in $\FAct{S}$, there is an $S$-map $\theta' : X \to Y$ such that $\phi \theta'=\psi$. Define a map $\theta:X \rightarrow Y$ by 
\[\theta(p) =
\left\{ \begin{array}{ll}
\theta'(p) & \text{if } \psi(p) \neq 0\\
0 & \text{otherwise } 
\end{array}\right..
\] 
Then clearly $\phi \theta=\psi$. We claim that $\theta$ is a graded map.

Let $\alpha \in \Gamma$, and let $p \in X_\alpha$ be such that $\psi(p) \neq 0$. Since $\psi$ is graded, we have $\psi(p) \in Z_\alpha \setminus \{0\}$. Since $\phi$ is graded, it could not be the case that $\deg(\theta(p)) = \beta$ for some $\beta \in \Gamma \setminus \{\alpha\}$, since otherwise we would have $\phi \theta(p) = \phi \theta(p)' \in Z_\beta \setminus \{0\}$, contradicting $\psi(p) \in Z_\alpha\setminus \{0\}$. Thus $\theta(p) \in Y_\alpha$, from which it follows that $\theta$ is graded, and hence is a morphism in $\FGrAct{S}$. Thus $X$ is projective in $\FGrAct{S}$. 

Conversely, suppose that $X$ is projective in $\FGrAct{S}$. Since $X$ is closed and $S$ has local units, for each $\alpha \in \Gamma$ and each $x\in X_\alpha \setminus \{0\}$ we can choose an idempotent $e_x \in E(S)$ such that $e_x x=x$. Define a function 
\begin{align*}
\phi: \bigsqcup_{\substack{x\in X_\alpha \setminus \{0\}\\ \alpha \in \Gamma}}Se_x(\alpha^{-1}) & \longrightarrow X\\
se_x\longmapsto sx.
\end{align*} 
(See~(\ref{butyestergdtgf}) for the notation.) Then $\phi$ is clearly a surjective $S$-map. Moreover, for all $s \in S$, $\alpha \in \Gamma$, and $x\in X_\alpha$ such that $sx \neq 0$, we have $\deg(se_x) = \deg(s)\alpha = \deg(sx)$, and so $\phi$ is a graded map. Since $X$ is projective in $\FGrAct{S}$, there is a graded $S$-map 
\begin{align*}
\psi: X & \longrightarrow \bigsqcup_{\substack{x\in X_\alpha \setminus \{0\}\\ \alpha \in \Gamma}}Se_x(\alpha^{-1})
\end{align*} 
such that $\phi\psi=1_X$. Now, by \cite[Lemma 3.1 (2)]{lawsonmorita} and \cite[Lemma 3.2 (1)]{lawsonmorita}, the $S$-set $\bigsqcup_{x,\alpha} Se_x(\alpha^{-1})$ is projective, viewed as an object of $\FAct{S}$. Hence, by \cite[Lemma 3.2 (3)]{lawsonmorita}, $X$ is also projective in $\FAct{S}$, upon viewing $\psi$ and $\phi$ as morphisms in $\FAct{S}$.

(2) Suppose that $X$ is indecomposable as an object of $\FAct{S}$. Then it could not be the case that $X = Y \sqcup Z$ for some nonzero objects $Y$ and $Z$ in $\FGrAct{S}$, since viewing $Y$ and $Z$ as objects of $\FAct{S}$, would give a decomposition of $X$ in $\FAct{S}$. Thus $X$ is indecomposable in $\FGrAct{S}$.

Conversely, suppose that $X$ is indecomposable in $\FGrAct{S}$, and that $X = Y \sqcup Z$ for some nonzero objects $Y$ and $Z$ in $\FAct{S}$. As closed left $S$-subsets of $X$, the sets $Y$ and $Z$ inherit a $\Gamma$-grading from $X$, and hence can be viewed as objects of $\FGrAct{S}$. This contradicts $X$ being indecomposable in $\FGrAct{S}$, and so it must also be indecomposable in $\FAct{S}$.
\end{proof}

\begin{lemma}\label{bjokrkd}
Let $S$ be a $\Gamma$-graded semigroup with local units. Then an object of $\FGrAct{S}$ is projective and indecomposable if and only if it is isomorphic to $Se(\alpha)$ for some $e \in E(S)$ and $\alpha \in \Gamma$.
\end{lemma}

\begin{proof}
First, note that each $Se(\alpha)$ in $\FGrAct{S}$ reduces to $Se$, when viewed as an object of $\FAct{S}$. Moreover, for any object of $\FGrAct{S}$ having $Se$ as the underlying left $S$-set, the grading is completely determined by the degree of $e$. Hence the objects of $\FGrAct{S}$ that reduce to ones of the form $Se$, when viewed as objects of $\FAct{S}$, are precisely those of the form $Se(\alpha)$.

Now, according to~\cite[Proposition 3.3]{lawsonmorita}, the projective indecomposable objects of $\FAct{S}$ are exactly the objects that are isomorphic to $Se$ for some $e \in E(S)$. The desired conclusion now follows from Lemma~\ref{projlemm}.
\end{proof}

\begin{thm}\label{hgbvgfhhf}
Let $S$ and $T$ be $\Gamma$-graded semigroups with local units. If $S$ and $T$ are graded Morita equivalent, then they are Morita equivalent. 
\end{thm}

\begin{proof}
Let $\mathcal{F} : \FGrAct{S} \rightarrow  \FGrAct{T}$ be a graded equivalence of categories. Also, let $\mathcal{EP}^{\gr}_S$, respectively $\mathcal{EP}^{\gr}_T$, denote the full subcategory of $\FGrAct{S}$, respectively $\FGrAct{T}$, consisting of indecomposable projective objects and morphisms between them. Since $\mathcal{F}$ preserves coproducts and projective objects, $\mathcal{F}$ induces a graded equivalence $\mathcal{F}: \mathcal{EP}^{\gr}_S \rightarrow \mathcal{EP}^{\gr}_T$. Finally, let $\mathcal{EP}_S$, respectively $\mathcal{EP}_T$, denote the full subcategory of $\FAct{S}$, respectively $\FAct{T}$, consisting of indecomposable projective objects and morphisms between them.

We will define a functor $\mathcal{H}: \mathcal{EP}_S \rightarrow \mathcal{EP}_T$, and show it to be faithful, full, and dense, implying that $\mathcal{EP}_S$ and $\mathcal{EP}_T$ are equivalent (see~\cite[Lemma 7.9.6]{bergman}). In the (short) proof of \cite[Theorem 3.4]{lawsonmorita} it is shown that $\mathcal{EP}_S$ is equivalent to the Cauchy completion $\mathcal{C}(S)$, for any semigroup $S$. Hence $\mathcal{EP}_S$ and $\mathcal{EP}_T$ being equivalent implies that so are $\mathcal{C}(S)$ and $\mathcal{C}(T)$. According to~\cite[Theorem 1.1]{lawsonmorita} this, in turn, implies that $S$ and $T$ are Morita equivalent. 

To define $\mathcal{H}$ we first need some additional information about $\mathcal{EP}_S$, $\mathcal{EP}_T$, $\mathcal{EP}^{\gr}_S$, and $\mathcal{EP}^{\gr}_T$. By~\cite[Proposition 3.3]{lawsonmorita}, the objects of $\mathcal{EP}_S$ are of the form $Se$, where $e \in E(S)$, and analogously for $\mathcal{EP}_T$. Likewise, by Lemma~\ref{bjokrkd}, the objects of $\mathcal{EP}^{\gr}_S$ are of the form $Se(\alpha)$, where $e \in E(S)$ and $\alpha \in \Gamma$, and analogously for $\mathcal{EP}^{\gr}_T$. Next, suppose that $\pi : Se \rightarrow Sf$ is a morphism in $\mathcal{EP}_S$, for some $e,f \in E(S)$, and let $a = \pi(e)$. Then $\pi(se) = s\pi(e) = sa$ for all $s \in S$. So from now on we will denote morphisms in $\mathcal{EP}_S$ as $\pi_a : Se \rightarrow Sf$, where $a = \pi_a(e)$. We also note that for any object $Se$ of $\mathcal{EP}_S$ and any $\alpha \in \Gamma$, we can view $Se(\alpha)$ as an object of $\mathcal{EP}^{\gr}_S$, where for each $se \in Se(\alpha) \setminus \{0\}$ we have $\deg(se) = \deg(s)\alpha^{-1}$. Since $\mathcal{F}: \mathcal{EP}^{\gr}_S \rightarrow \mathcal{EP}^{\gr}_T$ is a graded equivalence, for each $\alpha \in \Gamma$ we have \[\mathcal{T}_\alpha \mathcal{F} (Se(\varepsilon)) =  \mathcal{F} \mathcal{T}_\alpha (Se(\varepsilon)) =  \mathcal{F} (Se(\alpha)),\] and hence 
\begin{equation} \label{forgeteq}
\mathcal{U}\mathcal{F}(Se(\alpha)) = \mathcal{U}\mathcal{T}_\alpha \mathcal{F} (Se(\varepsilon)) = \mathcal{U} \mathcal{F} (Se(\varepsilon)),
\end{equation}
where $\mathcal{U}: \mathcal{EP}^{\gr}_S \rightarrow \mathcal{EP}_S$ denotes the forgetful functor. 

Now, for each object $Se$ of $\mathcal{EP}_S$ let $\mathcal{H}(Se) = \mathcal{U} \mathcal{F}(Se(\varepsilon))$. For each morphism $\pi_a : Se \rightarrow Sf$ in $\mathcal{EP}_S$, let 
\begin{equation*}
\alpha = 
\left\{ \begin{array}{ll}
\deg(a) & \text{if } a \neq 0  \\
\varepsilon & \text{otherwise } 
\end{array}\right.,
\end{equation*}
and let $\overline{\pi_a}: Se(\varepsilon) \rightarrow Sf(\alpha)$ be the the same function, viewed as a morphism in $\mathcal{EP}^{\gr}_S$. Note that for all $se \in Se(\varepsilon) \setminus \{0\}$ and $a \neq 0$ we have 
\[\deg(se) = \deg(s) = \deg(s)\alpha\alpha^{-1} = \deg(sa) = \deg(\overline{\pi_a}(se)),\]
from which it follows that $\overline{\pi_a}$ is indeed a graded $S$-map. In view of (\ref{forgeteq}), we can define $\mathcal{H}(\pi_a)=\mathcal{U}\mathcal{F}(\overline{\pi_a})$--see diagram below. 
\begin{equation*}
\xymatrix{
\mathcal{EP}_S \ar@{.>}[rr]^{\mathcal{H}} & & \mathcal{EP}_T \\
\mathcal{EP}^{\gr}_S \ar[rr]^-{\mathcal{F}} \ar[u]^{\mathcal{U}} & & \mathcal{EP}^{\gr}_T \ar[u]_{\mathcal{U}}}
\end{equation*}

To show that $\mathcal{H}$ is a functor it suffices to take two composable morphisms, $\pi_a: Se\rightarrow Sf$ and $\pi_b:Sf \rightarrow Sg$, in $\mathcal{EP}_S$, and prove that $\mathcal{H}(\pi_b\pi_a) = \mathcal{H}(\pi_b)\mathcal{H}(\pi_a)$. Viewing these as morphisms in $\mathcal{EP}^{\gr}_S$, we have $\overline{\pi_a}: Se(\varepsilon)\rightarrow Sf(\alpha)$ and $\overline{\pi_b}: Sf(\varepsilon)\rightarrow Sg(\beta)$, where $\alpha = \deg(a)$ (or $\alpha=\varepsilon$ if $a = 0$), and analogously for $\beta$. Writing $\pi_b \pi_a= \pi_{c}$ for some $c \in Sg$ (with $c \in S_{\alpha}S_{\beta}$), we also have $\overline{\pi_{c}}: Se(\varepsilon)\rightarrow Sg(\gamma)$, where $\gamma = \alpha\beta$ if $ab \neq 0$, and $\gamma=\varepsilon$ otherwise. Now, let $\overline{\pi_b}': Sf(\alpha)\rightarrow Sg(\gamma)$ be the morphism in $\mathcal{EP}^{\gr}_S$, which agrees with $\pi_b:Sf \rightarrow Sg$ as a function. Then $\overline{\pi_{c}} = (\overline{\pi_b}')(\overline{\pi_a})$, and by an argument similar to that for (\ref{forgeteq}), we have $\mathcal{U}\mathcal{F}(\overline{\pi_b}') = \mathcal{U}\mathcal{F}(\overline{\pi_b})$. Hence
\begin{multline}
\mathcal{H}(\pi_b \pi_a)=\mathcal{H}(\pi_{c})=\mathcal{U}\mathcal{F}(\overline{\pi_{c}})= \mathcal{U}\mathcal{F}((\overline{\pi_b}')(\overline{\pi_a})) \\ =\mathcal{U}\mathcal{F}(\overline{\pi_b}') \mathcal{U}\mathcal{F}(\overline{\pi_a})= \mathcal{U}\mathcal{F}(\overline{\pi_b})\mathcal{U}\mathcal{F}(\overline{\pi_a}) = \mathcal{H}(\pi_b)\mathcal{H}(\pi_a), \notag
\end{multline}
as desired.

To show that $\mathcal{H}$ is faithful, let $Se$ and $Sf$ be objects of $\mathcal{EP}_S$, and let $\pi_a, \pi_b:Se \rightarrow Sf$ be distinct morphisms. Then $\overline{\pi_a}: Se(\varepsilon) \rightarrow Sf(\alpha)$ and $\overline{\pi_b}: Se(\varepsilon) \rightarrow Sf(\beta)$ must be distinct as well (where $\alpha, \beta \in \Gamma$ are the appropriate degrees). Moreover, since $\mathcal{F}$ is faithful, we have $\mathcal{F}(\overline{\pi_a}) \not = \mathcal{F}(\overline{\pi_b})$. Now $\mathcal{F}(\overline{\pi_a})=\overline{\pi_{a'}}: Te' (\gamma) \rightarrow Tf' (\delta)$ and $\mathcal{F}(\overline{\pi_b})=\overline{\pi_{b'}}: Te' (\gamma) \rightarrow Tf' (\delta')$ in $\mathcal{EP}^{\gr}_T$, for some $a',b'\in T$ and $\gamma, \delta, \delta' \in \Gamma$. Since $\mathcal{F}(\overline{\pi_a}) \not = \mathcal{F}(\overline{\pi_b})$, we necessarily have $a'\not = b'$. Thus $\mathcal{U}(\overline{\pi_{a'}}) \not = \mathcal{U}(\overline{\pi_{b'}})$, and so $\mathcal{H}(\pi_a) \not = \mathcal{H}(\pi_b)$. 

To show that $\mathcal{H}$ is full, let $Se$ and $Sf$ be objects of $\mathcal{EP}_S$, and let $\phi : \mathcal{H}(Se) \rightarrow \mathcal{H}(Sf)$ be a morphism. Write $\mathcal{H}(Se) = Tg$ and $\mathcal{H}(Sf) = Th$ for some $g,h \in E(T)$. Then we can find $\delta \in \Gamma$ and $\overline{\pi_a} : Tg(\varepsilon) \rightarrow Th(\delta)$ such that $\mathcal{U}(\overline{\pi_a}) = \phi$. Since $\mathcal{F}$ is full, there exists a morphism $\overline{\psi}: Se(\alpha) \rightarrow Sf(\beta)$ in $\mathcal{EP}^{\gr}_S$ such that $\mathcal{F}(\overline{\psi}) = \overline{\pi_a}$, and hence $\mathcal{U}\mathcal{F}(\overline{\psi})=\phi$. Viewing $\overline{\psi}$ as a function $\psi : Se \rightarrow Sf$, and hence morphism in $\mathcal{EP}_S$, we have $\mathcal{H}(\psi) = \phi$.

Finally, to show that $\mathcal{H}$ is dense, let $Te$ be an object in $\mathcal{EP}_T$. Then $Te = \mathcal{U}(Te(\varepsilon))$, where $Te(\varepsilon)$ is in $\mathcal{EP}^{\gr}_T$. Since $\mathcal{F}$ is dense, $Te(\varepsilon)$ is isomorphic to $\mathcal{F}(Sf(\alpha))$ for some object $Sf(\alpha)$ in $\mathcal{EP}^{\gr}_S$. Hence \[\mathcal{H}(Sf) = \mathcal{U}\mathcal{F}(Sf(\varepsilon)) = \mathcal{U}\mathcal{F}(Sf(\alpha))\] is isomorphic to $Te$, as desired.
\end{proof}
 
In most of this paper we work with the categories $\Moc{S}$ and $\Gc{S}$, rather than $\FAct{S}$ and $\FGrAct{S}$. We conclude this section by observing that if the semigroup $S$ happens to have \emph{common} local units, then these categories coincide, respectively.

\begin{prop}
Let $S$ be a $\Gamma$-graded semigroup with common local units. Then $\Moc{S} = \FAct{S}$ and $\Gc{S} = \FGrAct{S}$.
\end{prop}

\begin{proof}
For both claims it suffices to show that every unital left $S$-set $X$ is closed. To conclude this, it is enough to show that the function $S \otimes_S X \rightarrow X$, defined by $s\otimes x \mapsto sx$, is injective for each $X$. So let $s\otimes x, t\otimes y \in S \otimes_S X$, and suppose that $sx=ty$. Since $S$ has common local units, there exists $u \in E(S)$ such that $us=s$ and $ut=t$. Then 
\[s\otimes x = us\otimes x = u\otimes sx = u\otimes ty = ut \otimes y = t\otimes y,\]
giving the desired conclusion.
\end{proof}

\section{Graded inverse semigroups}
 
\subsection{Graded Vagner-Preston theorem}\label{vagnersect} By Cayley's theorem, any group can be embedded in a group of symmetries of a set.  Similarly, by the Vagner-Preston theorem (see \cite[Theorem 5.1.7]{howie} or \cite[Theorem~1.5.1]{Lawson}), any inverse semigroup can be embedded in an inverse semigroup of partial symmetries of a set. Next we recall the relevant terminology, examine gradings on inverse semigroups of this sort, and give a graded version of the Vagner-Preston theorem.

Let $X$ be a nonempty set. For any $A,B \subseteq X$, a bijective function $\phi: A\rightarrow B$ is called a \emph{partial symmetry} of $X$. Here we let $\dom(\phi):=A$ and $\im(\phi): = B$. We also denote the set of all partial symmetries of $X$ by $\mathcal{I}(X)$. Then $\mathcal{I}(X)$ is an inverse semigroup with respect to composition of relations, known as the \emph{symmetric inverse monoid}. Specifically, for all $\phi,\psi \in \mathcal{I}(X)$, $\phi \psi$ is taken to be the composite of $\phi$ and $\psi$ as functions, restricted to the domain $\psi^{-1}(\im(\psi)\cap \dom(\phi))$. The empty function plays the role of the zero element in $\mathcal{I}(X)$.  

We denote the cardinality of a set $X$ by $|X|$.
 
\begin{prop} \label{symminverseprop}
Let $X$ be a set such that $|X| \geq 3$. Then any grading on $\mathcal{I}(X)$ is trivial.
\end{prop}

\begin{proof}
Let $\Gamma$ be a group, let $\chi : \mathcal{I}(X)\setminus \{0\} \rightarrow \Gamma$ be a grading, and let $\phi \in \mathcal{I}(X) \setminus \{0\}$. We will show that $\deg(\phi) = \chi(\phi)=\varepsilon$.

First, suppose that $\phi(x) = x$ for some $x \in \dom(\phi)$, and let $\psi \in \mathcal{I}(X)$ be the unique element with $\dom(\psi) = \{x\} = \im(\psi)$. Then, $\psi$ is an idempotent and so $\deg(\psi) = \varepsilon$. Since $\psi\phi = \psi$, it follows that
\[\varepsilon = \deg(\psi) = \deg(\psi\phi) = \deg(\psi)\deg(\phi) = \deg(\phi).\]

Now, take $\phi \in \mathcal{I}(X) \setminus \{0 \}$ to be arbitrary, and let $x \in \dom(\phi)$. Then, by our hypothesis on $X$, we can find some $y \in X \setminus \{x, \phi(x)\}$. Let $\psi \in \mathcal{I}(X)$ be defined by $\dom(\psi) = \{\phi(x), y\}$, $\im(\psi) = \{x,y\}$, $\psi(\phi(x)) = x$, and $\psi(y) = y$. Then, by the previous paragraph, we have $\deg(\psi) = \varepsilon = \deg (\psi\phi)$, and therefore
\[\varepsilon = \deg(\psi\phi) = \deg(\psi)\deg(\phi) = \deg(\phi),\]
as desired.
\end{proof}

If $|X| = 1$, then $\mathcal{I}(X) \setminus \{0\}$ consists of one idempotent element, and so any grading on $\mathcal{I}(X)$ is trivial. However, if $|X| = 2$, then $\mathcal{I}(X)$ has a nontrivial grading, as the next example shows.

\begin{example}
Let $X = \{x,y\}$, and write 
\[\mathcal{I}(X) = \{0, 1, \tau, \theta_{xx}, \theta_{xy}, \theta_{yx}, \theta_{yy}\},\]
where $\tau$ denotes the one nontrivial permutation of $X$, and $\theta_{ij}$ is the only element of $\mathcal{I}(X)$ such that $\dom(\theta_{ij}) = \{j\}$ and $\im(\theta_{ij}) = \{i\}$ ($i,j \in X$). Define $\phi : \mathcal{I}(X) \setminus \{0\} \to \Z_2$ by 
\[\phi(1) = \phi(\theta_{xx}) = \phi(\theta_{yy}) = 0\] 
and 
\[\phi(\tau) = \phi(\theta_{xy}) = \phi(\theta_{yx}) = 1.\] 
Then it is easy to check that $\phi$ is a grading.
\end{example}

In view of Proposition~\ref{symminverseprop}, a graded version of the Vagner-Preston theorem requires a graded analogue of $\mathcal{I}(X)$, which we construct next. Let $X$ be a nonempty $\Gamma$-graded set. For each nonempty $A \subseteq X$ and $\alpha \in \Gamma$, we set $A_\alpha:= A\cap X_\alpha$. For each $\alpha \in \Gamma$ let 
\begin{equation}\label{mothercomp}
\mathcal{I}(X)_\alpha:= \big \{\phi \in \mathcal{I}(X) \, \big \vert  \, \phi(\dom(\phi)_\beta) \subseteq X_{\alpha\beta} \text{ for all } \beta \in \Gamma \big \}, 
\end{equation}
and set 
\begin{equation}\label{mothersemigroup}
\mathcal{I}^{\gr}(X):= \bigcup _{\alpha \in \Gamma} \mathcal{I}(X)_\alpha. 
\end{equation}

Next we show that $\mathcal{I}^{\gr}(X)$ is a graded inverse semigroup. Among other things, it is a useful platform for exploring the differences between gradings, strong gradings (Definition~\ref{strgrdef}), and locally strong gradings (Definition~\ref{locallystrgrdef}).

\begin{prop}\label{meme}
Let $X$ be a nonempty $\Gamma$-graded set. Then the following hold. 
\begin{enumerate}[\upshape(1)]
\item $\mathcal{I}^{\gr}(X)$ is a $\Gamma$-graded inverse semigroup. 

\smallskip

\item $\mathcal{I}^{\gr}(X)$ is strongly $\Gamma$-graded if and only if $|X_{\alpha}|=|X_{\beta}|$ for all $\alpha, \beta \in \Gamma$.

\smallskip

\item $\mathcal{I}^{\gr}(X)$ is locally strongly $\Gamma$-graded if and only if $X_\alpha\not = \emptyset$ for all $\alpha \in \Gamma$.
\end{enumerate}
\end{prop}

\begin{proof}
(1) By (\ref{mothercomp}) and (\ref{mothersemigroup}), we have $\mathcal{I}^{\gr}(X) \subseteq \mathcal{I}(X)$, $0\in \mathcal{I}(X)_\alpha$ for every $\alpha\in \Gamma$, and $\mathcal{I}(X)_\alpha \cap \mathcal{I}(X)_\beta = \{0\}$ for all distinct $\alpha, \beta\in \Gamma$. Thus, to conclude that $\mathcal{I}^{\gr}(X)$ is a $\Gamma$-graded subsemigroup of $\mathcal{I}(X)$ it suffices to show that for all $\alpha, \beta \in \Gamma$,  $\phi\in\mathcal{I}^{\gr}(X)_\alpha$, and $\psi\in \mathcal{I}^{\gr}(X)_\beta$, such that $\psi\phi \not = 0$, we have $\psi\phi \in \mathcal{I}^{\gr}(X)_{\beta\alpha}$. Taking $\phi$ and $\psi$ as above, let $\gamma \in \Gamma$ and $x \in \dom (\psi\phi)_\gamma$. Then, by (\ref{mothercomp}), $\phi(x) \in (\im(\phi)\cap \dom(\psi))_{\alpha\gamma}$, and, consequently, $\psi(\phi(x)) \in \im(\psi\phi) _{\beta \alpha \gamma}$, showing that $\psi\phi \in \mathcal{I}^{\gr}(X)_{\beta\alpha}$.

Next, let $\alpha, \beta \in \Gamma$, $\phi\in\mathcal{I}^{\gr}(X)_\alpha$, and $x \in \dom (\phi^{-1})_\beta = \im(\phi)_\beta$, and write $\phi(y)=x$ for some $y \in \dom(\phi)$. Then \[\phi^{-1}(x) = y \in \dom(\phi)_{\alpha^{-1}\beta} = \im(\phi^{-1})_{\alpha^{-1}\beta},\] and so $\phi^{-1} \in \mathcal{I}^{\gr}(X)_{\alpha^{-1}} \subseteq \mathcal{I}^{\gr}(X)$. It follows that $\mathcal{I}^{\gr}(X)$ is an inverse subsemigroup of $\mathcal{I}(X)$. 

(2) Suppose that $|X_{\delta_1}|=|X_{\delta_2}|$ for all $\delta_1, \delta_2 \in \Gamma$, and let $\alpha, \beta \in \Gamma$ and $\phi \in \mathcal{I}(X)_{\alpha\beta}$. By hypothesis, for each $\gamma \in \Gamma$, we can find $Y_{\beta\gamma} \subseteq X_{\beta\gamma}$ such that $|Y_{\beta\gamma}| = |\dom(\phi)_\gamma|$. Now let $\psi \in \mathcal{I}(X)$ be such that $\dom(\psi)_\gamma = \dom(\phi)_\gamma$ and $\psi(\dom(\psi)_\gamma) = Y_{\beta\gamma}$ for each $\gamma \in \Gamma$, and let $\rho \in \mathcal{I}(X)$ be such that $\dom(\rho) = \bigcup_{\delta \in \Gamma}Y_{\beta\delta} = \im(\psi)$ and $\rho\psi = \phi$. Then clearly, $\psi \in \mathcal{I}(X)_{\beta}$ and $\rho \in \mathcal{I}(X)_{\alpha}$. Thus $\phi \in \mathcal{I}(X)_{\alpha}\mathcal{I}(X)_{\beta}$, which implies that $\mathcal I^{\gr}(X)$ is strongly graded.

Conversely, suppose that $|X_{\alpha}|<|X_{\beta}|$ for some $\alpha, \beta \in \Gamma$. Let $\phi \in \mathcal{I}(X)$ be such that $\dom(\phi)=X_{\beta}=\im(\phi)$. Then $\phi \in \mathcal{I}(X)_{\varepsilon}$. Suppose that $\phi = \rho\psi$ for some $\rho \in \mathcal{I}(X)_{\beta\alpha^{-1}}$ and $\psi \in \mathcal{I}(X)_{\alpha\beta^{-1}}$. Then necessarily $X_{\beta}\subseteq \dom(\psi)$ and $\psi(X_{\beta})\subseteq X_{\alpha}$, which contradicts $|X_{\alpha}|<|X_{\beta}|$. Hence $\phi \notin\mathcal{I}(X)_{\beta\alpha^{-1}}\mathcal{I}(X)_{\alpha\beta^{-1}}$, and so $\mathcal{I}^{\gr}(X)$ is not strongly graded.

(3) Suppose that $X_\gamma \not = \emptyset$ for all $\gamma \in \Gamma$, and let $\alpha, \beta \in \Gamma$ and $\phi \in \mathcal{I}(X)_{\alpha\beta} \setminus \{0\}$. Since $\phi \neq 0$, we can find some $x \in \dom(\phi)$, where, say, $x \in X_{\gamma}$ ($\gamma \in \Gamma$). By hypothesis, there exists $y \in X_{\beta\gamma}$. Now let $\rho \in \mathcal{I}(X)$ be the map determined by $\dom(\rho) = \{x\}$ and $\im(\rho) = \{y\}$. Then clearly $\rho \in \mathcal{I}(X)_{\beta} \setminus \{0\}$, and so \[\phi (\rho^{-1}\rho) = (\phi \rho^{-1})\rho \in \mathcal{I}(X)_{\alpha} \mathcal{I}(X)_{\beta} \setminus \{0\}.\] Since $\phi \rho^{-1}\rho \leq \phi$, we conclude that $\mathcal{I}^{\gr}(X)$ is locally strongly $\Gamma$-graded.

Conversely, suppose that $X_{\alpha} = \emptyset$ for some $\alpha \in \Gamma$. Let $\beta \in \Gamma$ be such that $X_{\beta} \neq \emptyset$, and let $\phi \in E(\mathcal{I}^{\gr}(X))$ be such that $\dom(\phi)=X_{\beta}=\im(\phi)$. Seeking a contradiction to Proposition~\ref{sussa}, suppose that there exists $\rho \in E(\mathcal{I}^{\gr}(X))_{\alpha^{-1}\beta} \setminus \{0\}$, such that $\rho \leq \phi$. Then $\rho = \psi\psi^{-1}$ for some $\psi \in \mathcal{I}^{\gr}(X)_{\alpha^{-1}\beta}$, and $\dom(\rho)\subseteq \dom(\phi)=X_{\beta}$. Since $\psi^{-1} \in \mathcal{I}^{\gr}(X)_{\alpha\beta^{-1}}$, we have $\psi^{-1}(X_{\beta})  \subseteq X_{\alpha} = \emptyset$. It follows that $\rho = 0$, producing the desired contradiction. Hence, if $X_{\alpha} = \emptyset$ for some $\alpha \in \Gamma$, then $\mathcal{I}^{\gr}(X)$ cannot be locally strongly $\Gamma$-graded.
\end{proof}
 
We are now ready for our graded version of the Vagner-Preston theorem. The construction is fundamentally the same as in the original theorem, but with some key differences.

\begin{prop}[Graded Vagner-Preston Theorem] \label{grvagnerpreston}
Let $S$ be a $\Gamma$-graded inverse semigroup. Then there is a graded injective homomorphism $\psi :S \rightarrow \mathcal{I}^{\gr}(X)$ for some $\Gamma$-graded set $X$. 
\end{prop}

\begin{proof}
Let $X=S \backslash \{0\}$. Then $X$ is a $\Gamma$-graded set, with respect to the grading induced by that on $S$. For each $\alpha \in \Gamma$ and $s \in S_\alpha \backslash \{0\}$ define a function
\begin{align*}
\theta_s: s^{-1}s S \backslash \{0\}  &\longrightarrow ss^{-1} S \backslash \{0\}\\
x &\longmapsto sx. 
\end{align*}
Note that if $s^{-1}sx \neq 0$ for some $x \in S$, then \[\theta_s(s^{-1}sx) = ss^{-1}sx = sx \neq 0,\] from which it follows that $\theta_s$ is well-defined. Also, if $\theta_s(s^{-1}sx) = \theta_s(s^{-1}sy)$ for some $x,y\in S$, then \[sx = ss^{-1}sx = ss^{-1}sy = sy,\] and so $s^{-1}sx = s^{-1}sy$. Thus $\theta_s$ is injective. Since \[\theta_s(s^{-1}ss^{-1}x) = \theta_s(s^{-1}x) = ss^{-1}x\] for all $x \in S$, $\theta_s$ is bijective. Moreover, for all $\beta \in \Gamma$ and appropriate $x \in S_\beta$ we have $\theta_s(s^{-1}sx)= sx \in S_{\alpha\beta}$, from which it follows that $\theta_s \in \mathcal{I}(X)_\alpha$. 

We can therefore define a map 
\begin{align*} 
\psi: S & \longrightarrow  \mathcal{I}^{\gr}(X)\\
s & \longmapsto \theta_s,
\end{align*}
where $\theta_0$ is understood to be the zero element of $\mathcal{I}^{\gr}(X)$. The last computation in the previous paragraph implies that $\psi(S_\alpha) \subseteq \mathcal{I}(X)_\alpha$ for all $\alpha \in \Gamma$, and so $\psi$ is a graded map. To show that $\theta$ is injective, suppose that $\theta_s = \theta_t$ for some $s,t \in S$. Then $s^{-1}s S = t^{-1}t S$, and so according to \cite[Lemma 5.1.6(1)]{howie}, $s^{-1}s = t^{-1}t$. Hence $s = \theta_s(s^{-1}s) = \theta_t(t^{-1}t) = t$, and so $\psi$ is injective.  

It remains to show that $\psi$ is a homomorphism. Clearly, $\theta_s \theta_0 = \theta_0 =\theta_0\theta_s$ for any $s \in S$. Showing that $\theta_s \theta_t = \theta_{st}$ for $s,t \in S\setminus \{0\}$ can be accomplished using exactly the same, somewhat lengthy, argument as in the textbook proof of the Vagner-Preston theorem (see, e.g., \cite[Theorem 5.1.7]{howie} or \cite[Theorem~1.5.1]{Lawson}), so we will not repeat it here.
\end{proof}

\subsection{Strongly graded inverse semigroups}\label{iuhfgbdftie}

Our next goal is to provide an analogue for inverse semigroups of Dade's theorem~\cite[Theorem~2.8]{dade} (see also~\cite[\S 1.5]{hazi} and \cite[Theorem~3.1.1]{grrings}), which describes strongly graded rings using equivalences of appropriate categories. We begin with several lemmas.

\begin{lemma} \label{funct-lemm}
Let $S$ be a $\Gamma$-graded semigroup with local units. Then the following hold.
\begin{enumerate}[\upshape(1)]

\item  Let $(-)_\varepsilon : \Gc{S} \rightarrow \Moc{S_\varepsilon}$ be the mapping defined by 
\begin{align*}
X & \longmapsto X_\varepsilon \\
\phi & \longmapsto \phi|_{X_\varepsilon}
\end{align*}
for all objects $X$ and morphisms $\phi : X \rightarrow Y$ in $\Gc{S}$, where $\phi|_{X_\varepsilon}$ denotes the restriction of $\phi$ to $X_\varepsilon$. Then $(-)_\varepsilon$ is a functor, to which we refer as the \emph{restriction functor}.

\smallskip 

\item Let $S \otimes_{S_\varepsilon} - : \Moc{S_\varepsilon} \rightarrow \Gc{S}$ be the mapping defined by 
\begin{align*}
X & \longmapsto S \otimes_{S_\varepsilon} X \\
\phi & \longmapsto 1_S  \otimes \phi
\end{align*}
for all objects $X$ and morphisms $\phi : X \rightarrow Y$ in $\Moc{S_\varepsilon}$, where $(1_S  \otimes \phi) (s\otimes x) = s \otimes \phi(x)$ for all $s\otimes x \in  S \otimes_{S_\varepsilon} X$. Here the $\Gamma$-grading on $S \otimes_{S_\varepsilon} X$ is as in (\ref{tensorgrade}), with $X$ given the trivial grading ($X=X_\varepsilon$). Then $S \otimes_{S_\varepsilon} -$ is a functor, to which we refer as the \emph{induction functor}.
\end{enumerate}
\end{lemma}

\begin{proof}
(1) Since $S_\varepsilon$ is a subsemigroup of $S$, any left $S$-set is automatically a left $S_\varepsilon$-set. Now let $X$ be any object of $\Gc{S}$. Then, clearly $S_\varepsilon X_\varepsilon \subseteq X_\varepsilon$. Moreover, since $X$ is a unital left $S$-set, and since $S$ has local units, for each $x \in X$ there exists $u \in E(S) \subseteq S_\varepsilon$ such that $ux = x$. It follows that $X_\varepsilon$ is a unital left $S_\varepsilon$-set, and therefore an object in $\Moc{S_\varepsilon}$.

Next, let $\phi : X \to Y$ be a morphism in $\Gc{S}$, i.e., a function that commutes with the $S$-action and respects the $\Gamma$-grading. Then $\phi(X_\varepsilon) \subseteq Y_\varepsilon$, and so restricting $\phi$ to $X_\varepsilon$ gives a function $\phi|_{X_\varepsilon}: X_\varepsilon \to Y_\varepsilon$. Also, for all $s \in S_\varepsilon$ and $x \in X_\varepsilon$ we have 
\[\phi|_{X_\varepsilon}(sx) = \phi(sx) = s\phi(x) = s\phi|_{X_\varepsilon}(x),\] 
and so $\phi|_{X_\varepsilon}: X_\varepsilon \to Y_\varepsilon$ is a morphism in $\Moc{S_\varepsilon}$. 

Finally, it is immediate that $(-)_\varepsilon$ preserves all identity morphisms, and that $(\phi \circ \psi)|_{X_\varepsilon} = \phi |_{X_\varepsilon} \circ \psi|_{X_\varepsilon}$ for all composable morphisms $\phi$ and $\psi$ in $\Gc{S}$. Therefore $(-)_\varepsilon$ is a functor.

(2) As discussed in \S \ref{grfdrferfsa}, $S \otimes_{S_\varepsilon} X$ is a $\Gamma$-graded left $S$-set for each object $X$ in $\Moc{S_\varepsilon}$. Since $S$ has local units, it is easy to see that $S \otimes_{S_\varepsilon} X$ is a unital left $S$-set, and therefore an object of $\Gc{S}$.

Next, let $\phi : X \rightarrow Y$ be a morphism in $\Moc{S_\varepsilon}$. Then the usual considerations about tensors (see \S \ref{grfdrferfsa}), along with the fact that $\phi$ is an $S_\varepsilon$-map, imply that $1_S  \otimes \phi : S \otimes_{S_\varepsilon} X \rightarrow S \otimes_{S_\varepsilon} Y$ is well-defined. Also, for all $s\otimes x \in S \otimes_{S_\varepsilon} X$ and $t \in S$ we have
\[(1_S  \otimes \phi)(t(s\otimes x)) = (1_S  \otimes \phi)(ts\otimes x) = ts \otimes \phi(x)  = t(s \otimes \phi(x)) = t (1_S  \otimes \phi)(s\otimes x),\]
and 
\[\deg(s\otimes x) = \deg(s)\deg(x) = \deg(s) = \deg(s \otimes \phi(x)) = \deg((1_S  \otimes \phi)(s\otimes x)).\]
Thus $1_S  \otimes \phi$ is a morphism in $\Gc{S}$.

Finally, it is easy to see that $S \otimes_{S_\varepsilon} -$ preserves all identity morphisms, and that \[1_S  \otimes (\phi \circ \psi) = (1_S  \otimes \phi) \circ (1_S  \otimes\psi)\] for all composable morphisms $\phi$ and $\psi$ in $\Moc{S_\varepsilon}$. Therefore $S \otimes_{S_\varepsilon} -$ is a functor.
\end{proof}

\begin{lemma} \label{sixmonthd}
Let $S$ be a $\Gamma$-graded inverse semigroup, $T$ a subsemigroup of $S$ such that $S_\varepsilon \subseteq T$, $Y$ a $\Gamma$-graded left $S$-set, and $X \subseteq Y_\varepsilon$. Then the function $\phi : S \otimes_T X \rightarrow SX$, defined by  $s\otimes x \mapsto sx$, is a graded bijective $S$-map.
\end{lemma}

\begin{proof}
It is easy to see that $\phi$ is well-defined, surjective, graded (using (\ref{tensorgrade})), and an $S$-map. To show that it is injective, suppose that $\phi(s\otimes x)=\phi(t\otimes y)$ for some $s\otimes x, t\otimes y \in S \otimes_T X$. Then $sx=ty$, and so $ty=ss^{-1}ty$. Next, note that if $sx = 0$, then $s\otimes x = s \otimes s^{-1}sx = 0 \otimes 0$, and likewise $t\otimes y = 0\otimes 0$. So we may assume that $sx = ty$ is nonzero. Since $X \subseteq Y_\varepsilon$, and $\deg(sx)=\deg(ty)$, we have $\deg(s) = \deg(t)$. In particular, $s^{-1}tt^{-1}t \in S_\varepsilon \subseteq T$. Thus, using the fact that idempotents commute in any inverse semigroup \cite[Theorem 5.1.1]{howie}, we have
\begin{multline*}
s\otimes x= ss^{-1}s\otimes x= s\otimes s^{-1}s x=s\otimes s^{-1}ty= s\otimes s^{-1}tt^{-1}ty=s s^{-1}tt^{-1}t\otimes y \\ =tt^{-1}ss^{-1}t\otimes y=t\otimes t^{-1}ss^{-1}ty=t\otimes t^{-1}ty=tt^{-1}t\otimes y=t\otimes y.
\end{multline*}
Hence $\phi$ is injective.
\end{proof}

\begin{lemma} \label{natranstlemm}
Let $S$ be a $\Gamma$-graded inverse semigroup, and let $(-)_\varepsilon$ and $S \otimes_{S_\varepsilon} - $ be as in Lemma \ref{funct-lemm}. For each object $Y$ in $\Moc{S_\varepsilon}$ define a function
\begin{align}\label{hgy4nd1}
\mu_Y : (S\otimes_{S_\varepsilon} Y)_\varepsilon  & \longrightarrow Y\\
s\otimes y &\longmapsto sy,  \notag
\end{align}
and for each object $X$ in $\Gc{S}$ define a function 
\begin{align}\label{hgy4nd}
\nu_X : S \otimes_{S_\varepsilon} X_\varepsilon & \longrightarrow X,\\
s\otimes x &\longmapsto sx  \notag.
\end{align}
Then $\mu : (S \otimes_{S_\varepsilon} -)_{\varepsilon} \rightarrow 1_{\Moc{S_\varepsilon}}$ and $\nu : S \otimes_{S_\varepsilon} ( -)_{\varepsilon}\rightarrow 1_{\Gc{S}}$ are natural transformations, and each $\mu_Y$ is an isomorphism in $\Moc{S_\varepsilon}$.
\end{lemma}

\begin{proof}
Since, as stipulated in Lemma~\ref{funct-lemm}, each object $Y$ of $\Moc{S_\varepsilon}$ is given the trivial grading in the construction $S\otimes_{S_\varepsilon} Y$, we see from (\ref{tensorgrade}) that  $(S\otimes_{S_\varepsilon} Y)_\varepsilon = S_\varepsilon \otimes_{S_\varepsilon} Y$. Moreover, since $Y$ is a unital left $S_\varepsilon$-set, we have $S_\varepsilon Y=Y$. Thus, by Lemma \ref{sixmonthd} (taking $X=Y$ and $T=S_\varepsilon$), $\mu_Y$ is a bijective $S_\varepsilon$-map for each $Y$, and therefore is an isomorphism in $\Moc{S_\varepsilon}$. Likewise, for each object $X$ in $\Gc{S}$, mapping $s\otimes x \mapsto sx$ gives a  graded bijective $S$-map $S \otimes_{S_\varepsilon} X_\varepsilon \rightarrow SX_\varepsilon \subseteq X$, by Lemma \ref{sixmonthd} (taking $T=S_\varepsilon$). Thus, each $\nu_X$ is a morphism in $\Gc{S}$.

Now let $\phi : X \to Y$ be an arbitrary morphism in $\Gc{S}$. Then for all $s\otimes x \in S \otimes_{S_\varepsilon} X_\varepsilon$, we have
\[\phi\nu_X(s\otimes x) = \phi(sx) = s\phi(x) = \nu_Y(s\otimes \phi(x)) = \nu_Y(1_S\otimes \phi |_{X_\varepsilon})(s\otimes x),\]
from which it follows that the diagram 
\[\xymatrix{
S \otimes_{S_\varepsilon} X_\varepsilon \ar[rr]^-{1_S\otimes \phi |_{X_\varepsilon}} \ar[d]_{\nu_X} & & S \otimes_{S_\varepsilon} Y_\varepsilon \ar[d]^{\nu_Y} \\
X \ar[rr]^-{\phi} & & Y}\]
commutes. Hence $\nu : S \otimes_{S_\varepsilon} ( -)_{\varepsilon}\rightarrow 1_{\Gc{S}}$ is a natural transformation. A nearly identical computation shows that $\mu : (S \otimes_{S_\varepsilon} -)_{\varepsilon} \rightarrow 1_{\Moc{S_\varepsilon}}$ is a natural transformation.
\end{proof}

We are ready for the main result of this section, which amounts to saying that $S$ is strongly graded if and only if each $\nu_X$ in (\ref{hgy4nd}) is an isomorphism. 

\begin{thm} \label{dadesthm} 
Let $S$ be a $\Gamma$-graded inverse semigroup. Then $S$ is strongly graded if and only if the categories $\Gc{S}$ and $\Moc{S_\varepsilon}$ are equivalent via the functors defined in Lemma \ref{funct-lemm} and natural transformations defined in Lemma \ref{natranstlemm}.
\end{thm}

\begin{proof}
Suppose that $S$ is strongly graded. By Lemma~\ref{natranstlemm}, to conclude that $\Gc{S}$ and $\Moc{S_\varepsilon}$ are equivalent, it suffices to show that, for each object $X$ in $\Gc{S}$, the morphism $\nu_X$ in (\ref{hgy4nd}) is an isomorphism. But by Lemma \ref{sixmonthd}, $\nu_X$ is an isomorphism, when viewed as a function $S \otimes_{S_\varepsilon} X_\varepsilon \rightarrow SX_\varepsilon$, and $S X_\varepsilon=X$, by Lemma~\ref{strongly-gr-lemma}.

Conversely, suppose that $\Gc{S}$ and $\Moc{S_\varepsilon}$ are equivalent via the relevant functors and natural transformations. Then for any object $X$ in $\Gc{S}$ and any $\alpha \in \Gamma$ we have an isomorphism
\begin{align*}
\nu_{X(\alpha)}: S \otimes_{S_\varepsilon} X(\alpha)_\varepsilon & \longrightarrow X(\alpha) \\
 s\otimes x &\longmapsto sx
\end{align*}
(see~(\ref{butyestergdtgf})). This implies that $SX_\alpha=X$ for all $\alpha \in \Gamma$, and therefore $S$ is strongly graded, by Lemma~\ref{strongly-gr-lemma}.   
\end{proof}

\section{Graded groupoids and inverse semigroups} \label{groupoidsection}

There has been recent interest~\cite{arasims,chr19} in systematically studying graded groupoids, as a result of their association with groupoid algebras. As mentioned in the Introduction (see (\ref{gffhdkcnejfksh})), there is a connection between groupoids of germs and inverse semigroups, as building blocks for combinatorial algebras. In this section we show that there is a tight relationship between the gradings on an inverse semigroup and those on the corresponding groupoid of germs. We also explore gradings on a class of inverse semigroups arising from ample groupoids. We begin with a brief review of groupoids and the associated notation.

\subsection{Groupoids} \label{groupoidsubsection}

A \emph{groupoid} is a small category in which every morphism is invertible. It can also be viewed as a generalisation of ``group", where multiplication is partially defined. 

Let $\mathscr{G}$ be a groupoid. We denote the set of objects of $\mathscr{G}$, also known as the \emph{unit space} of $\mathscr{G}$, by $\mathscr{G}^{(0)}$, and we identify these objects with the corresponding identity morphisms. For each morphism $x$ in $\mathscr{G}$, the object $\domr(x):=x^{-1}x$ is the \emph{domain} of $x$, and $\ran(x):=xx^{-1}$ is its \emph{range}. Thus, two morphisms $x$ and $y$ are composable as $xy$ if and only if $\domr(x)=\ran(y)$. Let 
\[\mathscr{G}^{(2)} := \{(x,y) \in \mathscr{G} \times \mathscr{G} \mid \domr(x) = \ran(y)\}\]
denote the set of composable pairs of morphisms of $\mathscr{G}$. For subsets $X,Y\subseteq \mathscr{G}$ of morphisms, we define
\begin{equation}\label{pofgtryhf}
XY=\big \{xy \mid x \in X, y \in Y, \text{ and } \domr(x)=\ran(y) \big\},
\end{equation}
and
\begin{equation}\label{pofgtryhf2}
X^{-1}=\big \{x^{-1} \mid x \in X \big \}.
\end{equation}

If $\mathscr{G}$ is a groupoid and $\Gamma$ is a group, then $\mathscr{G}$ is a $\Gamma$-\emph{graded groupoid} if there is functor $\phi:\mathscr{G} \rightarrow \Gamma$ (which can be viewed as a function from the set of all morphisms of $\mathscr{G}$ to $\Gamma$, such that $\phi(x)\phi(y) = \phi(xy)$ for all $(x,y) \in \mathscr{G}^{(2)}$). Setting $\mathscr{G}_\alpha:=\phi^{-1}(\alpha)$ for each $\alpha \in \Gamma$, we have
\[ \mathscr{G}= \bigcup_{\alpha \in \G} \mathscr{G}_\alpha, \]
where $\mathscr{G}_\alpha \mathscr{G}_\beta \subseteq \mathscr{G}_{\alpha\beta}$ for all $\alpha,\beta \in \Gamma$, and $\mathscr{G}_\alpha \cap \mathscr{G}_\beta = \emptyset$ for all distinct $\alpha,\beta \in \Gamma$. Note that $\mathscr{G}^{(0)} \subseteq \mathscr{G}_\varepsilon$.  We say that $\mathscr{G}$ is \emph{strongly graded} if $\mathscr{G}_\alpha \mathscr{G}_\beta =\mathscr{G}_{\alpha\beta}$ for all $\alpha,\beta \in \Gamma$. Analogously to Proposition~\ref{hgfgftgr}, it is shown in \cite[Lemma 3.1]{chr19} that for a $\Gamma$-graded groupoid $\mathscr{G}$ the following are equivalent.
\begin{enumerate}
\item $\mathscr{G}$ is strongly graded;
		
\smallskip

\item $\mathscr{G}_\alpha \mathscr{G}_{\alpha^{-1}}=\mathscr{G}_\varepsilon$ for every $\alpha \in \Gamma$;
		
\smallskip

\item $\domr(\mathscr{G}_\alpha)=\mathscr{G}^{(0)}$ for every $\alpha \in \Gamma$;
		
\smallskip

\item  $\ran(\mathscr{G}_\alpha)=\mathscr{G}^{(0)}$ for every $\alpha \in \Gamma$.
\end{enumerate}
 
\subsection{Groupoids of germs} 

Let $X$ be a nonempty set, and let $S$ be an inverse semigroup. We say that there is a \emph{partial action of $S$ on $X$} if there is a semigroup homomorphism $S\rightarrow \mathcal{I}(X)$ that preserves zero, where $\mathcal{I}(X)$ denotes the symmetric inverse monoid, discussed in \S \ref{vagnersect}. (A more common name for this notion in the literature is \emph{action}, however, we append ``partial" to avoid confusion with the concept described in \S \ref{grfdrferfsa}. See \cite{GH} for a discussion of other uses of ``partial action" in connection with semigroups.) We denote the image of $s$ under such a homomorphism by $\theta_s$, and set $X_s := \dom(\theta_s)$. In particular, $\theta_0$ is the empty function. We say that the partial action of $S$ on $X$ is \emph{non-degenerate} if $X=\bigcup_{e\in E(S)} X_e$. Using the fact that $ss^{-1}s=s$, one can show that $X_s=X_{s^{-1}s}$ and $\im(\theta_s) = X_{ss^{-1}}$, for all $s \in S$. 

Now let
\begin{equation*}
\mathscr{G}:=\bigcup_{s\in S} (\{s\}\times X_{s}) \subseteq S\times X,
\end{equation*}
and define a binary relation $\sim$ on $\mathscr{G}$ by letting $(s,x)\sim (t,y)$ whenever $x=y$, and there exists $e\in E(S)$ such that  $x\in X_e$ and $se=te$. It is easy to see that $\sim$ is an equivalence relation. We denote the equivalence class of $(s,x) \in \mathscr{G}$ by $[s,x]$, and call it the \emph{germ of $s$ at $x$}. Also, let $S\ltimes X := \mathscr{G}/ \sim$, and let \[(S\ltimes X)^{(0)} := \{[e,x] \mid e \in E(S), x \in X_e\}.\] Assuming that the partial action of $S$ on $X$ is non-degenerate, we identify the latter set with $X$, via the mapping $[e,x] \mapsto x$. For each $[s,x] \in S\ltimes X$ let $\domr([s,x]) = x$ and $\ran([s,x]) = \theta_s(x)$. Finally, for all $[s,x], [t,y] \in S\ltimes X$ let $[s,x]\inv = [s^{-1},sx]$, and let $[s,x][t,y]=[st,y]$, in case $\domr([s,x]) = \ran([t,y])$. It is routine to show that these operations are well-defined, making $S\ltimes X$ a groupoid, called the \emph{groupoid of germs}. (See \cite[\S 4]{exel20081} or \cite[p.\ 140]{paterson1} for more details.) If $S$ is $\Gamma$-graded, then mapping 
\begin{equation}\label{balconnn}
[s,x]\mapsto \deg(s),
\end{equation} 
which is easily seen to be well-defined, induces a grading $S\ltimes X \rightarrow \Gamma$. 

\begin{prop}\label{sthfyhry}
Let $S$ be a $\G$-graded inverse semigroup equipped with a non-degenerate partial action on a nonempty set $X$. If $S$ is strongly graded, then the groupoid of germs $S\ltimes X$ is strongly graded in the induced $\Gamma$-grading. 
\end{prop}

\begin{proof}
For each $\alpha \in \Gamma$, we have
\[(S\ltimes X)_\alpha=\{[s,x] \in S\ltimes X \mid s\in S_\alpha \}.\]
Since $S$ is strongly graded, by Proposition~\ref{hgfgftgr}, we have $E(S) = \{s^{-1}s \mid s \in S_\alpha\}$, and hence
\[\domr((S\ltimes X)_\alpha) =\bigcup_{s\in S_\alpha} X_{s^{-1}s}=\bigcup_{e\in E(S)}X_e=X=(S\ltimes X)^{(0)}.\] 
Thus $S\ltimes X$ is strongly graded, by~\cite[Lemma 3.1]{chr19} (see \S \ref{groupoidsubsection}). 
\end{proof}

It is natural to ask whether the converse of Proposition~\ref{sthfyhry} holds. In \S \ref{hfgfhdhksdkkls} we investigate gradings on graph inverse semigroups $\SS(E)$, and, in particular, we show in Corollary~\ref{source-cor} that $\SS(E)$ cannot be strongly $\mathbb{Z}$-graded if the graph $E$ has a source vertex. On the other hand, it is known that for any finite graph $E$ with no sinks, the graph groupoid $\mathscr{G}_E$ (see~\cite[\S 4.1.1]{chr19}) is strongly $\mathbb{Z}$-graded (see Corollary~\ref{hfghfyhff88}). Since $\mathscr{G}_E$ is the universal groupoid of $\SS(E)$ (i.e., the groupoid of germs of a certain partial action of $\SS(E)$, see~\cite[Definition 5.14]{st}), it follows that the converse of Proposition~\ref{sthfyhry} does not hold in general.

\subsection{Topological groupoids} \label{topgroupoid}

A \textit{topological groupoid} is a groupoid (whose set of morphisms is) equipped with a topology making inversion and composition continuous. A topological groupoid $\mathscr{G}$ is an \textit{\'etale groupoid} if $\mathscr{G}^{(0)}$ is locally compact and Hausdorff in the topology induced by that on $\mathscr{G}$, and $\domr: \mathscr{G} \to \mathscr{G}^{(0)}$ is a local homeomorphism. An open subset $X$ of an \'etale groupoid $\mathscr{G}$ is called a \emph{slice} or \emph{local bisection} if the restrictions $\domr|_{X}$ and $\ran|_{X}$ are injective (and hence are homeomorphisms onto their images). The collection of all slices forms a base for the topology of an \'etale groupoid $\mathscr{G}$ \cite[Proposition 3.5]{exel20081}, and $\mathscr{G}^{(0)}$ is a slice \cite[Proposition 3.2]{exel20081}. An \'etale groupoid is \emph{ample} if the compact slices form a base for its topology.

Let $\mathscr{G}$ be an ample groupoid, and set
\[\mathscr{G}^a=\{X\mid X \text{~is a compact slice of~} \mathscr{G}\}.\]
Then $\mathscr{G}^{a}$, sometimes denoted $\mathcal{S}(\mathscr{G})$ \cite{exel20081} or $\mathscr{G}^{co}$ \cite{paterson1}, is an inverse semigroup under the operations given in (\ref{pofgtryhf}) and (\ref{pofgtryhf2}) \cite[Proposition 2.2.4]{paterson1}, with $\emptyset$ as the zero element. Supposing that $\mathscr{G}$ is $\Gamma$-graded, we can build a graded version of $\mathscr{G}^{a}$. Specifically, we say that a slice $X$ of $\mathscr{G}$ is \emph{homogeneous} if $X \subseteq \mathscr{G}_\alpha$ for some $\alpha \in \Gamma$, and set
\begin{equation*}
\mathscr{G}^{h}=\{X\mid X \text{~is a homogeneous compact slice of~} \mathscr{G}\}.
\end{equation*}
Since $\mathscr{G}_{\alpha^{-1}} = \mathscr{G}_{\alpha}^{-1}$ for each $\alpha \in \Gamma$, we see that $\mathscr{G}^{h}$ is an inverse subsemigroup of $\mathscr{G}^a$. Moreover, it is easy to see that defining $\phi :\mathscr{G}^{h} \backslash \{\emptyset\} \rightarrow \Gamma$ by $\phi(X)= \alpha$, whenever $X\subseteq \mathscr{G}_\alpha$, turns $\mathscr{G}^{h}$ into a graded inverse semigroup.

The following proposition relates the gradings on $\mathscr{G}$ to those on the associated graded inverse semigroup $\mathscr{G}^h$. More specifically, $\mathscr{G}$ is strongly graded if and only if $\mathscr{G}^h$ satisfies a condition similar to ``locally strongly graded" (see Proposition~\ref{sussa}). 

\begin{prop}
Let $\mathscr{G}$ be an ample $\Gamma$-graded groupoid. Then $\mathscr{G}$ is strongly $\Gamma$-graded if and only if for all $\alpha, \beta \in \Gamma$, $X\in \mathscr{G}^h_{\alpha \beta} \backslash \{\emptyset\}$, and $u\in \domr(X)$, there is a compact open set $U\subseteq \mathscr{G}^{(0)}$ such that $u\in U$ and $XU \in \mathscr{G}^h_{\alpha} \mathscr{G}^h_{\beta}$. 
\end{prop}

\begin{proof}
We begin by defining a homomorphism $\pi: \mathscr{G}^h \rightarrow \mathcal{I}(\mathscr{G}^{(0)})$. Given $X\in \mathscr{G}^{h}$, since $X$ is a slice, we see that $XX^{-1} = \{xx^{-1} \mid x \in X\}$ and $X^{-1}X = \{x^{-1}x \mid x \in X\}$, and hence $XX^{-1}, X^{-1}X \subseteq \mathscr{G}^{(0)}$. Also, $Xx^{-1}x = \{x\}$ for all $x\in X$, and so we can define
\begin{align*}
\pi_X: X^{-1}X &\longrightarrow XX^{-1}\\
u &\longmapsto \ran(Xu).
\end{align*}
Then, clearly, $\pi_X \in \mathcal{I}(\mathscr{G}^{(0)})$. Also, it is easy to check that \[\pi_Y^{-1}(YY^{-1} \cap X^{-1}X) = Y^{-1}X^{-1}XY\] and $\pi_X \pi _Y =\pi _{XY}$, for all $X,Y \in \mathscr{G}^{h}$. Hence 
\begin{align*}
\pi:  \mathscr{G}^h &\longrightarrow \mathcal{I}(\mathscr{G}^{(0)})\\
X &\longmapsto \pi_X,
\end{align*}
is a homomorphism. Thus, there is a partial action of $\mathscr{G}^h$ on $\mathscr{G}^{(0)}$, and so we can construct the groupoid of germs $\mathscr{G}^h \ltimes \mathscr{G}^{(0)}$. Next, we note that according to~\cite[Proposition~5.4]{exel20081} and its proof, 
\begin{align*}
\phi:  \mathscr{G}^h \ltimes \mathscr{G}^{(0)} &\longrightarrow  \mathscr G\\
 [X,u]&\longmapsto Xu
\end{align*}
is a groupoid isomorphism. Moreover, $\phi$ respects the gradings on the above groupoids. For, letting $\alpha \in \Gamma$ and $[X,u] \in (\mathscr{G}^h \ltimes \mathscr{G}^{(0)})_\alpha$, we have $X \in \mathscr{G}^h_\alpha$, by (\ref{balconnn}), and hence $Xu \in \mathscr{G}_\alpha$. 

Now, suppose that $\mathscr{G}$ is strongly graded, and let $\alpha, \beta \in \Gamma$, $X\in \mathscr{G}^h_{\alpha \beta} \backslash \{\emptyset\}$, and $u\in \domr(X)$. Then $[X,u] \in (\mathscr{G}^h \ltimes \mathscr{G}^{(0)})_{\alpha \beta}$, and so $x:=Xu \in \mathscr{G}_{\alpha \beta}$. Hence, there exist $y\in \mathscr{G}_\alpha$ and $z\in \mathscr{G}_\beta$ such that $x=yz$. Since $\mathscr{G}$ is ample, we can choose compact slices $Y\in \mathscr{G}^h_\alpha$ and $Z \in \mathscr{G}^h_\beta$ such that $y\in Y$ and $z\in Z$. Since $X$ and $YZ$ are slices, we have $Xu=YZu$, and so $\phi([X,u]) = \phi([YZ,u])$. Since $\phi$ is injective, it follows that $[X,u]=[YZ,u]$. By the definition of the groupoid of germs, this means that there is a compact open set $U \subseteq \mathscr{G}^{(0)} (\subseteq \mathscr{G}_\varepsilon)$ such that $u\in U$ and $XU=YZU$. It follows that $XU = Y(ZU) \in \mathscr{G}^h_{\alpha} \mathscr{G}^h_{\beta}$. 

For the converse, suppose that for all $\alpha, \beta \in \Gamma$, $X\in \mathscr{G}^h_{\alpha \beta} \backslash \{\emptyset\}$, and $u\in \domr(X)$, there is a compact open set $U\subseteq \mathscr{G}^{(0)}$ such that $u\in U$ and $XU \in \mathscr{G}^h_{\alpha} \mathscr{G}^h_{\beta}$. Now let $\alpha, \beta \in \Gamma$ and $x \in \mathscr{G}_{\alpha \beta}$. Since $\mathscr{G}$ is ample, there exists $X\in \mathscr{G}^h_{\alpha\beta}\backslash \{\emptyset\}$ such that $x \in X$. Then, by hypothesis, we can find a compact open set $U\subseteq \mathscr{G}^{(0)}$ such that $\domr(x) \in U$ and $XU \in \mathscr{G}^h_{\alpha} \mathscr{G}^h_{\beta}$. It follows that $x=x(x^{-1}x) \in \mathscr{G}_{\alpha} \mathscr{G}_{\beta}$, and so $\mathscr{G}$ is strongly graded.
\end{proof}

\section{Semigroup rings}\label{semigroupringsec}

As mentioned before, our theory of graded semigroups was inspired by results about graded rings. Next we recall graded rings in more detail, as well as semigroup rings. We utilise these throughout the rest of the paper to draw closer connections between graded rings and graded semigroups.

Given a ring $A$ and group $\Gamma$, we say that $A$ is $\Gamma$-\emph{graded} if $A = \bigoplus_{\alpha \in \Gamma} A_\alpha$, where the $A_\alpha$ are additive subgroups of $A$, and $A_\alpha A_\beta \subseteq A_{\alpha \beta}$ for all $\alpha, \beta \in \Gamma$ (with $A_\alpha A_\beta$ denoting the set of all sums of elements of the form $ab$, for $a \in A_\alpha$ and $b \in A_\beta$). In this situation, $A$ is \emph{strongly} graded if $A_\alpha A_\beta = A_{\alpha \beta}$ for all $\alpha, \beta \in \Gamma$. Each $a \in A$ can be written uniquely as $a = \sum_{\alpha \in \Gamma} a_{\alpha}$, where $a_{\alpha} \in A_{\alpha}$ for each $\alpha \in \Gamma$, and all but finitely many of the $a_{\alpha}$ are zero. Here we refer to $a_{\alpha}$ as the \emph{homogeneous component of $a$ of degree $\alpha$}.

Next, given a ring $A$ and a semigroup $S$, we denote by $AS$ the corresponding semigroup ring, and by $A[S]$ the resulting \emph{contracted semigroup ring}, where the zero of $S$ is identified with the zero of $AS$. That is, $A[S] = AS/I$, where $I$ is the ideal of $AS$ generated by the zero of $S$. We denote an arbitrary element of $A[S]$ by $\sum_{s\in S} a^{(s)}s$ (or $\sum_{s\in S\setminus \{0\}} a^{(s)}s$), where $a^{(s)} \in A$, and all but finitely many of the $a^{(s)}$ are zero. 
 
When the semigroup $S$ is $\Gamma$-graded, one can naturally equip the ring $A[S]$ with a $\Gamma$-grading, producing a pleasant relationship between these two graded structures. Specifically, it is easy to check that defining 
\begin{equation} \label{smgpringgr}
A[S]_\alpha:= A[S_\alpha] = \Big\{\sum_{s\in S} a^{(s)}s \mid s \in S_\alpha \text{ whenever } a^{(s)}\neq 0\Big\},
\end{equation} 
for each $\alpha \in \Gamma$, turns $A[S]$ into a $\Gamma$-graded ring. We refer to this as the grading on $A[S]$ \emph{induced by the grading} of $S$. 

Similarly, if $A$ is a unital ring, $A[S]$ is $\Gamma$-graded, and $S=\bigcup_{\alpha \in \Gamma}  (S \cap A[S]_\alpha)$, then setting $S_\alpha:= S \cap A[S]_\alpha$, for each $\alpha \in \Gamma$, induces a grading on $S$. (Here we identify each $t \in S$ with $1t \in A[S]$, i.e., the element $\sum_{s\in S} a^{(s)}s$, where $a^{(t)}=1$, and the other coefficients are $0$.) We note that, generally speaking, $S \neq \bigcup_{\alpha \in \Gamma}  (S \cap A[S]_\alpha)$ for a $\Gamma$-grading on $A[S]$, as the following example shows. 

\begin{example}
Let $K$ be a field, and let $S=\big \{e_{ij}(k) \mid 1\leq i,j\leq 2, k\in K\big\}$ be the Rees matrix semigroup (with multiplication as in (\ref{stableprod})). It is easy to show that
$K[S]\cong \M_2 (K)$, the ring of $2 \times 2$ matrices with coefficients in $K$. Also, it is easy to see that defining 
\begin{align*}
K[S]_{0} = \left\{
\begin{pmatrix}
a & b-a \\ 0 & b
\end{pmatrix} \Big{|} \ a,b \in K \right\}
\mathrm{ \;\;\; and \;\;\; } 
K[S]_{1} = \left\{
\begin{pmatrix}
d & c \\d & -d
\end{pmatrix} \Big{|} \ c,d \in K \right\},
\end{align*}
gives a $\mathbb Z_2$-grading on $K[S]$. Then 
\begin{align*}
S \cap K[S]_{0} = \left\{
\begin{pmatrix}
0 & 0 \\ 0 & 0
\end{pmatrix} \right\} 
\mathrm{ \;\;\; and \;\;\; } 
S \cap K[S]_{1} = \left\{
\begin{pmatrix}
0 & c \\0 & 0
\end{pmatrix} \Big{|} \ c \in K \right\},
\end{align*}
which implies that the grading on $K[S]$ does not induce a grading on $S$. 
\end{example}

We will show that when the grading on $A[S]$ does induce a grading on $S$, the two objects share certain properties. Before stating the next result, let us review the smash product for rings, first introduced by Cohen and Montgomery~\cite{cohenmont}.

Let $A = \bigoplus_{\alpha \in \Gamma} A_\alpha$ be a $\Gamma$-graded ring, and let $A\#\Gamma$ denote the set of formal sums $\sum_{\alpha \in \Gamma} a^{(\alpha)} P_{\alpha}$, where each $a^{(\alpha)} \in A$, and all but finitely many of the $a^{(\alpha)}$ are zero. We define addition on $A\#\Gamma$ via 
\begin{equation}
\sum_{\alpha \in \Gamma} a^{(\alpha)} P_{\alpha} + \sum_{\alpha \in \Gamma} b^{(\alpha)} P_{\alpha} = \sum_{\alpha \in \Gamma} (a^{(\alpha)}+b^{(\alpha)}) P_{\alpha},
\end{equation}
and define multiplication by letting
\begin{equation}\label{smashmult2}
a^{(\alpha)} P_{\alpha}b^{(\beta)} P_{\beta} = a^{(\alpha)}b^{(\beta)}_{\alpha\beta^{-1}}P_{\beta}
\end{equation}
for all $a^{(\alpha)}, b^{(\beta)} \in A$ and $\alpha,\beta \in \Gamma$ (where $b^{(\beta)}_{\alpha\beta^{-1}}$ is the homogeneous component of $b^{(\beta)}$ of degree $\alpha\beta^{-1}$), and extending linearly. With these operations $A\#\Gamma$ becomes a ring, called the \emph{smash product of $A$ by $\Gamma$}. See \cite[\S 7.1]{grrings} for an alternative description of these rings.

\begin{prop}\label{hfhffhghggh}
Let $S$ be a $\Gamma$-graded semigroup and $A$ a unital ring. Also view $A[S]$ as a $\Gamma$-graded ring via the grading induced by that on $S$. Then the following hold.
\begin{enumerate}[\upshape(1)]

\item $S$ is a strongly $\Gamma$-graded semigroup if and only if $A[S]$ is a strongly $\Gamma$-graded ring. 

\smallskip

\item There is a natural ring isomorphism $A[S\# \Gamma] \cong A[S] \# \Gamma$. 
\end{enumerate}
\end{prop}

\begin{proof}
(1) Suppose that $S$ is strongly $\Gamma$-graded. Let $\alpha, \beta \in \Gamma$, and let $\sum_{s\in S} a^{(s)}s$ be an element of $A[S]_{\alpha\beta}$. Then for each $s$ with $a^{(s)} \neq 0$, we have $s=p^{(s)}q^{(s)}$, for some $p^{(s)}\in S_\alpha$ and $q^{(s)}\in S_\beta$. Hence 
\[\sum_{s\in S}a^{(s)}s = \sum_{s\in S} (a^{(s)} p^{(s)})(1q^{(s)})\in A[S]_{\alpha} A[S]_{\beta},\] 
which implies that $A[S]$ is strongly $\Gamma$-graded.

Conversely, suppose that $A[S]$ is strongly $\Gamma$-graded. Let $\alpha, \beta \in \Gamma$, and let $t\in S_{\alpha\beta}$. Then 
\[1t=\Big(\sum_{r\in S} a^{(r)}r\Big)\Big(\sum_{s\in S} b^{(s)}s\Big) = \sum_{r,s \in S} a^{(r)}b^{(s)}rs,\] 
for some $\sum_{r\in S} a^{(r)}r \in A[S]_{\alpha}$ and $\sum_{s\in S} b^{(s)}s \in A[S]_{\beta}$. Necessarily $a^{(r)}b^{(s)} =1$ for some $r,s \in S$, and the remaining coefficients are $0$, from which it follows that $t=rs \in S_\alpha S_\beta$. Thus $S$ is strongly $\Gamma$-graded.

(2) Define 
\begin{align*}
\phi :A[S\# \Gamma] & \longrightarrow A[S]\# \Gamma \\ \notag
\sum_{sP_\alpha \in S\# \Gamma\setminus\{0\}} a^{(sP_\alpha)}(sP_\alpha) & \longmapsto \sum_{\alpha \in \Gamma} \Big(\sum_{s\in S\setminus\{0\}} a^{(sP_\alpha)}s\Big) P_{\alpha}. 
\end{align*}
Clearly, $\phi$ is a bijection that respects addition. Thus, to conclude that $\phi$ is a ring isomorphism, it suffices to check that $\phi$ respects multiplication on summands. Now, let $a^{(sP_{\alpha})}(sP_{\alpha}), b^{(tP_{\beta})}(tP_{\beta}) \in A[S\# \Gamma]$. Then using (\ref{smashmult1}), (\ref{smgpringgr}), and (\ref{smashmult2}), we have
\begin{align*} 
\phi(a^{(sP_\alpha)}(sP_\alpha) b^{(tP_{\beta})}(tP_{\beta})) 
& = \left\{ \begin{array}{ll}
\phi(a^{(sP_\alpha)}b^{(tP_{\beta})}stP_{\beta}) & \text{if } st \neq 0 \text{ and } t\in S_{\alpha\beta^{-1}} \\
0 & \text{otherwise } 
\end{array}\right. \\ 
& = \left\{ \begin{array}{ll}
(a^{(sP_\alpha)}b^{(tP_{\beta})}st)P_{\beta} & \text{if } st \neq 0 \text{ and } t\in S_{\alpha\beta^{-1}} \\
0 & \text{otherwise } 
\end{array}\right.  \\ 
& = (a^{(sP_\alpha)}s)(b^{(tP_{\beta})}t)_{\alpha\beta^{-1}}P_{\beta}  \\
& = (a^{(sP_\alpha)}s)P_{\alpha} (b^{(tP_{\beta})}t)P_{\beta} \\
& = \phi(a^{(sP_\alpha)}(sP_\alpha))\phi(b^{(tP_{\beta})}(tP_{\beta})),
\end{align*}
as desired.
\end{proof}

While influence has typically flowed from ring theory to semigroup theory, the \emph{Rees}, or \emph{Munn, matrix ring}~\cite{marki} is an example of a ring construction that was directly influenced by semigroup theory. Here we give a graded version of this idea, which produces a rich class of graded rings, and does not seem to have appeared in the literature before. We will also relate these rings to the graded Rees matrix semigroups discussed in \S \ref{reesmatrixgroup}.

Let $A$ be a $\Gamma$-graded unital ring, and let $I$ and $J$ be non-empty (index) sets.  For all $i \in I$ and $j \in J$, fix $\alpha_i, \beta_j \in \Gamma$, and set $\overline{\alpha} = (\alpha_i)_{i\in I} \in \Gamma^I$ and $\overline{\beta}=(\beta_j)_{j\in J}\in \Gamma^J$. Let $B:=\M_{I,J}(A)$ denote the abelian group of all $I\times J$ matrices over $A$, with only finitely many nonzero entries. For each $\delta \in \Gamma$ let 
\begin{equation}\label{mtrxgrd}
B_\delta := \{(a_{ij}) \in B \mid i\in I, j \in J, a_{ij}\in A_{\alpha_i \delta \beta_j^{-1}} \}.
\end{equation}
Then $B=\bigoplus_{\delta \in \Gamma} B_\delta$. Next, let $p = (a_{ij})$ be a $J\times I$ matrix (possibly with infinitely many nonzero entries), such that  $a_{ji}\in A_{\beta_j \alpha_i^{-1}}$ for all $i\in I, j \in J$, and define multiplication by
\begin{equation} \label{grmtrxmult}
a.b:=a p b
\end{equation}
for all $a,b \in B$. It is easy to see that with this operation $B$ becomes a ring, which we denote by $\M_{I,J}^p(A)[\overline{\alpha}][\overline{\beta}]$. It is also easy to check that if $a \in B_\delta$ and $b \in B_\gamma$ for some $\delta, \gamma \in \Gamma$, then $apb \in B_{\delta \gamma}$. So $\M_{I,J}^p(A)[\overline{\alpha}][\overline{\beta}]$ is a $\Gamma$-graded ring, which we call a \emph{graded Rees matrix ring}.

Note that if $I=J=\{1,\dots,n\}$, $\overline{\alpha} = \overline{\beta}=(\alpha_1,\dots,\alpha_n)$, and $p$ is the identity matrix, then $\M_{I,J}^p(A)[\overline{\alpha}][\overline{\beta}]$ is simply the graded matrix ring $\M_n(A)(\alpha_1,\dots,\alpha_n)$, which has received considerable attention in the literature--see~\cite{hazi,grrings}. 

\begin{prop}
Let $A$ be a unital ring, and let $S$ be a $\Gamma$-graded semigroup. Then for all nonempty sets $I$ and $J$, tuples $\overline{\alpha} = (\alpha_i)_{i\in I} \in \Gamma^I$ and $\overline{\beta} = (\beta_j)_{j\in J}\in \Gamma^J$, and $J\times I$ matrices $p = (s_{ij})$ with $s_{ji}\in S_{\beta_j \alpha_i^{-1}}$ for all $i\in I, j \in J$, there is a natural graded ring isomorphism 
\[A [\E_{I,J}^p(S)[\overline{\alpha}][\overline{\beta}]] \cong \M_{I,J}^p(A[S])[\overline{\alpha}][\overline{\beta}],\]
viewing the former as a $\Gamma$-graded ring via the grading induced by that on $\E_{I,J}^p(S)[\overline{\alpha}][\overline{\beta}]$. 
\end{prop}

\begin{proof}
Define 
\begin{align*}
\phi :A[\E_{I,J}^p(S)[\overline{\alpha}][\overline{\beta}]] & \longrightarrow \M_{I,J}^p(A[S])[\overline{\alpha}][\overline{\beta}] \\ \notag
\sum_{e_{ij}(s) \in \E_{I,J}^p(S)[\overline{\alpha}][\overline{\beta}]} a^{(e_{ij}(s))}e_{ij}(s) & \longmapsto \Big(\sum_{s \in S}a^{(e_{ij}(s))}s\Big). \notag
\end{align*}
Clearly, $\phi$ is a bijection that respects addition. Using (\ref{hfbvhsphwq}) and (\ref{grmtrxmult}), it is routine to check that $\phi$ respects multiplication, and hence it is a ring homomorphism. Finally, it follows from (\ref{Reesgrade}), (\ref{smgpringgr}), and (\ref{mtrxgrd}) that $\phi$ respects the grading.
\end{proof}

\section{Graph inverse semigroups}\label{hfgfhdhksdkkls}

In this section we explore gradings on a rich class of inverse semigroups, known as \emph{graph inverse semigroups}, first introduced in~\cite{ash}, along with gradings on various related objects. (See~\cite{MM} for more on the history of graph inverse semigroups, and~\cite{jones} for an alternative perspective on them.) Among other results, we classify the strongly graded graph inverse semigroups, and the graded graph inverse semigroups that produce strongly graded Leavitt path algebras (see (\ref{graphobjects})), which are defined below.

\subsection{Definitions and basics}\label{graphs-sect}

Recall that a \emph{directed graph} $E=(E^0,E^1,\ra,\so)$ consists of two sets, $E^0$ and $E^1$ (containing \emph{vertices} and \emph{edges}, respectively), together with functions $\so,\ra:E^1 \to E^0$, called \emph{source} and \emph{range}, respectively. A \emph{path} $x$ in $E$ is a finite sequence of (not necessarily distinct) edges $x=e_1\cdots e_n$ such that $\ra(e_i)=\so(e_{i+1})$ for $i=1,\dots,n-1$. In this case, $\so(x):=\so(e_1)$ is the \emph{source} of $x$, $\ra(x):=\ra(e_n)$ is the \emph{range} of $x$, and $|x|:=n$ is the \emph{length} of $x$. If $x = e_1\cdots e_n$ is a path in $E$ such that $\so(x)=\ra(x)$ and $\so(e_i)\neq \so(e_j)$ for every $i\neq j$, then $x$ is called a \emph{cycle}. For a vertex $v \in E^0$, we say that $v$ is a \emph{sink} if $\so^{-1}(v) = \emptyset$, that $v$ is a \emph{source} if $\ra^{-1}(v) = \emptyset$, and that $v$ is \emph{regular} if $0 < |\so^{-1}(v)| < \aleph_0$. We view the elements of $E^0$ as paths of length $0$ (extending $\so$ and $\ra$ to $E^0$ via $\so(v)=v$ and $\ra(v)=v$ for all $v\in E^0$), and denote by $\pth(E)$ the set of all paths in $E$. A infinite sequence $e_1e_2\cdots$ of edges in $E^1$ is called an \emph{infinite path} if $\ra(e_i)=\so(e_{i+1})$ for all $i \geq 1$. Given a finite or infinite path $p$ in $E$ and $x \in \pth(E)$, we say that $x$ is an \emph{initial subpath} of $p$ if $p=xq$ for some path $q$. Finally, $E$ is said to be \emph{row-finite} if $|\so^{-1}(v)| < \aleph_0$ for every $v \in E^0$. From now on we will refer to directed graphs as simply ``graphs".

\begin{deff}\label{shdnfhyhgkdk}
Given a graph $E=(E^0,E^1,\ra,\so)$, the \emph{graph inverse semigroup $\SS(E)$ of $E$} is the semigroup (with zero) generated by the sets $E^0$ and $E^1$, together with $\{e^{-1} \mid e\in E^1\}$, satisfying the following relations for all $v,w\in E^0$ and $e,f\in E^1$ (where $\delta$ is the Kronecker delta):
\begin{enumerate}
\item[(V)] $vw = \delta_{v,w}v$;

\item[(E1)] $\so(e)e=e\ra(e)=e$;

\item[(E2)] $\ra(e)e^{-1}=e^{-1}\so(e)=e^{-1}$;

\item[(CK1)] $e^{-1}f=\delta _{e,f}\ra(e)$.
\end{enumerate}
\end{deff}

We define $v^{-1}=v$ for each $v \in E^0$, and for any path $y=e_1\cdots e_n$ ($e_1,\dots, e_n \in E^1$) we let $y^{-1} = e_n^{-1} \cdots e_1^{-1}$. With this notation, it is easy to see that every nonzero element of $\SS(E)$ can be written as $xy^{-1}$ for some $x, y \in \pth(E)$, such that $\ra(x)=\ra(y)$. It is well-known that representations in this form of nonzero elements of $\SS(E)$ are unique. (This follows, for example, from the model for $\SS(E)$ constructed in~\cite[\S 3]{Paterson}.) It is also easy to verify that $\SS(E)$ is indeed an inverse semigroup, with $(xy^{-1})^{-1} = yx^{-1}$ for all $x, y \in \pth (E)$.

As semigroups defined by generators and relations, graph inverse semigroups lend themselves naturally to being graded (see Example~\ref{jgjgjejjiii}). Let $E$ be a graph, $\Gamma$ a group, and $\omega :E^1\rightarrow \Gamma$ a ``weight" map. Now extend $\omega$ to a function $\omega : \pth(E) \rightarrow \Gamma$ by letting \[\omega (e_1\cdots e_n) = \omega (e_1)\cdots \omega (e_n)\] for all $e_1,\dots, e_n \in E^1$, and letting $\omega (v) = \varepsilon$ for all $v \in V$. Then it is easy to see that $\SS(E)$ is $\Gamma$-graded, via 
\begin{align} \label{graphinvgr}
 \SS(E) \setminus \{0\} & \longrightarrow \Gamma \\
 xy^{-1} & \longmapsto \omega (x)\omega (y)^{-1}. \notag
\end{align}
Now letting $\Gamma = \Z$ and taking $w(e)=1$ for each $e \in E^1$, we obtain a $\Z$-grading $\phi : \SS(E) \setminus \{0\} \to \Z$ on $\SS(E)$, where $\phi(xy^{-1}) = |x|-|y|$ for all $x,y \in \pth(E)$. We refer to this as the \emph{natural} $\Z$-grading of $\SS(E)$.

We conclude this subsection with a couple of easy observations that relate properties of a graph $E$ to the natural partial order on $\SS(E)$, which will be useful later on.

\begin{lemma} \label{locstrgr2}
Let $E$ be a graph. Then the following are equivalent. 
\begin{enumerate}[\upshape(1)]
\item $E$ has no sinks;

\smallskip

\item There are no minimal, with respect to $\leq$, idempotents in $\SS(E)\setminus \{0\}$.
\end{enumerate}
\end{lemma}

\begin{proof}
(1) $\Rightarrow$ (2) Suppose that $E$ has no sinks, and that $u \in E(\SS(E)) \setminus \{0\}$. Then it is easy to see that $u = xx^{-1}$ for some $x \in \pth(E)$~\cite[Lemma 15(1)]{MM}. By hypothesis, there is some $e \in E^1$ satisfying $\so(e) = \ra(x)$. Then $u > xee^{-1}x^{-1}$, and so $u$ is not a minimal idempotent.

(2) $\Rightarrow$ (1) Suppose that (2) holds, and that $v \in E^0$. By hypothesis, there must exist $u \in E(\SS(E))\setminus \{0\}$ such that $u < v$. Writing $u = xx^{-1}$ for some $x \in \pth(E)$, necessarily $|x| \geq 1$, and $\so(x) = v$. Thus $v$ is not a sink.
\end{proof}

\begin{lemma} \label{locstrgr3}
Let $E$ be a graph. Then the following are equivalent.
\begin{enumerate}[\upshape(1)]
\item $E$ is row-finite;

\smallskip

\item For every maximal, with respect to $\leq$, idempotent $u$ in $\SS(E)$, there are only finitely many maximal idempotents in $\{v \in E(\SS(E)) \mid v< u\}$.
\end{enumerate}
\end{lemma}

\begin{proof}
First, suppose that $u \in E^0$, and $v \in E(\SS(E)) \setminus \{0\}$ is such that $v <u$. Then, by the defining relations of $\SS(E)$, we have $v \notin E^0$. Moreover, by~\cite[Lemma 15(4)]{MM}, the maximal (nonzero) idempotents in $E(\SS(E)) \setminus E^0$ are precisely the elements of the form $ee^{-1}$, for some $e \in E^1$. It follows that if $v$ is maximal in $\{v \in E(\SS(E)) \mid v< u\}$, then it must be the case that $v = ee^{-1}$ for some $e \in E^1$, where necessarily $\so(e) = u$.

(1) $\Rightarrow$ (2) Suppose that $E$ is row-finite, and that $u \in E(\SS(E))$ is maximal. We may assume that $u \neq 0$, since otherwise $\SS(E) = \{0\}$, and (2) holds vacuously. Then, by~\cite[Lemma 15(3)]{MM}, $u \in E^0$. By the above, either $\{v \in E(\SS(E)) \mid v< u\} = \{0\}$, or the maximal idempotents in $\{v \in E(\SS(E)) \mid v< u\}$ are of the form $ee^{-1}$ ($e \in E^1$), where $\so(e) = u$. Since $E$ is row-finite, there can be only finitely many such elements.

(2) $\Rightarrow$ (1) Suppose that $u \in E^0$. Then, by~\cite[Lemma 15(3)]{MM}, $u$ is a maximal idempotent in $\SS(E)$. Supposing that (2) holds, there are only finitely many maximal idempotents in $\{v \in E(\SS(E)) \mid v< u\}$. By the first paragraph, these maximal idempotents are precisely the elements of $\SS(E)$ of the form $ee^{-1}$ ($e \in E^1$), where $\so(e) = u$ (unless $\{v \in E(\SS(E)) \mid v< u\} = \{0\}$). It follows that $u$ can emit only finitely many edges, and so $E$ is row-finite.
\end{proof}

\subsection{Strongly graded graph inverse semigroups}

In this subsection we give a reasonably complete description of the graph inverse semigroups that are strongly graded, paying particular attention to the natural $\Z$-grading.

\begin{lemma} \label{str-lemma}
Let $\SS(E)$ be a strongly $\Gamma$-graded graph inverse semigroup. Then for all $\alpha \in \Gamma$ and all $x \in \pth(E)$, there exists $y \in \pth(E)$ such that $xy^{-1} \in \SS(E)_\alpha$ and $\ra(y) = \ra(x)$. Moreover, if $\alpha \neq \varepsilon$, then $y \neq x$. 
\end{lemma}

\begin{proof}
Let $\alpha \in \Gamma$ and $x \in \pth(E)$. Then $xx^{-1} \in \SS(E)_\varepsilon$, since $xx^{-1}$ is an idempotent. Since $\SS(E)$ is strongly graded, we can find $p,q,s,t \in \pth(E)$ such that $pq^{-1} \in \SS(E)_\alpha$, $st^{-1} \in \SS(E)_{\alpha^{-1}}$, and $pq^{-1}st^{-1} = xx^{-1}$. In particular, $pq^{-1}st^{-1} \neq 0$, and hence (by (E1), (E2), and (CK1)), there exists $r \in \pth(E)$ such that either $q = sr$ or $s=qr$. 

If $q = sr$, then, since $\ra(s)=\ra(t)$, we have 
\[xx^{-1} = pq^{-1}st^{-1} = pr^{-1}s^{-1}st^{-1} = pr^{-1}t^{-1} = p(tr)^{-1}.\] 
By the aforementioned uniqueness of the representations of the nonzero elements of $\SS(E)$, we conclude that $p=x$. Hence $xq^{-1} = pq^{-1} \in \SS(E)_\alpha$, and $pq^{-1} \neq 0$ implies that $\ra(q) = \ra(x)$, as desired. 

Similarly, if $s=qr$, then 
\[xx^{-1}  = pq^{-1}st^{-1} = pq^{-1}qrt^{-1} = prt^{-1},\] 
which implies that $x=t$. Hence $sx^{-1}  = st^{-1} \in \SS(E)_{\alpha^{-1}}$, and $\ra(s) = \ra(x)$. Therefore $xs^{-1} \in \SS(E)_\alpha$, again giving the desired conclusion.

The final claim follows from the fact that if $\alpha \neq \varepsilon$, then $xx^{-1} \notin \SS(E)_\alpha$, since $\SS(E)_\alpha$ contains no nonzero idempotents.
\end{proof}

\begin{cor} \label{source-cor}
Let $E$ be a graph having a source vertex $v \in E^0$. Then any strong grading on $\SS(E)$ is trivial.
\end{cor}

\begin{proof}
Suppose that $\SS(E)$ is strongly $\Gamma$-graded. Suppose further that the grading is not trivial, and let $\alpha \in \Gamma \setminus \{\varepsilon\}$. Taking $x = v$, by Lemma~\ref{str-lemma}, there exists $y \in \pth(E)$ such that $\ra(y) = v$ and $y \neq v$ (and $y^{-1} \in \SS(E)_\alpha$), contrary to the hypothesis that $v$ is a source. Hence it must be the case that $\Gamma = \{\varepsilon\}$, i.e., the grading must be trivial.
\end{proof}

\begin{cor} \label{usual-grading}
Let $E$ be a graph. Then $\SS(E)$ is strongly graded in the natural $\Z$-grading if and only if $E$ is empty.
\end{cor}

\begin{proof}
Suppose that $E$ is the empty graph. Then $\SS(E) = \{0\}$, and $\SS(E)_{n} = \{0\}$ for all $n \in \Z$. Thus $\SS(E)_{n}\SS(E)_{m} = \SS(E)_{n+m}$ for all $n,m \in \Z$.

For the converse, suppose that $\SS(E)$ is strongly graded in the natural $\Z$-grading, and that $E$ is nonempty. Then we can find $x \in E^0$, and so, by Lemma~\ref{str-lemma}, there exists $y \in \pth(E)$ such that $y^{-1} \in \SS(E)_{1}$, $\ra(y) = x$, and $y \neq x$. Since $\ra(y) = x$ and $y \neq x$, necessarily $|y| \geq 1$. But then $y^{-1} \in \SS(E)_{1}$ contradicts the definition of the natural $\Z$-grading. Thus if $\SS(E)$ is strongly graded in the natural $\Z$-grading, then $E$ must be empty.
\end{proof}

\begin{lemma} \label{zn-lemma}
Let $E$ be a nonempty graph with no source vertices, and $n$ a positive integer. Then $\SS(E)$ is strongly $\Z/n\Z$-graded, via
\begin{align*} 
\phi :  \SS(E) \setminus \{0\} & \longrightarrow \Z/n\Z\\
 xy^{-1} & \longmapsto (|x|-|y|) + n\Z. 
\end{align*}
\end{lemma}

\begin{proof}
The map $\phi$ is a grading, since it is the composite of the natural grading $\SS(E) \setminus \{0\} \to \Z$ with the quotient group homomorphism $\Z \to \Z/n\Z$. To show that $\phi$ is a strong grading, let $0 \leq a,b < n$ be integers, and let $xy^{-1} \in \SS(E)_{\overline{a+b}}$, for some $x,y \in \pth(E)$ with $\ra(x)=\ra(y)$ (where $\overline{c} : = c + n\Z$ for all $c \in \Z$). Since $E$ has no sources, we can find $z \in \pth(E)$ such that $\ra(z) = \ra(x)$ and $|z| = |y|+b$. Then
\[\phi(xz^{-1}) = \overline{|x| - |y| - b} = \overline{a + b - b} = \overline{a},\] and 
\[\phi(zy^{-1}) = \overline{|y| + b - |y|} = \overline{b}.\]
Hence \[xy^{-1} = (xz^{-1})(zy^{-1})  \in \SS(E)_{\overline{a}}\SS(E)_{\overline{b}},\] and so $\phi$ is a strong $\Z/n\Z$-grading.
\end{proof}

\begin{thm} \label{strgrgis}
A graph inverse semigroup $\SS(E)$ has a nontrivial strong grading if and only if the graph $E$ is nonempty and has no source vertices.
\end{thm}

\begin{proof}
This follows immediately from Corollary~\ref{source-cor} and Lemma~\ref{zn-lemma}, upon noting that if $E$ is empty, then $\SS(E) = \{0\}$.
\end{proof}

We note that the grading constructed in Lemma~\ref{zn-lemma} is essentially the only  strong grading applicable to the entire class of graph inverse semigroups. More specifically, if $E$ is the graph with one vertex and one edge, then this grading is effectively the only strong grading for $\SS(E)$. To see this and make it more precise, we recall that in this case $\SS(E)\setminus \{0\}$ is the bicyclic monoid, having the following presentation as a semigroup: 
\[\langle x, x^{-1} \mid x^{-1}x = 1\rangle,\] 
where we identify $x$ with the edge of $E$ and $1$ with the vertex of $E$. It follows that any grading $\phi : \SS(E) \setminus \{0\} \to \Gamma$ is a semigroup homomorphism, and is completely determined by $\phi(x)$. Thus $\phi(\SS(E) \setminus \{0\})$ is necessarily a cyclic group, and hence isomorphic to either $\Z$ or $\Z/n\Z$, for some integer $n$. Now, if $\phi(\SS(E) \setminus \{0\}) \cong \Z$, then $\phi(x)$ is  necessarily mapped to either $-1$ or $1$ under this isomorphism. Composing with the isomorphism $\Z \to \Z$ that sends $-1 \mapsto 1$, if necessary, we can then identify $\phi$ with the natural $\Z$-grading of $\SS(E)$. Hence, by Corollary~\ref{usual-grading}, $\phi$ is not a strong grading in this case. Thus, if $\phi : \SS(E) \setminus \{0\} \to \Gamma$ is a strong grading, then $\phi(\SS(E) \setminus \{0\}) \cong \Z/n\Z$ for some integer $n$. Replacing $\Z/n\Z$ with an isomorphic copy, if necessary, in this situation $\phi$ can be identified with the grading in Lemma~\ref{zn-lemma}.

We conclude this subsection with a description of the locally strongly $\Z$-graded graph inverse semigroups (see Definition~\ref{locallystrgrdef}).

\begin{prop} \label{locstrgrgraph}
Let $E$ be a graph. Then the following are equivalent. 
\begin{enumerate}[\upshape(1)]
\item $\SS(E)$ is locally strongly graded in the natural $\Z$-grading;

\smallskip

\item For all $v \in E^{0}$ and all $n \in \Z$ there exist $x,y \in \pth(E)$ such that $\so(x) = v$, $\ra(x) = \ra(y)$, and $|x|-|y| = n$.
\end{enumerate}
\end{prop}

\begin{proof}
(1) $\Rightarrow$ (2) Suppose that (1) holds, and that $v \in E^{0}$ and $n \in \Z$. By Proposition~\ref{sussa}, there exists $x \in \SS(E)_n\setminus \{0\}$ such that $xx^{-1} \leq v$. Writing $x = yz^{-1}$ for some $y,z \in \pth(E)$ with $\ra(y) = \ra(z)$, we have $|y|-|z| = n$, and necessarily $\so(y) = v$.

(2) $\Rightarrow$ (1) Supposing that (2) holds, by Proposition~\ref{sussa}, to prove (1) it suffices to take arbitrary $n \in \Z$ and $u \in E(\SS(E)) \setminus \{0\}$, and show that $v \leq u$ for some $v \in E(\SS(E))_n \setminus \{0\}$.

It is easy to see that $u = xx^{-1}$ for some $x \in \pth(E)$ \cite[Lemma 15(1)]{MM}. By (2), there exist $y,z \in \pth(E)$ such that $\so(y) = \ra(x)$, $\ra(y) = \ra(z)$, and  $|y|-|z| = n-|x|$. Letting $q = xyz^{-1}$, we see that $q \in \SS(E)_n$. Hence $v: = qq^{-1} \in E(\SS(E))_n  \setminus \{0\}$, and clearly $uv = v$. Thus $v \leq u$, as desired.
\end{proof}

\subsection{Path algebras}\label{hfgfhyhhvgfgffd}

Given a field $K$ and a graph $E$, the contracted semigroup ring (see \S \ref{semigroupringsec}) $K[\SS(E)]$ is called the \emph{Cohn path $K$-algebra $C_K(E)$ of $E$}. Furthermore, the ring 
\begin{equation*}
L_K(E) := K[\SS(E)] \big / \Big\langle  v~-~\sum_{e\in \so^{-1}(v)} ee^{-1} \ \bigl\vert \ v \in E^0 \text{ is regular} \Big \rangle, 
\end{equation*}
is called the \emph{Leavitt path $K$-algebra of $E$}. (See~\cite{AAS}.)

Gradings on Leavitt path algebras have been studied in several papers. More specifically, the natural $\Z$-grading on $\SS(E)$ induces one on $C_K(E)$ (see \S \ref{semigroupringsec}), which in turn induces a grading on $L_K(E)$. We refer to these as the \emph{natural $\Z$-grading} on $C_K(E)$, respectively $L_K(E)$. It is shown in~\cite{hazi,chr19} that $L_K(E)$ is strongly graded with respect to the natural $\Z$-grading if and only if $E$ has no sinks, is row-finite, and satisfies the following condition.
\begin{enumerate}
\item[(Y)] For every natural number $n$ and every infinite path $p$ in $E$, there exists an initial subpath $x$ of $p$ and a path $y \in \pth(E)$ such that $\ra(y) = \ra(x)$ and $|y| - |x| = n$.
\end{enumerate}
Our next goal is to translate condition (Y) into one on $\SS(E)$, which will allow us to relate the semigroup more closely with the corresponding Leavitt path algebra. This requires introducing a new type of grading.

\begin{deff} 
Let $S$ be a $\Gamma$-graded inverse semigroup. We say that $S$ is \emph{saturated strongly $\Gamma$-graded} if for every $\alpha \in \Gamma$ and every infinite strictly descending chain of idempotents $u_0 > u_1 > \cdots$ in $S$, where $u_0$ is maximal with respect to $\leq$, there exist $n \geq 0$ and $v \in E(S)_{\alpha} \ (=\{ss^{-1} \mid s \in S_\alpha\})$ such that $u_0 \geq v \geq u_n$.   
\end{deff}

Note that, by Proposition~\ref{hgfgftgr}, for every strongly $\Gamma$-graded inverse semigroup $S$ we have $E(S) = E(S)_{\alpha}$ for all $\alpha \in \Gamma$, which implies that $S$ is saturated strongly $\Gamma$-graded. As shown in Example~\ref{locgrex}, this condition is, however, independent of ``locally strongly graded".

\begin{lemma} \label{locstrgr1}
Let $E$ be a graph. Then the following are equivalent.
\begin{enumerate}[\upshape(1)]
\item $E$ satisfies condition (Y);

\smallskip

\item $\SS(E)$ is saturated strongly graded in the natural $\Z$-grading.
\end{enumerate}
\end{lemma}

\begin{proof}
(1) $\Rightarrow$ (2) Suppose that (1) holds, let $n \in \Z$, and let  $u_0 > u_1 > \cdots$ be a chain of idempotents in $\SS(E)$, where $u_0$ is maximal with respect to $\leq$. It is easy to show that $u_0 \in E^0$~\cite[Lemma 15(3)]{MM}, that $u_1 = x_1x_1^{-1}$ for some $x_1 \in \pth(E)$ with $\so(x_1) = u_0$~\cite[Lemma 15(1,2)]{MM}, that $u_2 = x_1x_2x_2^{-1}x_1^{-1}$ for some $x_2 \in \pth(E)$~\cite[Lemma 15(2)]{MM}, and so on. Writing $u_i = x_1\cdots x_ix_i^{-1}\cdots x_1^{-1}$ for each $i \geq 1$, we conclude that $x_1x_2 \cdots$ is an an infinite path in $E$.

Now suppose that $n \leq 0$. Then, by (1), there exists an initial subpath $y$ of $x_1x_2 \cdots$ and a path $z \in \pth(E)$ such that $\ra(z) = \ra(y)$ and $|z| - |y| = |n|$. We can write $y=x_1 \cdots x_kt$ for some $k \geq 0$ and initial subpath $t$ of $x_{k+1}$. Then $yz^{-1} \in \SS(E)_{n}$, and so \[v:= yy^{-1} = yz^{-1}zy^{-1} \in E(\SS(E))_{n}.\] Moreover, $u_0 \geq v$, since $\so(y) = u_0$, and clearly \[v = x_1 \cdots x_ktt^{-1}x_k^{-1} \cdots x_1^{-1} \geq u_{k+1},\] as desired.

Next suppose that $n >0$. Since $x_1x_2 \cdots$ is an an infinite path, we can find an initial subpath $y \in \pth(E)$ such that $|y| = n$. Then, certainly, $v:=yy^{-1} \in E(\SS(E))_{n}$, $u_0 \geq v$, and $v \geq u_k$ for some $k \geq 0$, again giving the desired conclusion.

(2) $\Rightarrow$ (1) Suppose that (2) holds, let $n$ be a natural number, and let $p = e_1e_2 \cdots $ be an infinite path in $E$, for some $e_1, e_2, \dots \in E^1$. Then \[\so(e_1) > e_1e_1^{-1} > e_1e_2e_2^{-1}e_1^{-1} > \cdots\] is a chain of idempotents in $\SS(E)$, where $\so(e_1)$ is maximal with respect to $\leq$, by~\cite[Lemma 15(3)]{MM}. Hence, by (2), there exist $m \geq 1$ and $v \in E(\SS(E))_{-n}$ such that \[\so(e_1) \geq v \geq e_1\cdots e_me_m^{-1}\cdots e_1^{-1}.\] Then $v = yz^{-1}zy^{-1}$ for some $y,z \in \pth(E)$ such that $\ra(z) = \ra(y)$ and $|z| - |y| = n$. Finally, since \[\so(e_1) \geq yy^{-1} \geq e_1\cdots e_me_m^{-1}\cdots e_1^{-1},\] it must be the case that $y$ is an initial subpath of $e_1\cdots e_m$, and hence of $p$, by~\cite[Lemma 15(2)]{MM}. Thus $E$ satisfies condition (Y).
\end{proof}

\begin{example} \label{locgrex}
Consider the following graphs.
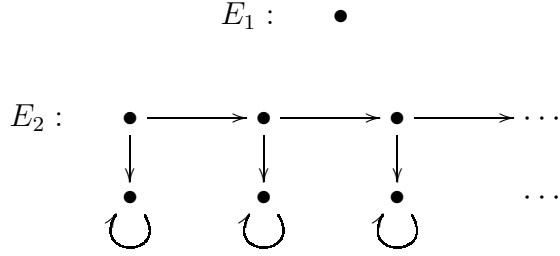
\begin{equation*}
\xymatrix@=15pt{
E_1: & {\bullet}}
\end{equation*}

\begin{equation*}
\xymatrix@=15pt{
E_2: & {\bullet} \ar[rr] \ar[d] && {\bullet} \ar[rr] \ar[d]&& {\bullet} \ar[rr]\ar[d] && \cdots\\
& {\bullet}\ar@(dr,dl) && {\bullet}\ar@(dr,dl) && {\bullet} \ar@(dr,dl) && \cdots}
\end{equation*}

\vspace{.3in}

\noindent It is easy to see that $E_1$ satisfies condition (Y) but not condition (2) in Proposition~\ref{locstrgrgraph}, whereas $E_2$ satisfies the latter but not the former. Thus, by Proposition~\ref{locstrgrgraph} and Lemma~\ref{locstrgr1}, $\SS(E_1)$ is saturated strongly graded, but not locally strongly graded, whereas $\SS(E_2)$ is locally strongly graded, but not saturated strongly graded, in the natural $\Z$-grading. The two conditions on gradings are therefore independent.
\end{example}

We are now ready for the main result of this section, which classifies the graph inverse semigroups $\SS(E)$ for which the corresponding Leavitt path algebras are strongly graded.

\begin{thm}\label{hfghfyhff}
Let $E$ be a nonempty graph. Then the following are equivalent.
\begin{enumerate}[\upshape(1)]

\item $E$ has no sinks, is row-finite, and satisfies condition (Y);

\smallskip

\item $L_K(E)$ is strongly graded in the natural $\Z$-grading, for any field $K$; 

\smallskip

\item There are no minimal idempotents in $\SS(E)\setminus \{0\}$, for every maximal idempotent $u$ in $\SS(E)$ there are only finitely many maximal idempotents in $\{v \in E(\SS(E)) \mid v< u\}$, and $\SS(E)$ is saturated strongly graded in the natural $\Z$-grading;

\smallskip

\item For every maximal idempotent $u$ in $\SS(E)$ there are only finitely many maximal idempotents in $\{v \in E(S) \mid v< u\}$, and $\SS(E)$ is locally strongly graded and saturated strongly graded in the natural $\Z$-grading.
\end{enumerate}
\end{thm}

\begin{proof}
(1) $\Leftrightarrow$ (2) This follows from \cite[Theorem 4.2]{chr19}. 

(1) $\Leftrightarrow$ (3) This follows from Lemmas~\ref{locstrgr2}, \ref{locstrgr3}, and~\ref{locstrgr1}.

(4) $\Rightarrow$ (1) If (4) holds, then $E$ must satisfy condition (2) in Proposition~\ref{locstrgrgraph}, which can easily be seen to imply that $E$ cannot have sinks. The desired conclusion now follows from Lemmas~\ref{locstrgr3} and~\ref{locstrgr1}.

(1) $\Rightarrow$ (4) By Lemmas~\ref{locstrgr3} and~\ref{locstrgr1}, it suffices to show that if $E$ has no sinks and satisfies condition (Y), then it also satisfies condition (2) in Proposition~\ref{locstrgrgraph}. 

Thus assume that $E$ has no sinks and satisfies condition (Y), and let $v \in E^{0}$ and $n \in \Z$. Since $E$ has no sinks, there is an infinite path $p$ in $E$ having source $v$. If $n \leq 0$, then condition (Y) implies that there exists an initial subpath $x$ of $p$ and a path $y \in \pth(E)$ such that $\ra(y) = \ra(x)$ and $|x| - |y|= n$. If $n > 0$, then letting $x \in \pth(E)$ be an initial subpath of $p$ such that $|x| = n$, and letting $y=\ra(x)$, we have $\so(x) = v$ and $|x|-|y| = n$. In either case, (2) in Proposition~\ref{locstrgrgraph} is satisfied.
\end{proof}

Restricting to row-finite graphs, we obtain a much cleaner statement, involving only conditions on the grading of $\SS(E)$.

\begin{cor}
Let $E$ be a nonempty row-finite graph. Then the following are equivalent.
\begin{enumerate}[\upshape(1)]

\item $E$ has no sinks and satisfies condition (Y); 

\smallskip

\item $L_K(E)$ is strongly graded in the natural $\Z$-grading, for any field $K$;

\smallskip

\item $\SS(E)$ is locally strongly graded and saturated strongly graded in the natural $\Z$-grading.
\end{enumerate}
\end{cor}

\begin{proof}
This follows from Theorem~\ref{hfghfyhff} and Lemma~\ref{locstrgr3}.
\end{proof}

The next corollary gives an analogue of Theorem~\ref{hfghfyhff} for Cohn path algebras. 

\begin{cor} \label{Cohncor}
Let $K$ be a field and $E$ a nonempty graph. Then neither $\SS(E)$ nor $C_K(E)$ is strongly graded in the natural $\mathbb \Z$-grading. However, if $E$ has no source vertices, then $\SS(E)$ and $C_K(E)$ are strongly $\Z/n\Z$-graded, for any positive integer $n$.
\end{cor}

\begin{proof}
This follows from Corollary~\ref{usual-grading}, Lemma~\ref{zn-lemma}, and Proposition~\ref{hfhffhghggh}(1).
\end{proof}

Let $\mathscr{G}$ be a $\Gamma$-graded Hausdorff ample groupoid (see \S \ref{topgroupoid}), and let $K$ be a field. Then the \emph{enveloping algebra of $\mathscr{G}^h$}, defined by
\begin{equation*}
K\langle \mathscr{G}^h \rangle := K[\mathscr{G}^h] \big / \big\langle  X + Y - X \cup Y \mid  X\cap Y = \emptyset, \text{ and }  X\cup Y \in \mathscr{G}^h_{\alpha} \text{ for some }\alpha \in \Gamma \big \rangle,
\end{equation*} 
is a $\Gamma$-graded $K$-algebra, via the grading inherited from $K[\mathscr{G}^h]$ (see (\ref{smgpringgr})). One can show that for any graph $E$, there is a naturally $\Z$-graded \emph{boundary path groupoid} $\mathscr{G}_E$ such that 
\begin{equation*}
K\langle \mathscr{G}^h_E \rangle \cong_{\gr} L_K(E) \cong_{\gr} A_K(\mathscr{G}_E), 
\end{equation*}
where  $A_K(\mathscr{G}_E)$ is the \emph{Steinberg algebra} of $\mathscr{G}_E$, and $\cong_{\gr}$ denotes graded isomorphism. We will not discuss $K\langle \mathscr{G}^h \rangle$, $\mathscr{G}_E$, or $A_K(\mathscr{G}_E)$ in further detail here, and instead refer the reader to \cite{rigby} for a comprehensive treatment of these objects. (See also~\cite[\S 6.3]{wehrung} for enveloping algebras of Boolean inverse semigroups, which we briefly visit in \S\ref{gradedbooleanring}.) We note, however, that Theorem~\ref{hfghfyhff} has the following consequence.

\begin{cor}\label{hfghfyhff88}
Let $E$ be a nonempty graph. Then the following are equivalent.

\begin{enumerate}[\upshape(1)]
\item $E$ has no sinks, is row-finite, and satisfies condition (Y);

\smallskip

\item $L_K(E)$ is strongly graded in the natural $\Z$-grading, for any field $K$;

\smallskip

\item $\mathscr{G}_E$ is strongly graded in the natural $\Z$-grading; 

\smallskip

\item For every maximal idempotent $u$ in $\SS(E)$ there are only finitely many maximal idempotents in $\{v \in E(\SS(E)) \mid v< u\}$, and $\SS(E)$ is locally strongly graded and saturated strongly graded in the natural $\Z$-grading.
\end{enumerate}
\end{cor}

\begin{proof}
(1) $\Leftrightarrow$ (2) $\Leftrightarrow$ (4) This follows from Theorem~\ref{hfghfyhff}.

(2) $\Leftrightarrow$ (3) This follows from \cite[Theorem 3.11]{chr19}. 
\end{proof}

If the graph happens to be finite, then the previous corollary has the following very pleasant form.

\begin{cor}
Let $E$ be a finite nonempty graph. Then the following are equivalent.

\begin{enumerate}[\upshape(1)]
\item $E$ has no sinks;

\smallskip

\item $L_K(E)$ is strongly graded in the natural $\Z$-grading, for any field $K$;

\smallskip

\item $\mathscr{G}_E$ is strongly graded in the natural $\Z$-grading; 

\smallskip

\item $\SS(E)$ is locally strongly graded in the natural $\Z$-grading.
\end{enumerate}
\end{cor}

\begin{proof}
This follows from Corollary~\ref{hfghfyhff88}, upon noting that $E$ being finite implies that $E$ satisfies condition (Y) and $\SS(E)$ is saturated strongly graded in the natural $\Z$-grading.
\end{proof}

\subsection{Covering graphs}\label{covergraph}
In this subsection we show that the smash product of a graph inverse semigroup with any group is graded isomorphic to the inverse semigroup of its covering graph, which we recall next (see~\cite[\S2]{gr} and \cite[Definition 2.1]{kp}).

Let $E$ be a graph, $\G$ a group, and $\omega :E^{1}\rightarrow \G$ a ``weight" function. The \emph{covering graph $\overline{E}$ of $E$ with respect to $\omega$} is defined by
\begin{gather*}
    \overline{E}^{0} = \{v_{\a} \mid  v\in E^{0}\text{ and } \a\in\G\} \ \text{ and } \
    \overline{E}^{1} = \{e_\a \mid e \in E^1\text{ and } \a \in \G\},
\end{gather*}
with $\so(e_\a) = \so(e)_\a$ and $\ra(e_\a) = \ra(e)_{\omega (e)^{-1}\a}$ for all $e \in E^{1}$ and $\alpha \in \Gamma$. The covering graph $\overline{E}$ inherits the weight function from $E$, as follows: 
\begin{align}\label{leoleoleoel}
\overline{E}^{1} & \longrightarrow \G\\\notag 
e_\alpha &\longmapsto \omega (e).\notag
\end{align}

\begin{example}\label{onetwo3}
Let $E$ be a graph and define $\omega:E^1 \rightarrow \Z$ by $\omega (e) = 1$ for all $e \in E^1$. Then $\overline{E}$ (sometimes denoted $E\times_1 \Z$) is given by
\begin{gather*}
    \overline{E}^0 = \big\{v_n \mid v \in E^0 \text{ and } n \in \Z \big\} \ \text{ and } \
    \overline{E}^1 = \big\{e_n \mid e\in E^1 \text{ and } n\in \Z \big\},
\end{gather*}
where $\so(e_n) = \so(e)_n$ and $\ra(e_n) = \ra(e)_{n-1}$ for all $e \in E^{1}$ and $n \in \Z$. 

To construct more concrete examples, consider the following graphs.
\begin{equation*}
{\def\labelstyle{\displaystyle}
E_1 : \quad \,\, \xymatrix{
 u \ar@(lu,ld)_e\ar@/^0.9pc/[r]^f & v \ar@/^0.9pc/[l]^g
 }} \qquad \quad
{\def\labelstyle{\displaystyle}
E_2: \quad \,\, \xymatrix{
   u \ar@(dl,lu)^e  \ar@(ur,rd)^f}}
\end{equation*}
Then the corresponding covering graphs (with $\omega$ as before) are as follows.

\begin{equation*}
\xymatrix@=15pt{
\overline{E}_1: & \cdots \ {u_{-1}}  && \ar[ll]_{e_{0}} {u_{0}} \ar[dll]_(0.4){f_{0}} && \ar[ll]_{e_1} {u_{1}} \ar[dll]_(0.4){f_1}  && \ar[ll]_{e_{2}}  u_2 \ar[dll]_(0.4){f_{2}} \ \cdots\\
& \cdots \ {v_{-1}} && {v_{0}} \ar[ull]^(0.4){g_{0}}  && {v_{1}} \ar[ull]^(0.4){g_1} && v_2 \ar[ull]^(0.4){g_{2}} \ \cdots \\
& \text{\bf Level -1} && \text{\bf Level 0} && \text{\bf Level 1} && \text{\bf Level 2}\\ 
\overline{E}_2: & \cdots \ {u_{-1}} & & \ar@/^0.9pc/[ll]^{f_{0}} \ar@/_0.9pc/[ll]_{e_{0}}  {u_{0}} & & \ar@/^0.9pc/[ll]^{f_1} \ar@/_0.9pc/[ll]_{e_1} {u_{1}} & & \ar@/^0.9pc/[ll]^{f_{2}} \ar@/_0.9pc/[ll]_{e_{2}} u_2 \ \cdots}
\end{equation*}
\end{example}

Notice that for any graph $E$, the covering graph $\overline{E}$ is \emph{acyclic} (i.e., has no cycles) and \emph{stationary} (i.e., informally, the pattern of vertices and edges on ``level'' $n$ repeats on ``level'' $n+1$). 

\begin{thm}\label{gdhfthfhfhdsjje}
Let $E$ be a graph with a weight function $\omega :E^{1}\rightarrow \G$, and let $\overline{E}$ be its covering graph with respect to $\omega$. Then assigning
\begin{align}\label{gigfbds} 
v_\alpha &\longmapsto vP_\alpha\\ \notag
e_\alpha &\longmapsto eP_{\omega(e)^{-1}\alpha}\\ \notag
e^{-1}_\alpha &\longmapsto e^{-1}P_\alpha \notag
\end{align}
induces a graded isomorphism $\phi: \SS(\overline{E}) \rightarrow \SS(E) \# \Gamma$.
\end{thm}

\begin{proof}
To show that the assignments in (\ref{gigfbds}) induce a homomorphism, it suffices to prove that the function $\phi$ induced by those assignments preserves the defining relations of $\SS(\overline{E})$--see Definition~\ref{shdnfhyhgkdk}.

To check that $\phi$ preserves the relations (V), let $v,w \in E^0$ and $\alpha,\beta \in \Gamma$. Then, noting that as idempotents, $v, w \in \SS(E)_\varepsilon$, and using Definition \ref{hghghgqq}, we have
\begin{align*}
\phi(v_{\alpha}w_{\beta}) & = \phi(\delta_{v_{\alpha},w_{\beta}}v_{\alpha}) \\
& = \left\{ \begin{array}{ll}
vP_\alpha & \text{if } \alpha = \beta \text{ and } v=w \\
0 & \text{otherwise } 
\end{array}\right. \\
& =  vP_\alpha wP_\beta = \phi(v_{\alpha})\phi(w_{\beta}).
\end{align*}
For (E1), let $e\in E^1$ and $\alpha \in \Gamma$. Then, since $\so(e_\alpha) = \so(e)_\alpha$, we have
\[\phi(\so(e_\alpha)e_\alpha) = \phi(e_\alpha) = eP_{\omega (e)^{-1}\alpha} = \so(e)P_\alpha eP_{\omega (e)^{-1}\alpha} = \phi(\so(e_\alpha))\phi(e_\alpha),\]
and, using the fact that $\ra(e_\alpha) = \ra(e)_{\omega (e)^{-1}\alpha}$, we have
\[\phi(e_\alpha\ra(e_\alpha)) = \phi(e_\alpha) = eP_{\omega (e)^{-1}\alpha} = eP_{\omega (e)^{-1}\alpha} \ra(e)P_{\omega (e)^{-1}\alpha} = \phi(e_\alpha)\phi(\ra(e_\alpha)).\]
That $\phi$ preserves the relations (E2) can be verified analogously. Finally, for (CK1), let $e,f \in E^1$ and $\alpha,\beta \in \Gamma$. Then
\begin{align*}
\phi(e^{-1}_\alpha f_\beta) & = \phi(\delta_{e_\alpha, f_\beta}\ra(e_\alpha))  = \phi(\delta_{e_\alpha, f_\beta}\ra(e)_{w(e)^{-1}\alpha}) \\ \notag
& = \left\{ \begin{array}{ll}
\ra(e)P_{\omega (e)^{-1}\alpha} & \text{if } \alpha = \beta \text{ and } e=f \\
0 & \text{otherwise } 
\end{array}\right.\\
& =e^{-1} P_\alpha  f P_{\omega (f)^{-1} \beta} = \phi (e^{-1}_\alpha)\phi(f_\beta). \notag
\end{align*}
Thus $\phi$ is a semigroup homomorphism.

Next, let $v \in V$, $e \in E^1$, and $\alpha \in \Gamma$. Then, using (\ref{graphinvgr}) and (\ref{leoleoleoel}), we have $v_\alpha \in \SS(\overline{E})_\varepsilon$, $e_\alpha \in  \SS(\overline{E})_{\omega (e)}$, and $e_\alpha^{-1} \in  \SS(\overline{E})_{\omega (e)^{-1}}$. On the other hand, by (\ref{ghfgdthfyr1}), $\phi(v_\alpha) = vP_\alpha \in (\SS(E) \# \Gamma)_\varepsilon$, $\phi(e_\alpha) = eP_{\omega (e)^{-1}\alpha}\in (\SS(E) \# \Gamma)_{\omega (e)}$, and $\phi(e_\alpha^{-1}) = e^{-1}P_{\alpha}\in (\SS(E) \# \Gamma)_{\omega (e)^{-1}}$. Thus $\phi$ preserves the degrees of the generators of $\SS(\overline{E})$. Since it is a homomorphism, it follows that $\phi$ is a graded map.

Since every element of $\SS(E) \# \Gamma$ is a product of elements of the forms $vP_\alpha$, $eP_\alpha$, and $e^{-1}P_\alpha$, for some $v \in V$, $e \in E^1$, and $\alpha \in \Gamma$, and since $\phi$ is a homomorphism, it follows immediately from (\ref{gigfbds}) that it is surjective. So it remains to show that $\phi$ is injective.
 
Let \[(e_1)_{\alpha_1} \cdots (e_n)_{\alpha_n} (f_m)_{\beta_m}^{-1} \cdots (f_1)_{\beta_1}^{-1} \in \SS(\overline{E}) \setminus \{0\},\] where the $(e_i)_{\alpha_i}, (f_i)_{\beta_i} \in \overline{E}^{0} \cup \overline{E}^{1}$, and suppose that $\phi$ maps this element to zero. Then
\[0 = e_1P_{\omega (e_1)^{-1}\alpha_1} \cdots e_nP_{\omega (e_n)^{-1}\alpha_n}f_m^{-1}P_{\beta_m} \cdots f_1^{-1}P_{\beta_1},\]
which implies that at least one of the following must be the case: $\ra(e_i) \neq \so(e_{i+1})$ for some $i$, $\ra(f_i) \neq \so(f_{i+1})$ for some $i$, $\ra(e_n) \neq \ra(f_m)$, $\omega (e_i) \neq \alpha_i\alpha_{i+1}^{-1}$ for some $i$, $\omega (f_i) \neq \beta_i\beta_{i+1}^{-1}$ for some $i$, $\omega (e_n)^{-1}\alpha_n \neq \omega (f_m)^{-1}\beta_m$. But each of these conditions implies that $(e_1)_{\alpha_1} \dots (e_n)_{\alpha_n} (f_m)_{\beta_m}^{-1} \dots (f_1)_{\beta_1}^{-1} = 0$, producing a contradiction. So $\phi$ maps nonzero elements to nonzero elements.

Next, let $s, t \in \SS(\overline{E}) \setminus \{0\}$, and suppose that $\phi(s) = \phi(t)$. Writing \[s = (e_1)_{\alpha_1} \cdots (e_n)_{\alpha_n} (f_m)_{\beta_m}^{-1} \cdots (f_1)_{\beta_1}^{-1} \text{ and } t = (g_1)_{\delta_1} \cdots (g_p)_{\delta_p} (h_r)_{\gamma_r}^{-1} \cdots (h_1)_{\gamma_1}^{-1},\] and using the fact that $\phi(s) \neq 0 \neq \phi(t)$, we have 
\[e_1 \cdots e_n f_m^{-1} \cdots f_1^{-1} P_{\beta_1} = \phi(s) = \phi(t) = g_1 \cdots g_p h_r^{-1} \cdots h_1^{-1} P_{\gamma_1},\]
along with appropriate compatibility conditions on the weights (as in the previous paragraph). It follows that $s=t$, completing the proof.
\end{proof}

It is proved in~\cite{arasims} that $L_K(E)\# \mathbb Z \cong L_K(\overline{E})$, using skew products for groupoids and Steinberg algebras. The following is an analogous result for Cohn algebras, which we can prove directly, employing the smash product for semigroups.

\begin{cor}
For any graph $E$ and field $K$, we have $C_K(E)\# \mathbb Z \cong C_K(\overline{E})$. 
\end{cor}

\begin{proof}
Since $C(E)=K[\SS(E)]$, by Proposition~\ref{hfhffhghggh}(2) and Theorem~\ref{gdhfthfhfhdsjje}, we have 
\[C_K(E)\# \mathbb Z = K[\SS(E)] \# \mathbb Z\cong K[\SS(E) \# \mathbb Z]\cong K[\SS(\overline{E})]= C_K(\overline{E}).\qedhere\]
\end{proof}

\section{Further directions}

We conclude with some ideas for further research on graded semigroups, that we have not explored in detail.

\subsection{Graded Green's relations} 

Green's relations are a fundamental tool for studying semigroups, and so it is natural to consider graded versions thereof in the context of graded semigroups. So given a $\Gamma$-graded semigroup $S$ and $s,t \in S$, write $s\, \mathscr{L}^{\gr} \, t$ if $s\, \mathscr{L} \, t$ and $s,t\in S_\alpha$, for some $\alpha \in \Gamma$. The other \emph{graded Green's relations} $\mathscr{R}^{\gr}$, $\mathscr{H}^{\gr}$, $\mathscr{D}^{\gr}$, and $\mathscr{J}^{\gr}$ can be defined similarly. These relations partition $S$ into finer equivalence classes than the usual Green's relations, and so have the potential to shed additional light on the structure of $S$. For example, letting $L_s$ denote the $\mathscr{L}$-class of $s\in S$, we have $L_s = \bigcup_{\alpha \in \Gamma} (L_s)_{\alpha}$, where $(L_s)_{\alpha} = S_{\alpha} \cap L_s$. Recall that Green's lemma~\cite[Lemma~2.2.1]{howie} provides a bijection $\rho : L_s \rightarrow L_t$, whenever $s\, \mathscr{R} \, t$. It is easy to obtain a graded version of this result. Specifically, for all $s,t \in S$ and $\alpha \in \Gamma$, if $s\, \mathscr{R} \, t$, then there is a bijection $\rho : L_s \rightarrow L_t$ such that $\rho((L_s)_\alpha) = (L_s)_{\alpha\gamma}$, where $\gamma=\deg(t)\deg(s)^{-1}$. It may be interesting to investigate graded Green's relations more closely in the future.

\subsection{Graded Boolean inverse semigroups}\label{gradedbooleanring}

Let us mention another class of inverse semigroups that seem well-suited to the graded setting. An inverse semigroup $S$ is called \emph{Boolean} if $E(S)$ is a generalised Boolean lattice and every orthogonal pair $u, v \in E(S)$ (i.e., $u^{-1}v = 0 = v u^{-1}$, denoted $u \perp v$) has a supremum, denoted $u \oplus v$. (See \cite[\S 3.1]{wehrung} for more details.) Given a Boolean inverse semigroup $S$, the \emph{type semigroup} of $S$ is the commutative monoid $\Typ(S)$ generated by $\{\typ(u) \mid u \in E(S)\}$, subject to the following relations, for all $u,v \in E(S)$:
\begin{enumerate}
\item $\typ(0) = 0$;
    
\item $\typ(u) = \typ(v)$ whenever $u \, \mathscr{D} \, v$;
    
\item $\typ(u \oplus v) = \typ(u) + \typ(v)$ whenever $u \perp v$.
\end{enumerate}
(See~\cite[\S 4.1]{wehrung} for more details.) Now, if $S$ is a $\Gamma$-graded Boolean inverse semigroup, then it is easy to see that $S_{\varepsilon}$ is also a Boolean inverse semigroup. So it is natural to seek descriptions of the relations among these semigroups, and those among $\Typ(S)$ and $\Typ(S_{\varepsilon})$. Additionally, given a $\Gamma$-graded Boolean inverse semigroup $S$ and a field (or, more generally, unital ring) $K$, one can define the \emph{enveloping algebra} 
\begin{equation*}
K\langle S \rangle := K[S] \big / \big\langle u+v-u\oplus v \mid u \perp v \big \rangle 
\end{equation*} 
of $S$ (see~\cite[\S 6.3]{wehrung}), and investigate the relationships among the $K$-algebras $K\langle S \rangle$ and $K\langle S_{\varepsilon} \rangle$.

Type semigroups are of particular interest to us because of their connection to combinatorial algebras. More specifically, letting $S=\mathscr{G}_E^h$ be the inverse semigroup associated to the boundary path groupoid $\mathscr{G}_E$ (see~\S \ref{groupoidsection} and~\S \ref{hfgfhyhhvgfgffd}), $\Typ(S)$ is related to the non-stable $K$-theory of the corresponding graph $C^*$-algebra and Leavitt path algebra. It is believed that $\Typ(S)$ could be used to find a complete invariant for the algebras in question (see~\cite{luiz}). 

\section*{Acknowledgement}

We are grateful to the referee for a very thoughtful review, and suggestions that have led to improvements in the exposition.

\end{document}